\documentclass[reqno,11pt]{amsart}
\usepackage{amsthm,amsfonts,amssymb,euscript,mathrsfs,graphics,color,amsmath,latexsym,marginnote}

\usepackage{hyperref}

\usepackage{graphicx}

\numberwithin{equation}{section}

\def\a{{\alpha}}
\def\C{{\mathbb{C}}}

\def\bar{\overline}
\def\R{{\mathbb R}}

\def\R{{\bf R}}

\def\bar{\overline}

\def\R{\mathbb{R}}

\newtheorem{theorem}{Theorem}[section]
\newtheorem{lemma}[theorem]{Lemma}
\newtheorem{proposition}[theorem]{Proposition}

\newtheorem{remark}[theorem]{Remark}

\begin{document}
	
	\title{Nonlinear Landau damping for the unconfined Vlasov-Poisson system in 2D}
	
	\author{Alexandru D. Ionescu}
	\address{Princeton University}
	\email{aionescu@math.princeton.edu}
	
	\author{Quoc-Hung Nguyen}
	\address{State Key Laboratory of Mathematical Sciences, Academy of Mathematics and Systems Science, Chinese Academy of Sciences, Beijing 100190, China; Institute of Mathematics, Academy of Mathematics and Systems Science, the Chinese Academy of Sciences, Beijing 100190, China}
	\email{qhnguyen@amss.ac.cn}
	
	\maketitle

	\begin{abstract}
		We prove nonlinear Landau damping for the two-dimensional Vlasov--Poisson system in the unconfined Euclidean setting near the homogeneous Poisson equilibrium
	\[
	\mu(v)=\frac{1}{2\pi}(1+|v|^2)^{-3/2}.
	\]
	For small, smooth, localized, and neutral perturbations we establish global existence, quantitative pointwise decay estimates for the electric field, and scattering to free transport. The result we prove here appears to be the first nonlinear Landau damping result for solution of the unconfined 2D Vlasov--Poisson system.
	
	The uniform Penrose stability bound degenerates at low spatial frequencies, giving rise to weakly damped plasma oscillations. We separate a faster-decaying regular field from oscillatory components with temporal factors $\cos(t)$ and $\sin(t)$. The proof combines a Volterra representation of the density, weighted estimates for the characteristic flow and nonlinear moments, and integration by parts in time. Neutrality supplies the spatial cancellation needed in the critical two-dimensional estimates.
	\end{abstract}

	\setcounter{tocdepth}{1}

	\tableofcontents
	
\section{Introduction}

\subsection{The stability problem} We investigate the asymptotic stability of the Poisson equilibrium for the
Vlasov--Poisson system in two spatial dimensions. Writing the total
distribution as $\mu+f$, the perturbation satisfies
\begin{equation}\label{eq2}
	\begin{split}
		\partial_t f+v\cdot\nabla_x f+E\cdot\nabla_v\mu
		&=-E\cdot\nabla_v f,\\
		E&=\nabla_x\Delta_x^{-1}\rho,\qquad
		\rho(t,x)=\int_{\R^2}f(t,x,v)\,dv,
	\end{split}
\end{equation}
where $(x,v)\in\R^2\times\R^2$ and
\begin{equation}\label{PoEq}
	\mu(v)=\frac{1}{2\pi}(1+|v|^2)^{-3/2}.
\end{equation}
We prove global existence, decay of the electric field, and scattering
to free transport for small, localized perturbations satisfying a
neutrality condition. Only three derivatives of the initial data are
required.

The decay of the electric field near a spatially homogeneous equilibrium
is known as Landau damping, following Landau's prediction in
\cite{landau}. In the unconfined problem, dispersion and phase mixing
interact with the long-range Poisson force. A fundamental difference compared to the periodic problem is that the Penrose linear stability
criterion becomes degenerate: its uniform lower bound fails near
the plasma frequencies as the spatial frequency tends to zero.
The resulting weakly damped oscillations have dimension-dependent
decay. In two dimensions, controlling their nonlinear effect on the
characteristics requires both oscillatory cancellation and decay
in velocity.

For the Poisson equilibrium \eqref{PoEq}, the linear response can be computed
explicitly since $\widehat\mu(\xi)=e^{-|\xi|}$. The density satisfies
\begin{equation}\label{rhoVolterra}
	\widehat\rho(t,\xi)+\int_0^t(t-s)e^{-(t-s)|\xi|}
	\widehat\rho(s,\xi)\,ds=\widehat{\mathbf R_0}(t,\xi),
\end{equation}
where $\mathbf R_0$ is the source defined in subsection~\ref{intro-field}. In the
linearized problem, it is simply the free transport density
$\mathbf R_0(t,x)=\int_{\R^2}f_0(x-tv,v)\,dv$. The resolvent of
\eqref{rhoVolterra} has Laplace poles at $-|\xi|\pm i$. These approach
the imaginary axis as $|\xi|\to0$, making the degeneration of the
Penrose condition explicit.

More precisely, the Volterra equation gives
\begin{equation}\label{linrep}
	\rho(t)=\mathbf R_0(t)
	-\int_0^t\sin(s)e^{-s|\nabla|}\mathbf R_0(t-s)\,ds,
\end{equation}
where $|\nabla|$ denotes the Fourier multiplier with symbol $|\xi|$.
The explicit time oscillation in \eqref{linrep} persists in the
nonlinear problem and is central to our analysis. Following the
decompositions in \cite{AIonescu2022,NWZ2024}, we separate a
faster-decaying regular field from larger oscillatory components
with phases $\cos(t)$ and $\sin(t)$. The regular part is controlled
by its decay, while the oscillatory amplitudes have time derivatives
that decay faster than the amplitudes themselves. This allows
integration by parts in time to compensate for their weaker decay.
The decomposition thus separates two complementary mechanisms:
direct time integrability and cancellation from time oscillations.
Exploiting both is particularly important in two dimensions, where
the available dispersion is weaker.

\subsection{Previous work}
The Vlasov--Poisson system has an extensive theory of global existence,
regularity, and long-time behavior. We refer, among many other works, to
\cite{arsenev,batt,EHorst1982,CBardos1985,FBouchut1991,CBardos2018,
	KPfaffelmoser1992,JSchaeffer1991,EHorst1993,RTGlassey1996,
	LionsPL1991,HJHwang2011,SHCHoi2011,JSmulevici2016,Xwang2018}.
We discuss here the results most closely related to nonlinear damping
near a homogeneous equilibrium.

In the spatially periodic (confined) setting, Mouhot--Villani
\cite{CMouhot2011} established nonlinear Landau damping for small
analytic and Gevrey perturbations of Penrose-stable equilibria. Bedrossian--Masmoudi--Mouhot \cite{JBedrossian2016} developed a different
approach in Gevrey classes, clarifying the role of nonlinear plasma
echoes, and Grenier--Nguyen--Rodnianski \cite{NguyenTT2020} gave a
simplified proof for analytic and Gevrey data. The plasma echoes can prevent a corresponding damping theory for general
Sobolev perturbations, as shown in \cite{Bedinsta}, although special echo solutions outside the
analytic and Gevrey classes can still damp \cite{GNRechoes}.
At the sharp weighted Gevrey-$3$ threshold, Ionescu--Pausader--Wang--Widmayer \cite{IPWWGevrey} proved
asymptotic stability and solved the final-state problem, constructing
nonlinear wave and scattering operators.
For related Vlasov-type equations on the torus, Gagnebin--Iacobelli
\cite{GagnebinIacobelli2023} proved nonlinear damping for the
massless-electron model, while
Benedetto--Caglioti--Gagnebin--Iacobelli--Rossi \cite{BCGIR2026}
solved the final-state problem for Gevrey data of order less than $3$.

In the unconfined setting, spatial dispersion permits damping results
with finitely many derivatives, but the low-frequency behavior of the
Poisson interaction requires a different analysis. The dimension of the ambient space is connected to the rate of dispersion, and becomes important.
The linearized dynamics and its dispersive oscillatory component
were studied in \cite{HNRcmp,JBedrossian2020}. T. T. Nguyen
\cite{NguyenSurvival} analyzed the survival threshold separating
undamped oscillatory modes from the Landau-damped regime for radial
equilibria. For the unscreened problem near a nonzero homogeneous equilibrium,
the first nonlinear asymptotic stability and scattering result was obtained by
Ionescu--Pausader--Wang--Widmayer \cite{AIonescu2022} for the 3D Poisson equilibrium, using a
Lagrangian representation of the density and a decomposition of the
field into regular and oscillatory components.

Nguyen--Wei--Zhang \cite{NWZ2024} gave a shorter proof of the
three-dimensional result under a neutrality assumption. Their
decomposition uses the equations for the first macroscopic moments
to isolate the oscillatory terms. We use this formulation here:
the zeroth, first, and second moments determine the field components,
and the nonlinear analysis reduces to weighted estimates for these
moments and their derivatives.

For screened interactions, Bedrossian--Masmoudi--Mouhot
\cite{BMMscreened2018} proved nonlinear Landau damping near stable
homogeneous equilibria in $\R^3$ for small, localized Sobolev
perturbations. Han-Kwan--Nguyen--Rousset \cite{HanKwanD2021}
subsequently gave a Lagrangian proof with weaker regularity
assumptions and essentially sharp decay estimates. Further
refinements include sharp bounds \cite{HNX1} and higher derivative
estimates \cite{NguyenTT2020t}. The two-dimensional Vlasov--Poisson system with massless electrons was treated in \cite{HNX2}. The unscreened
Poisson force considered here has a different low-frequency structure,
including the plasma oscillations visible in \eqref{linrep}.
For the relativistic Vlasov--Klein--Gordon system, T. T. Nguyen
\cite{NguyenBelow} established nonlinear damping in the
plasma-oscillation regime below the survival threshold.

The main difficulty in passing from three to two dimensions is that
the weaker dispersion no longer suffices for several characteristic
and moment estimates used in the three-dimensional argument.
In particular, the spatial weight controlling the leading linear
field is not integrable. We compensate for this by combining time
oscillations with velocity decay along characteristics, and by
isolating the leading terms in their second velocity derivatives.
These estimates allow us to retain an integrable spatial weight for
the remaining field components and to close the nonlinear argument.

\subsection{Main result}
Our main theorem gives quantitative pointwise bounds for the electric
field and convergence of $f(t,x+tv,v)$ to an asymptotic profile.
We first fix some notation. For an integer $j\ge0$, $|\nabla_{x,v}^j f|$ denotes the sum of the
absolute values of all mixed derivatives of total order $j$, and
repeated spatial indices are summed. We define $\nabla\Delta^{-1}$
as the Fourier multiplier with symbol $-i\xi/|\xi|^2$ for $\xi\ne0$.
Fix a radial cutoff $\chi\in C_c^\infty([0,\infty))$ with
$0\le\chi\le1$, $\chi=1$ on $[0,1]$, and $\chi=0$ on
$[\frac32,\infty)$. For $z_j\in\mathbb R^{d_j}$, write
\[
\langle z_1,\ldots,z_m\rangle
:=\big(1+\sum_{j=1}^m|z_j|^2\big)^{1/2}.
\]

\begin{theorem}[Global existence, decay, and scattering]\label{mainthm}
	Let $\kappa_0=10^{-4}$ and assume that $\mu$ denotes the Poisson equilibrium \eqref{PoEq}. Assume that $f_0\in C^3(\mathbb{R}_x^2\times\mathbb{R}_v^2;\mathbb{R})$ satisfies the smallness condition
	\begin{equation}\label{a16'}
		[f_0]:=\sum_{j=0}^{3}\sup_{x,v}\langle x,v\rangle^{3+20\kappa_0}\langle v\rangle^{3+\kappa_0}\big|\nabla_{x,v}^j f_0(x,v)\big|
		\leq \epsilon_0,
	\end{equation}
	where $\epsilon_0>0$ is sufficiently small, and the neutrality condition
	\begin{equation}\label{a16}
		\int_{\mathbb{R}^2\times\mathbb{R}^2}f_0(x,v)\,dx\,dv=0.
	\end{equation}
	Then the initial-value problem \eqref{eq2} with
	$f(0,x,v)=f_0(x,v)$ admits a unique global classical solution
	$(f,E)$, with $E$ denoting the electric field, in the class
	\begin{gather*}
		f\in C([0,\infty);C^1_{x,v})
		\cap C^1([0,\infty);C^0_{x,v}),\quad
		\langle x\rangle^2(E,\nabla_xE)
		\in L^\infty_{\mathrm{loc}}
		\bigl([0,\infty);L^\infty_x(\mathbb R^2)\bigr).
	\end{gather*}
	
	In addition, the associated electric field admits the decomposition
	\begin{equation}
		E(t,x)
		=
		E^{reg}(t,x)
		+\cos(t)E^{\cos}(t,x)
		+\sin(t)E^{\sin}(t,x),
	\end{equation}
	and the following estimates hold:
	
	\begin{itemize}
		\item \textbf{Electric-field decay.}
		For every $n\in\{0,1,2\}$, $t>0$, and $x\in\mathbb{R}^2$,
		\begin{align*}
			&|\nabla_x^n E^{reg}(t,x)|
			+\sum_{*\in\{\cos,\sin\}}|\nabla_x^n\partial_t E^{*}(t,x)|
			\lesssim [f_0]\langle t,x\rangle^{-2-\kappa_0}\langle t\rangle^{-\frac12-\frac{2n}{3}},\\
			&\sum_{*\in\{\cos,\sin\}}|\nabla_x^{1+n}E^{*}(t,x)|
			\lesssim [f_0]\langle t,x\rangle^{-2-\kappa_0}
			\begin{cases}
				\langle t\rangle^{2\kappa_0-\frac{2(n+1)}{3}},
				&n\in\{0,1\},\\
				\langle t\rangle^{-\frac{11}{6}},&n=2,
			\end{cases}\\
			&\sum_{*\in\{\cos,\sin\}}|E^{*}(t,x)|
			\lesssim [f_0]\langle t,x\rangle^{-2}\langle t\rangle^{50\kappa_0}.
		\end{align*}
		
		\item \textbf{Scattering of the distribution function.}
		There exists a Lipschitz profile
		$f_\infty\in C^{0,1}_{x,v}$ such that, for every $t>0$ and
		$(x,v)\in\mathbb R^2\times\mathbb R^2$,
		\begin{equation}\label{po}
			\begin{split}
				&\langle v\rangle^{3}\left|	f(t,x,v)- 	f_\infty(x-tv,v)\right|
				\lesssim [f_0]\langle t\rangle^{-\frac12-\kappa_0}\Big(
				\langle x-tv,v\rangle^{-2-19\kappa_0}
				+\langle t\rangle^{\frac12+51\kappa_0}\langle t,x\rangle^{-2}\\
				&\quad+\int_0^t\langle s\rangle^{-\frac23+\kappa_0}
				\langle s,x-(t-s)v\rangle^{-2}\,ds+\langle t\rangle^{\frac12+\kappa_0}
				\int_t^\infty\langle s\rangle^{-\frac12}
				\langle s,x-(t-s)v\rangle^{-2}\,ds\Big)\\
				&\qquad\qquad\lesssim [f_0]\langle t\rangle^{-\frac12-\kappa_0}.
			\end{split}
		\end{equation}
	\end{itemize}
\end{theorem}

\begin{remark}
	The neutrality assumption \eqref{a16} is essential for the critical
	two-dimensional estimates in this paper. It was not needed in the
	three-dimensional stability theorem of \cite{AIonescu2022}, where
	the Coulomb field already decays like $|x|^{-2}$. In 2D, nonneutral initial data produce a $|x|^{-1}$ far field decay, which is insufficient for convergence, but neutrality cancels this leading term and gains one power of spatial
	decay for the linear
	oscillatory field.
	
	The neutrality condition is preserved by \eqref{eq2}:
	\begin{equation}\label{a17}
		\iint_{\R^2\times\R^2}f(t,x,v)\,dx\,dv=0,\qquad t\ge0.
	\end{equation}
	Moreover, if $\mu+f_0\ge0$, then $\mu+f(t)\ge0$ for all $t\ge0$,
	since the total distribution $\mu+f$ is constant along the
	characteristic flow.
\end{remark}

\subsection{Main ideas of the proof}
We now outline the main ideas used to prove Theorem~\ref{mainthm}.

\subsubsection{Field decomposition and bootstrap}\label{intro-field}
We introduce the modified velocity moments
\[
\mathbf R_k(t,x):=\int_{\R^2}v^{\otimes k}
\Big[f(t,x,v)+\int_0^t
E(s,x-(t-s)v)\cdot\nabla_v\mu(v)\,ds\Big]\,dv,
\]
for $k=0,1,2$. Here $\mathbf R_0$ is the source in \eqref{linrep}.
The moment identities \eqref{z5}--\eqref{z5'} express its first two
time derivatives through $\mathbf R_1$, $\mathbf R_2$, and quadratic
expressions in $E$. Integrating \eqref{linrep} by parts in time
therefore expresses the field in terms of these three moments and
quadratic field terms, leading to the regular and oscillatory
components described above.

We split the oscillatory amplitudes into the explicit terms determined
by the initial data and the remaining contributions:
\[
E^\ast=E^\ast_{lin}+E^\ast_{non},\qquad \ast\in\{\cos,\sin\}.
\]
The neutrality condition yields
\[
|E^{\cos}_{lin}(t,x)|+|E^{\sin}_{lin}(t,x)|
\lesssim[f_0]\langle t,x\rangle^{-2}.
\]
This spatial weight is not integrable in two dimensions. In contrast,
$E^{reg}$ and $E^\ast_{non}$ can be controlled with the integrable
weight $\Phi(t,x)=\langle t,x\rangle^{-2-\kappa_0}$, with stronger
time decay for $E^{reg}$ and $\partial_tE^\ast_{non}$. We therefore
estimate the explicit linear terms separately and close the
bootstrap for the remaining components. With the norms and $M_T$
defined in Section~2, the main estimate is
\[
\|E^{reg}\|_{1,T}
+\|(E^{\cos}_{non},E^{\sin}_{non})\|_{2,T}
\lesssim[f_0]+M_T^2.
\]
The different weights assigned to the amplitudes and their time and
spatial derivatives are chosen to exploit the oscillatory factors in
subsequent integrations by parts.

\subsubsection{Characteristic estimates}
Let $(X_{s,t},V_{s,t})$ denote the backward characteristic flow, and let
$Y_{s,t}$ and $W_{s,t}$ denote its position and velocity deviations
from free transport, as defined in \eqref{Lan6}. Their
integral equations contain the field evaluated along perturbed
straight lines. Time integration along such lines gains decay in
velocity; for example,
\begin{equation}\label{time-velocity-gain}
	\int_s^t\langle\tau\rangle^{a_2}
	\langle\tau,x-\tau v\rangle^{-a_1}\,d\tau
	\lesssim
	\langle s\rangle^{a_2-a_1+1}\langle v\rangle^{-1},
	\qquad a_1>1,\quad a_2-a_1+1<0.
\end{equation}
This gain is important when estimating the first two velocity
moments of the nonlinear terms. For the regular field, the time
decay is sufficient for direct integration. For the oscillatory
field, we also integrate by parts in time using the explicit
$\cos(t)$ and $\sin(t)$ factors. The differentiated amplitudes
have better decay, while derivatives of the trajectory are
controlled by the characteristic equations.

Differentiating twice in velocity introduces growing time factors,
so the higher-order estimates require a separate treatment.
We isolate the leading oscillatory contribution to
$\nabla_v^2Y_{s,t}$ and prove better bounds for the remainder.
This expansion, together with estimates for the shifted
characteristics and the inverse velocity map, is the main input
from Section~4 into the moment estimates.

\subsubsection{Nonlinear moments}
These modified moments split as $\mathbf R_k=\mathcal I_k+\mathcal R_k$,
where $\mathcal I_k$ is the initial-data contribution and
$\mathcal R_k$ is the nonlinear term.
The latter is built from the difference between the field transported
by the free flow and by the nonlinear characteristic flow. More
generally, we estimate
\begin{equation}\label{defN-unified}
	\begin{split}
		\mathcal N(F,\omega,b)(t,x)
		:=\int_0^t\!\!\int_{\R^2}b(s,v)\big\{
		&F(s,x-(t-s)v)\omega(v)\\
		&-F(s,X_{s,t}(x,v))\omega(V_{s,t}(x,v))
		\big\}\,dv\,ds.
	\end{split}
\end{equation}
In particular, $\mathcal R_k=\mathcal N(E,\nabla_v\mu,v^{\otimes k})$,
with the vector indices of $E$ and $\nabla_v\mu$ contracted.
A velocity change of variables puts both field factors on the free
ray and collects the difference into a small multiplier
$\Sigma_{s,t}$. Bounds for this multiplier give the basic
cancellation estimate. To gain decay for spatial derivatives,
we also use the change of variables $w=x-tv$. This transfers
derivatives to the field, velocity weights, and characteristic
corrections with a factor $t^{-1}$. The resulting remainders are
estimated using the bounds from Section~4 and further integrations
by parts in the oscillatory terms.

The most delicate contribution contains the leading part of
$\nabla_v^2Y_{s,t}$. Its pointwise estimate alone does not give the
decay needed for the highest field derivatives. We prove an
additional fractional difference estimate for this term and
formulate the resulting gain as an estimate against test kernels.
This allows derivatives to be distributed between the moment terms
and the linear response kernels. The weighted moment estimates in
Proposition~\ref{maines} then give the field bounds needed to close
the bootstrap. Thus the argument uses both pointwise decay and
additional regularity of the principal oscillatory term.

\subsection{Organization}
In Section 2 we introduce the moment equations, the field decomposition, and
the main bootstrap propositions. In Section 3 we collect preliminary estimates,
and in Section 4 we prove bounds on the characteristics and second velocity
derivative expansion. We prove the nonlinear moment estimates in 
Section 5 and then combine these ingredients in Section 6 to prove the bootstrap
proposition. In Section 7 we establish a suitable local existence and continuation criterion, matching our bootstrap norms.
Finally, in Section 8 we prove global existence and scattering, thus completing
the proof of Theorem~\ref{mainthm}.

\subsection*{Acknowledgements}
\begingroup
\emergencystretch=1em
A.~D.~I.\ was supported in part by NSF-UEFISCDI grant DMS-2407694
and a Simons Collaboration Grant on Wave Turbulence.
Q.-H.~Nguyen was supported by the CAS Project for Young Scientists in
Basic Research (Grant No.~YSBR-031) and the NSFC (Grant Nos.~1251101538
and 12595282).
\par\endgroup

\subsection*{AI use disclosure}
This paper was written by the authors and was not generated by artificial
intelligence. OpenAI's ChatGPT (Sol 5.6 and Astra 6) assisted with language
editing, organization, bibliographic checks, and checking calculations.
All mathematical content is the authors' own; they reviewed all AI-assisted
revisions and take full responsibility for the paper.

	\section{Dynamics of the density and bootstrap setup} 
	In this section, we derive the dynamical system satisfied by the macroscopic
	density and introduce the bootstrap framework that will be used
	to close the nonlinear estimates.
	\subsection{Characteristic flow}
	We introduce the backward characteristic flow associated with the Vlasov equation.
	For $(s,t)\in\mathcal{I}_T^2:=\{(s,t)\in[0,T]^2:\, s\le t\}$ and $(x,v)\in\mathbb{R}^2\times\mathbb{R}^2$,
	let $(X_{s,t},V_{s,t})$ solve the system of ODEs
	\begin{equation}\label{Lan1}
		\begin{aligned}
			\partial_s X_{s,t}(x,v) &= V_{s,t}(x,v), 
			&\qquad X_{t,t}(x,v) &= x,\\
			\partial_s V_{s,t}(x,v) &= E(s,X_{s,t}(x,v)),
			&\qquad V_{t,t}(x,v) &= v.
		\end{aligned}
	\end{equation}
	
	To measure the deviation from free transport, we define
	\begin{equation}\label{Lan6}
		\begin{aligned}
			Y_{s,t}(x,v) &:= X_{s,t}(x+tv,v)-x-sv,\\
			W_{s,t}(x,v) &:= V_{s,t}(x+tv,v)-v,
		\end{aligned}
	\end{equation}
	so that
	\begin{align*}
		&X_{s,t}(x,v)=x-(t-s)v+ Y_{s,t}(x-tv,v),\\&
		V_{s,t}(x,v)=v+ W_{s,t}(x-tv,v).
	\end{align*}
	By construction, $Y_{t,t}=0$ and $W_{t,t}=0$. 
	From \eqref{Lan1} we obtain the integral representations
	\begin{equation}\label{definitionYW}
		\begin{aligned}
			W_{s,t}(x,v) &= -\int_s^t E(\tau,x+\tau v+Y_{\tau,t}(x,v))\,d\tau,\\
			Y_{s,t}(x,v) &= \int_s^t (\tau-s)\,E(\tau,x+\tau v+Y_{\tau,t}(x,v))\,d\tau.
		\end{aligned}
	\end{equation}
	Using the method of characteristics, the solution of \eqref{eq2} satisfies
	\begin{equation}\label{t1}
		f(t,x,v)
		= f_0(X_{0,t}(x,v),V_{0,t}(x,v))
		- \int_0^t E(s,X_{s,t}(x,v))\cdot\nabla_v\mu(V_{s,t}(x,v))\,ds .
	\end{equation}

	The characteristic flow satisfies the usual composition property: for
	$0\le s\le t_1\le t\le T$,
	\begin{equation}\label{flow-composition}
		(X_{s,t},V_{s,t})
		=
		(X_{s,t_1},V_{s,t_1})
		\circ
		(X_{t_1,t},V_{t_1,t}).
	\end{equation}
	In terms of the correction variables $Y_{s,t}$ and $W_{s,t}$, this becomes
	the following renormalized cocycle identity:
	\begin{align}
		Y_{s,t}(x,v)
		&=
		Y_{t_1,t}(x,v)
		-
		(t_1-s)W_{t_1,t}(x,v)
		\notag\\
		&\quad+
		Y_{s,t_1}
		\left(
		x+Y_{t_1,t}(x,v)-t_1W_{t_1,t}(x,v),
		v+W_{t_1,t}(x,v)
		\right),
		\label{a31}
		\\
		W_{s,t}(x,v)
		&=
		W_{t_1,t}(x,v)
		+
		W_{s,t_1}
		\left(
		x+Y_{t_1,t}(x,v)-t_1W_{t_1,t}(x,v),
		v+W_{t_1,t}(x,v)
		\right).
		\label{a32}
	\end{align}
	We will use this identity later to analyze the large-time behavior
	of the characteristics. In particular, it allows us to compare the
	corrections associated with two different terminal times and is a key tool in the
	study of scattering for the characteristic flow.
	
	\subsection{Moment quantities}
	A direct computation using \eqref{eq2} shows that the function 
	\begin{equation*}
		\begin{split}
			F(t,x,v)&:=f(t,x,v)+\int_0^t E(s,x-(t-s)v)\cdot\nabla_v\mu(v)\, ds\\
			&=f_0(X_{0,t}(x,v),V_{0,t}(x,v))\\
			&+\int_0^t\big[E(s,x-(t-s)v)\cdot\nabla_v\mu(v)-E(s,X_{s,t}(x,v))\cdot\nabla_v\mu(V_{s,t}(x,v))\big]\, ds
		\end{split}
	\end{equation*}
	satisfies the equation
	\begin{equation*}
		\begin{split}
			\partial_{t}F(t,x,v)+v\cdot \nabla_xF(t,x,v)=-E(t,x)\cdot\nabla_v f(t,x,v).
		\end{split}
	\end{equation*}
	On $[0,T]$, the bound $|E(s,x)|\leq C_T\langle x\rangle^{-2}$
	and bounded characteristic corrections give
	\[
	\int_0^t\!\bigl(
	|E(s,x-(t-s)v)|+|E(s,X_{s,t})|
	\bigr)\,ds\lesssim_T\langle v\rangle^{-1}.
	\]
	Since $\langle V_{s,t}\rangle\asymp_T\langle v\rangle$,
	$|\nabla_v\mu(v)|\lesssim\langle v\rangle^{-4}$, and \eqref{a16'}
	holds, the representation above implies
	\[
	\int_{\mathbb R^2}\langle v\rangle^2|F(t,x,v)|\,dv<\infty.
	\]
	Integrating in $v$ yields
	\begin{equation}\label{z1}
		\begin{split}
			\frac{d}{dt}\Big[\int_{ \mathbb{R}^2}F(t,x,v)\,dv\Big]&=-\partial_{x^i}\int_{ \mathbb{R}^2}v^iF(t,x,v)\,dv,\\
			\frac{d}{dt}\Big[\int_{ \mathbb{R}^2}v^jF(t,x,v)\,dv\Big]&=-\partial_{x^i}\int_{ \mathbb{R}^2}v^iv^jF(t,x,v)\,dv+E_j(t,x)\rho(t,x). 
		\end{split}
	\end{equation}
	For $k=0,1,2$, we define
	\begin{equation}\label{z1.5}
		\begin{split}
			&\mathcal{I}_k(E)(t,x):=	\int_{\R^2} v^{\otimes k}f_0(X_{0,t}(x,v),V_{0,t}(x,v))\,dv,\\
			&\mathcal{R}_k(E_1,E)(t,x):=\int_{0}^{t}\int_{\R^2} v^{\otimes k}\big\{E_1(s,x-(t-s)v)\cdot\nabla_v\mu(v)\\
			&\qquad\qquad\qquad\qquad\qquad\qquad-E_1(s,X_{s,t}(x,v))\cdot\nabla_v\mu(V_{s,t}(x,v))\big\}\,dvds,\\
			&\mathcal{R}_k(E)=\mathcal{R}_k(E,E),
		\end{split}
	\end{equation}
	and
	\begin{equation}\label{z1.6}
		\mathbf{R}_k(t,x)=	\mathbf{R}_k(E)(t,x):=\mathcal{I}_k(E)(t,x)+\mathcal{R}_k(E)(t,x)=\int_{\R^2} v^{\otimes k}F(t,x,v)\,dv.
	\end{equation}
	In $\mathcal R_k(E_1,E)$, the characteristic flow
	$(X_{s,t},V_{s,t})$ solves \eqref{Lan1} with force field $E$.
	Since $\operatorname{div}_xE=\rho$ and $\operatorname{curl}_xE=0$
	by \eqref{eq2}, for $j'\in\{1,2\}$ we have
	\[
	\rho E_{j'}
	=\partial_{x^j}(E_jE_{j'})-E_j\partial_{x^j}E_{j'}
	=\partial_{x^j}(E_jE_{j'})-\frac12\partial_{x^{j'}}|E|^2.
	\]
	Combining this identity with \eqref{z1} gives
	\begin{align}\label{z5}
		&	\frac{d}{dt}	\mathbf{R}_0(t,x)=-\operatorname{div}	\mathbf{R}_1(t,x),\\&
		\frac{d}{dt}	\mathbf{R}_1(t,x)=-\operatorname{div}	\mathbf{R}_2(t,x)+\operatorname{div}(E(t,x)\otimes E(t,x))-\frac{1}{2}\nabla(| E(t,x)|^2)\label{z5'}.
	\end{align}
	
	\subsection{Macroscopic density equation}
	
	Using $\operatorname{div}_xE=\rho$, integration by parts in $v$
	gives, for $0\leq s\leq t$,
	\begin{align*}
	\int_{\mathbb R^2}E(s,x-(t-s)v)\cdot\nabla_v\mu(v)\,dv=(t-s)\int_{\mathbb R^2}\rho(s,x-(t-s)v)\mu(v)\,dv.
	\end{align*}
	The boundary term vanishes since $E(s)$ is bounded and
	$\mu(v)=O(\langle v\rangle^{-3})$.
	Integrating the definition of $F$ in $v$ and using \eqref{z1.6}
	therefore yields the Volterra equation
	\begin{equation}\label{rhoEq}
		\rho(t,x)
		+ \int_0^t\!\!\int_{\R^2} (t-s)\rho(s,x-(t-s)v)\mu(v)\,dvds
		= \mathbf{R}_0(t,x).
	\end{equation}
	Taking the Fourier transform in $x$ gives
	\begin{equation}\label{Volterra}
		\widehat{\rho}(t,\xi)
		+ \int_0^t (t-\tau)\widehat{\rho}(\tau,\xi)\widehat{\mu}((t-\tau)\xi)\,d\tau
		= \widehat{\mathbf{R}_0}(t,\xi).
	\end{equation}
	
	As shown in \cite{AIonescu2022}, this equation admits the representation
	\begin{equation}\label{z40}
		\rho(t)
		= \mathbf{R}_0(t)
		- \int_0^t \sin(s)e^{-s|\nabla|}\mathbf{R}_0(t-s)\,ds .
	\end{equation}
	
	Let $$\tilde\rho_{j,0}(x)=\tilde\rho_{j,0}(f_0)(x):=\int_{\R^2} v^{\otimes j}f_0(x,v)\,dv.$$ Then
	$\mathbf{R}_0(0)=\tilde\rho_{0,0}$ and $\mathbf{R}_1(0)=\tilde\rho_{1,0}$.
	Integrating by parts and using \eqref{z5}, we obtain
	\begin{equation}\label{z40.1}
		\rho(t)= \cos(t)e^{-t|\nabla|}\tilde\rho_{0,0}
		+ \int_0^t \cos(t-s)e^{-(t-s)|\nabla|}
		\big(|\nabla|\mathbf{R}_0-\operatorname{div}\mathbf{R}_1\big)(s)\,ds .
	\end{equation}
	
	This representation reveals the oscillatory structure
	responsible for the resonance at frequency $1$. We introduce a cutoff $\chi(|\nabla|)$ supported at low frequencies and use 
	\eqref{z5}–\eqref{z5'} to arrive at the decomposition
	\begin{equation}\label{z40.2}
		\begin{split}
			\rho(t,x)&=\cos(t)e^{-t|\nabla|} \tilde\rho_{0,0}(x)+\sin(t)\chi(|\nabla|)e^{-t|\nabla|} (|\nabla|\tilde\rho_{0,0}-\operatorname{div}\tilde{\rho}_{1,0})(x)\\&+\int_{0}^{t}\cos(s)(1-\chi)(|\nabla|)e^{-s|\nabla|} (|\nabla|\mathbf{R}_0-\operatorname{div}\mathbf{R}_1)(t-s)(x)ds\\&
			+	\int_{0}^{t}\sin(t-s)\chi(|\nabla|)e^{-(t-s)|\nabla|} (|\nabla|^2\mathbf{R}_0-2|\nabla|\operatorname{div}\mathbf{R}_1+\operatorname{div}^2\mathbf{R}_2)(s)(x)ds\\&
			-	\int_{0}^{t}\sin(t-s)\chi(|\nabla|)e^{-(t-s)|\nabla|} (\operatorname{div}^2(E\otimes E)-\frac{1}{2}\Delta(| E|^2))(s)(x)ds.
		\end{split}
	\end{equation}
	
	\subsection{Electric field decomposition}
	Applying $\nabla\Delta^{-1}$, we write
	\begin{equation}\label{de1}
		\begin{split}
			&E(t)=	E^{reg}(t)+E^{osc}(t)=E^{reg}(t)+\cos(t) E^{\cos}(t)+\sin(t)E^{\sin}(t),\\
			&E^{\cos}:=E^{\cos}_{non}+E_{lin}^{\cos},\qquad E^{\sin}:=E^{\sin}_{non}+E_{lin}^{\sin},
		\end{split}
	\end{equation}
	where
	\begin{equation}\label{de1.1}
		\begin{split}
			E^{reg}(t)&:=\int_{0}^{t}\cos(t-s)(1-\chi)(|\nabla|)\nabla\Delta^{-1}e^{-(t-s)|\nabla|} (|\nabla|\mathbf{R}_0-\operatorname{div}\mathbf{R}_1)(s)\,ds\\
			&-	\int_{0}^{t}\sin(t-s)\chi(|\nabla|)e^{-(t-s)|\nabla|}\nabla\Delta^{-1} (\operatorname{div}^2(E\otimes E)-\frac{1}{2}\Delta(| E|^2))(s)\,ds,\\
			E_{lin}^{\cos}(t)&:=	E_{lin}^{\cos}(f_0)(t):=\nabla\Delta^{-1}e^{-t|\nabla|} \tilde\rho_{0,0}(f_0),\\
			E_{lin}^{\sin}(t)&:=	E_{lin}^{\sin}(f_0)(t):=\chi(|\nabla|)\nabla\Delta^{-1}e^{-t|\nabla|} [|\nabla|\tilde\rho_{0,0}(f_0)-\operatorname{div}\tilde{\rho}_{1,0}(f_0)],
		\end{split}
	\end{equation}
	and the nonlinear oscillatory components are defined as
	\begin{equation}\label{z40.4}
		\begin{split}
			E^{\cos}_{non}&:=-\int_{0}^{t}\sin(s)\chi(|\nabla|)e^{-(t-s)|\nabla|} \nabla\Delta^{-1}(|\nabla|^2\mathbf{R}_0-2|\nabla|\operatorname{div}\mathbf{R}_1+\operatorname{div}^2\mathbf{R}_2)(s)\,ds,\\
			E^{\sin}_{non}&:=\int_{0}^{t}\cos(s)\chi(|\nabla|)e^{-(t-s)|\nabla|} \nabla\Delta^{-1}(|\nabla|^2\mathbf{R}_0-2|\nabla|\operatorname{div}\mathbf{R}_1+\operatorname{div}^2\mathbf{R}_2)(s)\,ds.
		\end{split}
	\end{equation}
	
	\subsection{The main bootstrap propositions}
	
	We define $\langle t,x\rangle=\sqrt{1+|x|^2+t^2}$ and
	\begin{equation}\label{cbc1}
		\Phi(t,x):=\langle t,x\rangle^{-2-\kappa_0},~~\kappa_0=10^{-4}.
	\end{equation}
	For a spacetime function $g\in L^\infty([0,T]\times\R^2)$, we define our main norms as
	\begin{equation}\label{cbc2}
		\begin{split}
		&	\|g\|_{1,T}=\sum_{n=0}^{2}\sup_{t\in [0,T]}\big\| \langle t\rangle^{\frac{1}{2}+\frac{2n}{3}} \Phi^{-1}(t)\nabla^n g(t)\big\|_{L^{\infty}_x},\\
			&	\|g\|_{2,T}=\sum_{n=0}^{2}\sup_{t\in [0,T]}\|\langle t\rangle^{\frac{1}{2}+\frac{2n}{3}}	\Phi^{-1}(t)\nabla^n \partial_tg(t)\|_{L^{\infty}_x}+\sup_{t\in [0,T]}\|\langle t\rangle^{-50\kappa_0}\Phi^{-1}(t)g(t)\|_{L^{\infty}}\\&\quad+\sum_{n\in\{1,2\}}\sup_{t\in [0,T]}\|\langle t\rangle^{-2\kappa_0+\frac{2n}{3}}\Phi^{-1}(t)\nabla^ng(t)\|_{L^{\infty}}+\sup_{t\in[0,T]}
			\|\langle t\rangle^{\frac{11}{6}}\Phi^{-1}(t)\nabla_x^3g(t)\|_{L_x^{\infty}}.
		\end{split}
	\end{equation}
	All terms use the same spatial weight $\Phi^{-1}(t,x)$; only the time
	weights vary. At each fixed time, this requires only
	$\langle x\rangle^{-2-\kappa_0}$ decay, which is integrable in $\mathbb R^2$.
	
	We are now ready to state our main propositions.

	\begin{proposition}[Bootstrap]\label{propoglobal3d}
		There exist constants $C\geq1$ and
		$\epsilon_0\in(0,(1+2C)^{-1}]$, independent of $T$, the initial
		data $f_0$, and the solution, such that for every $T>0$,
		if $(f,E)$ is a solution of \eqref{eq2} on $[0,T]$ satisfying
		\begin{equation}\label{bo1}
			M_T := \|E^{reg}\|_{1,T} + \|(E^{\cos}_{non},E^{\sin}_{non})\|_{2,T} + [f_0] \le \epsilon_0,
		\end{equation}
		and if the initial data $f_0$ satisfies the neutrality assumption \eqref{a16}, then
		\begin{align}\label{B}
			\|E^{reg}\|_{1,T} + \|(E^{\cos}_{non},E^{\sin}_{non})\|_{2,T}
			\leq C [f_0] + CM_T^2.
		\end{align}
		In particular,
		\begin{align}\label{B'}
			\|E^{reg}\|_{1,T} + \|(E^{\cos}_{non},E^{\sin}_{non})\|_{2,T}
			\leq 2C [f_0].
		\end{align}
	\end{proposition}
	
	\begin{proposition}[Initial local existence and continuation]\label{Localwell}
		Let $T\geq0$, let $\epsilon\in(0,\epsilon_0]$, where $\epsilon_0$
		is as in Proposition~\ref{propoglobal3d}, and assume that $f_0$
		satisfies the neutrality condition \eqref{a16}.
		If $T>0$, assume that $(f,E)$ is a solution of \eqref{eq2}
		on $[0,T]$ with initial datum $f_0$ and
		\begin{equation}\label{bo1'}
		A_T+[f_0]\leq\epsilon\leq\epsilon_0,
		\end{equation}
		where
		\begin{equation}\label{local-AT}
			A_T:=\|E^{reg}\|_{1,T}+\|(E^{\cos}_{non},E^{\sin}_{non})\|_{2,T}.
		\end{equation}
		If $T=0$, assume only that $[f_0]\leq\epsilon$.
		
		Then there exist $\delta=\delta(T)>0$ and a unique solution
		$(f,E)$ on $[0,T+\delta]$ with initial datum $f_0$,
		agreeing with the given solution on $[0,T]$ when $T>0$, and satisfying
		\begin{equation}\label{bo2}
			A_{T+\delta}\leq C_1\epsilon,
		\end{equation}
		where $C_1\geq1$ is independent of $T$, $\epsilon$, $f_0$,
		and the solution.
	\end{proposition}

	\section{Some preliminary estimates}
	
	Let $G(t,\cdot)$ denote the convolution kernel of
	$e^{-t|\nabla|}$, so that $\widehat G(t,\xi)=e^{-t|\xi|}$. 
	We decompose $G=\chi(|\nabla|)G+(1-\chi)(|\nabla|)G$ into low- and high-frequency components using a smooth cutoff.
	By standard techniques in Fourier analysis, we obtain the following kernel estimates.
	
	\begin{lemma}\label{keres}
		For every $t\geq 0$, $x\in\mathbb R^2$, and all integers $j_1,j_2\ge0$ we have
		\begin{align}\label{esG1}
			&	|\partial_{t}^{j_1}\nabla_x^{j_2}m_{-1}(\nabla)\chi(|\nabla|)G(t,x)|\lesssim \frac{1}{\langle x,t\rangle^{1+j_1+j_2}},\\
			&|\partial_{t}^{j_1}\nabla_x^{j_2}m_{-1}(\nabla)(1-\chi)(|\nabla|)G(t,x)|\lesssim \frac{1}{(t+|x|)^{1+j_1+j_2}\langle x,t\rangle^{7}}.\label{esG2}
		\end{align}
	where $m_{-1}\in C^\infty(\R^2\setminus\{0\})$ is a homogeneous symbol of degree $-1$.
	\end{lemma}
	
	We now estimate the functions $E^{\cos}_{lin}$ and $E^{\sin}_{lin}$.
	\begin{lemma}\label{eslinear}
		Assume that $f_0\in C^3(\mathbb R_x^2\times\mathbb R_v^2;\mathbb R)$,
		$[f_0]<\infty$, and the neutrality condition \eqref{a16} holds.
		Let $\ast\in\{\cos,\sin\}$ and $k\in\{0,1,2,3\}$. Then, for every
		$t\geq0$ and $x\in\mathbb R^2$,
		\begin{align}\label{a23}
			|\nabla_{x,t}^k E_{lin}^{\ast}(t,x)|
			\lesssim [f_0]
			\begin{cases}
				\langle t,x\rangle^{-2-k},&k\in\{0,1\},\\
				\langle t\rangle^{-k+1+20\kappa_0}
				\langle t,x\rangle^{-3-20\kappa_0},&k\in\{2,3\},
			\end{cases}
		\end{align}
		where
		\[
		|\nabla_{x,t}^kF|:=
		\sum_{j+|\alpha|=k}|\partial_t^j\partial_x^\alpha F|.
		\]
	\end{lemma}
	
	\begin{proof}
		Let $p:=3+20\kappa_0$. By linearity, it is
		enough to prove the estimate when $[f_0]=1$. The neutrality condition
		\eqref{a16} and the definition \eqref{a16'} give
		\begin{equation}\label{a21.1}
			\int_{\mathbb{R}^2}\tilde\rho_{0,0}(x)\,dx = 0
		\end{equation}
		and
		\begin{align}\label{a21}
			\sum_{n=0}^{3}
			|\nabla_x^n(\tilde\rho_{0,0},\tilde\rho_{1,0})(x)|
			\lesssim \langle x\rangle^{-p}.
		\end{align}
		
		Let $K_1$ and $K_2$ denote the kernels of the operators
		\[
		\chi(|\nabla|)m_{-1}(\nabla)e^{-t|\nabla|}
		\quad\hbox{and}\quad(1-\chi)(|\nabla|)m_{-1}(\nabla)e^{-t|\nabla|},
		\]
		where $m_{-1}\in C^\infty(\R^2\setminus\{0\})$ is a homogeneous symbol of degree $-1$. Lemma~\ref{keres} gives
		\begin{align}\label{a20}
			|\nabla^n_{x,t}K_1(t,x)|&\lesssim \langle t,x\rangle^{-1-n},\\
			\label{a19}|\nabla^n_{x,t}K_2(t,x)|&\lesssim (t+|x|)^{-1-n}\langle t,x\rangle^{-7}
		\end{align}
		for any $n\geq 0$. 
		
		Fix $k\in\{0,1,2,3\}$ and a mixed derivative $\partial_{x,t}^{\gamma}$ with $|\gamma|=k$, and recall the definition of $E_{{lin}}^{\cos}$ in \eqref{de1.1}. Since every time derivative of \(e^{-t|\nabla|}\) contributes a factor \(-|\nabla|\), on the high-frequency support all \(k\) derivatives may be transferred to \(\widetilde\rho_{0,0}\), with the resulting degree-zero angular multipliers absorbed into \(m_{-1}\). Thus
		\begin{align*}
			|\partial_{x,t}^{\gamma}E_{lin}^{\cos}(t,x)|&\lesssim\mathbf 1_{0\leq t\leq1}|K_2(t)|*|\nabla_x^k\tilde\rho_{0,0}|(x)+
			\mathbf 1_{t\geq1}|\partial_{x,t}^{\gamma}K_2(t)|*|\tilde\rho_{0,0}|(x)\\
			&+|\partial_{x,t}^{\gamma}K_1(t)*\tilde\rho_{0,0}(x)|.
		\end{align*}
		
		For $t>0$ and $n\geq0$ (also for $t=0$ when $n=0$, with $t^{-n}=1$), a direct near--far decomposition gives
		\begin{equation}\label{rg1}
			\int_{\mathbb{R}^2}(t+|x-y|)^{-1-n}
			\langle t,x-y\rangle^{-7}\langle y\rangle^{-p}\,dy
			\lesssim t^{-n}\langle t\rangle^{-3}
			\langle t,x\rangle^{-p}.
		\end{equation}
		Using \eqref{a19}, \eqref{a21}, and \eqref{rg1} with $n=0$ when
		$0\leq t\leq1$ and with $n=k$ when $t\geq1$, we obtain
		\begin{align}\label{a22}
			|\partial_{x,t}^{\gamma}E_{{lin}}^{\cos}(t,x)|
			\lesssim\langle t\rangle^{-3}\langle t,x\rangle^{-p}+|\partial_{x,t}^{\gamma}K_1(t)*\tilde\rho_{0,0}(x)|.
		\end{align}

		It remains to estimate the low-frequency term. Put
		$R:=\langle t,x\rangle$. By \eqref{a21.1},
		\begin{align*}
			|\partial_{x,t}^{\gamma}K_1(t)*\tilde\rho_{0,0}(x)|&\le\int_{\mathbb{R}^2}
			|\partial_{x,t}^{\gamma}K_1(t,x-y)-\partial_{x,t}^{\gamma}K_1(t,x)|\,|\tilde\rho_{0,0}(y)|\,dy.
		\end{align*}
		On $|y|\leq R/2$, the mean-value theorem, \eqref{a20}, and
		$\int |y||\tilde\rho_{0,0}(y)|\,dy\lesssim1$ give a bound
		$\lesssim R^{-2-k}$. On $|y|>R/2$, \eqref{a20} gives
		\begin{align*}
			&\int_{|y|>R/2}
			\bigl(\langle t,x-y\rangle^{-1-k}+R^{-1-k}\bigr)
			\langle y\rangle^{-p}\,dy\lesssim
			\begin{cases}
				R^{-2-k},&k\in\{0,1\},\\
				\langle t\rangle^{-k+1+20\kappa_0}R^{-p},&k\in\{2,3\}.
			\end{cases}
		\end{align*}
		Indeed, split once more according to $|x-y|\leq R/2$ or $|x-y|>R/2$, and use
		\[
		\int_{|z|\leq R/2}\langle t,z\rangle^{-1-k}\,dz
		\lesssim
		\begin{cases}
			R,&k=0,\\
			\log(2+R),&k=1,\\
			\langle t\rangle^{1-k},&k\in\{2,3\},
		\end{cases}
		\qquad
		\int_{|y|>R/2}\langle y\rangle^{-p}\,dy\lesssim R^{2-p}.
		\]
		Since $R\geq\langle t\rangle$ and $p=3+20\kappa_0$, this proves
		\begin{equation}\label{rg2}
			|\partial_{x,t}^{\gamma}K_1(t)*\tilde\rho_{0,0}(x)|\lesssim
			\begin{cases}
				R^{-2-k},&k\in\{0,1\},\\
				\langle t\rangle^{-k+1+20\kappa_0}R^{-p},&k\in\{2,3\}.
			\end{cases}
		\end{equation}
		The first term on the right-hand side of \eqref{a22} is no larger
		than the right-hand side of \eqref{rg2}; hence \eqref{a23} follows
		for $E_{{lin}}^{\cos}$.
		
		We consider now the component $E^{\sin}_{{lin}}$. Since
		$\partial_te^{-t|\nabla|}=-|\nabla|e^{-t|\nabla|}$, we have
		\begin{equation*}
		E_{{lin}}^{\sin}=-(\partial_tK_1)*\tilde\rho_{0,0}-\sum_{j=1}^{2}(\partial_{x_j}K_1)*\tilde\rho_{1,0}^j.
		\end{equation*}
		Thus \eqref{a20} with one additional derivative, followed by
		\eqref{a21}, gives
		\[
		|\partial_{x,t}^{\gamma}E_{lin}^{\sin}(t,x)|
		\lesssim
		\int_{\mathbb R^2}\langle t,x-y\rangle^{-2-k}
		\langle y\rangle^{-p}\,dy\lesssim
		\begin{cases}
			R^{-2-k},&k\in\{0,1\},\\
			\langle t\rangle^{-k+1+20\kappa_0}R^{-p},&k\in\{2,3\},
		\end{cases}
		\]
	        where the last bounds follow by considering several cases. This gives the bounds \eqref{a23} for the $E_{lin}^{\sin}$ component, as claimed.	
		\end{proof}	
	
	\subsection{Time–space integral estimates}
	We prove now several elementary integral bounds.
	In Lemmas~\ref{z216}--\ref{Le3} and their proofs, implicit constants
	may depend on the fixed exponents and on $\kappa_0$, but are independent
	of $s,t,x,v$. No uniformity near borderline exponent values is asserted.

	\begin{lemma}\label{z216}
		(i) If $0\le s\le t$ and $x,v\in\R^2$, then
		\begin{equation*}
			\int_s^t \langle\tau,x-\tau v\rangle^{-1}\,d\tau
			\lesssim
			\langle v\rangle^{-1}\log\!\Big(2+\frac{t\langle v\rangle}{s+1}\Big).
		\end{equation*}
		
		(ii) Assume that $a_1>1$, $a_2\in\mathbb R$, and set
			$\gamma:=a_2-a_1+1.$
	 If $\gamma>0$ and $0\leq s\leq t$, then
		\begin{equation}\label{z210}
			\int_s^t\langle\tau\rangle^{a_2}\langle\tau,x-\tau v\rangle^{-a_1}\,d\tau\lesssim\langle t\rangle^{\gamma}
			\langle v\rangle^{-1}.
		\end{equation}
		On the other hand, if $\gamma<0$, then \begin{equation}\label{z212}
			\int_s^t
			\langle\tau\rangle^{a_2}
			\langle\tau,x-\tau v\rangle^{-a_1}\,d\tau
			\lesssim
			\langle s\rangle^{\gamma}
			\langle v\rangle^{-1}.
		\end{equation}
	If $\gamma=0$, then
		\begin{equation}\label{z212-borderline}
			\int_s^t\langle\tau\rangle^{a_2}
			\langle\tau,x-\tau v\rangle^{-a_1}\,d\tau
			\lesssim\langle v\rangle^{-1}
			\log\!\left(2+\frac{t+1}{s+1}\right).
		\end{equation}
	\end{lemma}
	
	\begin{proof} (i) If $|v|\le 1$, then $\langle v\rangle\approx 1$ and
		\[
		\int_s^t \langle \tau,x-\tau v\rangle^{-1}\,d\tau\le\int_s^t \langle \tau\rangle^{-1}\,d\tau\lesssim\log\Big(2+\frac{t}{s+1}\Big)\lesssim\langle v\rangle^{-1}\log\Big(2+\frac{t\langle v\rangle}{s+1}\Big).
		\]
		Assume now that $|v|\ge 1$. Let $e=v/|v|$ and
		decompose $x=x_\parallel e+x_\perp$, where
		$x_\perp\cdot e=0$. Dropping the nonnegative term $|x_\perp|^2$,
		we estimate
		\[
		\int_s^t \langle \tau,x-\tau v\rangle^{-1}\,d\tau\lesssim\int_s^t\big((s+1)^2+|x_\parallel-\tau|v||^2\big)^{-1/2}\,d\tau\lesssim |v|^{-1}\int_J \big((s+1)^2+r^2\big)^{-1/2}\,dr,
		\]
		where $J$ is an interval of length $\leq |v|(t-s)\leq |v|t$.
			Since the integrand is even and nonincreasing in $|r|$, its integral
			over $J$ is bounded by that over the centered interval below. Thus
		\[
		\int_s^t \langle \tau,x-\tau v\rangle^{-1}\,d\tau\lesssim |v|^{-1}\int_{-t|v|}^{t|v|}\big((s+1)^2+r^2\big)^{-1/2}\,dr\lesssim\langle v\rangle^{-1}\log\Big(2+\frac{t\langle v\rangle}{s+1}\Big),
		\]
		as desired.

		(ii) By rotation invariance, we may assume that $v=(v_1,0)$, $v_1\geq 0$. Then
		\begin{equation}\label{gr3.5}
			|x-\tau v|^2=|x_1-\tau v_1|^2+|x_2|^2,\qquad 
			\langle\tau,x-\tau v\rangle^{-a_1}\lesssim\langle \tau,x_1-\tau v_1\rangle^{-a_1}.
		\end{equation}
		
		We first prove a dyadic estimate. Assume $T\geq 1$ and let $I_T$ be a time interval with
		$\langle\tau\rangle\approx T$. Then
		\begin{equation}\label{gr3.7}
			\int_{I_T}\langle\tau,x-\tau v\rangle^{-a_1}\,d\tau\lesssim\langle v\rangle^{-1}T^{1-a_1}.
		\end{equation}
		Indeed, this follows directly from \eqref{gr3.5} if $|v|\le1$. On the other hand, if $|v|\ge1$, then we make the change of variables $r=x_1-\tau v_1$ to bound
		\[
		\int_{I_T}\langle \tau,x_1-\tau v_1\rangle^{-a_1}\,d\tau\lesssim|v|^{-1}\int_{\mathbb R}(T^2+r^2)^{-a_1/2}\,dr\lesssim|v|^{-1}T^{1-a_1}.
		\]
		This proves the dyadic estimate \eqref{gr3.7}. Multiplying by $\langle\tau\rangle^{a_2}\approx T^{a_2}$ on $I_T$, we get
		\[\int_{I_T}\langle\tau\rangle^{a_2}\langle\tau,x-\tau v\rangle^{-a_1}\,d\tau\lesssim\langle v\rangle^{-1}T^{a_2-a_1+1}.
		\]
		Summing over dyadic intervals intersecting $[s,t]$, the largest scale
			gives $\langle t\rangle^\gamma$ when $\gamma>0$, while the smallest
			gives $\langle s\rangle^\gamma$ when $\gamma<0$.
			When $\gamma=0$, the number of intervals is
			$\lesssim\log\!\left(2+\frac{t+1}{s+1}\right)$.
			This proves \eqref{z210}--\eqref{z212-borderline}.
	\end{proof}
	
	We record two auxiliary estimates that will be used repeatedly. 
		
	\begin{lemma}\label{Le1}
		For all $t\ge0$ and $x\in\mathbb R^2$, the following estimates hold:
		\begin{align}
			\int_{\mathbb R^2}
			\langle x-tv\rangle^{-2-\kappa_0}
			\frac{dv}{\langle v\rangle^{2+\kappa_0}}
			&\lesssim
			\langle t\rangle^{\kappa_0}
			\langle t,x\rangle^{-2-\kappa_0},
			\label{z201}
			\\
			\int_{\mathbb R^2}
			\langle x-tv\rangle^{-2}
			\frac{dv}{\langle v\rangle^{2+\kappa_0}}
			&\lesssim
			\log(2+t)
			\langle t,x\rangle^{-2}.
			\label{z201p}
		\end{align}
	\end{lemma}
\begin{proof}
For $0\leq t\leq1$, bounded $x$ follows by integrability; for $|x|>2$,
use the velocity splitting below and
$|x|^{-\kappa_0}\log(2+|x|)\lesssim_{\kappa_0}1$.
We now assume $t\geq1$.
		
		We first prove \eqref{z201p}. Assume first that $|x|\le2t$. Then
		$\langle t,x\rangle\approx\langle t\rangle$. By the change of variables
		$w=tv$, we get
		\begin{align*}
			\int_{\mathbb R^2}
			\langle x-tv\rangle^{-2}
			\frac{dv}{\langle v\rangle^{2+\kappa_0}}
			&=
			t^{-2}
			\int_{\mathbb R^2}
			\langle x-w\rangle^{-2}
			\frac{dw}{\langle w/t\rangle^{2+\kappa_0}}\lesssim
			t^{-2}\log(2+t).
		\end{align*}
		This proves \eqref{z201p} when $|x|\le2t$.
		
		On the other hand, if $|x|\geq 2t$ then we split the velocity integral into two
		regions. If $|v|\le |x|/(2t)$, then $|x-tv|\gtrsim |x|$, and hence
		\begin{align*}
			\int_{|v|\le |x|/(2t)}
			\langle x-tv\rangle^{-2}
			\frac{dv}{\langle v\rangle^{2+\kappa_0}}
			&\lesssim
			\langle x\rangle^{-2}.
		\end{align*}
For $|v|>|x|/(2t)$, set $w=tv$ and split further at $|w|=2|x|$ to obtain
\[
\int_{|v|>|x|/(2t)}\langle x-tv\rangle^{-2}
\frac{dv}{\langle v\rangle^{2+\kappa_0}}
\lesssim t^{\kappa_0}|x|^{-2-\kappa_0}\log(2+|x|)
\lesssim_{\kappa_0}|x|^{-2}\log(2+t).
\]
Here $\lambda=|x|/t\geq2$ and
$\lambda^{-\kappa_0}\log(2+\lambda t)\leq\log(2+t)+\kappa_0^{-1}$.
This completes the proof of \eqref{z201p}.
	The proof of \eqref{z201} is similar and easier, since the spatial kernel
		has stronger decay.
		\end{proof}

	For $\sigma\in\mathbb R$, define
	\begin{align}
		\mathbf S_\sigma(t,x)
		&:=
		\int_0^t\int_{\mathbb R^2}
		\langle s,x-(t-s)v\rangle^{-2-\kappa_0}
		\langle s\rangle^\sigma
		\frac{dv\,ds}{\langle v\rangle^{2+\kappa_0}},
		\label{S-sigma-def}
		\\
		\mathbf T_\sigma(t,x)
		&:=
		\int_0^t\int_{\mathbb R^2}
		\langle s,x-(t-s)v\rangle^{-2}
		\langle s\rangle^\sigma
		\frac{dv\,ds}{\langle v\rangle^{2+2\kappa_0}}.
		\label{T-sigma-def}
	\end{align}
		We use the following convention: for $\gamma\in\mathbb R$ and $a\ge0$,
	\begin{equation}\label{positive-part-convention}
		\langle a\rangle^{\gamma_+}
		:=
		\mathbf 1_{\gamma>0}\langle a\rangle^\gamma
		+
		\mathbf 1_{\gamma<0}
		+
		\mathbf 1_{\gamma=0}\log(2+a).
	\end{equation}
	Here $\langle a\rangle^{\gamma_+}$ denotes the entire weight in
	\eqref{positive-part-convention}; at $\gamma=0$ it equals $\log(2+a)$.
	Additional powers of $\langle a\rangle$ are written as separate factors.
	
	\begin{lemma}\label{Le3}
		For any $\sigma\in\mathbb R$, $t\ge0$, and $x\in\mathbb R^2$, one has
		\begin{align}
			\mathbf S_\sigma(t,x)
			&\mathrel{\lesssim_{\sigma,\kappa_0}}
			\langle t\rangle^{(\sigma-\kappa_0+1)_+}
			\langle t\rangle^{\kappa_0}
			\langle t,x\rangle^{-2-\kappa_0},
			\label{z143a}
			\\
			\mathbf T_\sigma(t,x)
			&\mathrel{\lesssim_{\sigma,\kappa_0}}
			\log(2+t)
			\langle t\rangle^{(\sigma+1)_+}
			\langle t,x\rangle^{-2}.
			\label{z143b}
		\end{align}
	\end{lemma}
	
		\begin{proof} We may assume that $t\geq 1$ and consider two cases.
		\medskip
		
		{\bf{Case 1: $\langle x\rangle\leq 10 t$.}} In this case, $\langle t,x\rangle\approx \langle t\rangle$, so it is enough to prove
		\begin{align}
			\mathbf{S}_{\sigma}(t,x)
			&\lesssim
			\langle t\rangle^{(\sigma-\kappa_0+1)_+}\langle t\rangle^{-2},
			\label{z142a}\\
			\mathbf{T}_{\sigma}(t,x)
			&\lesssim
			\log(2+t)\langle t\rangle^{(\sigma+1)_+}\langle t\rangle^{-2}.
			\label{z142b}
		\end{align}

		We first estimate $\mathbf{S}_{\sigma}(t,x)$. We use the decomposition
		\begin{align*}
			\mathbf{S}_{\sigma}(t,x)&\lesssim
			\int_0^t\int_{\mathbb R^2}
			\langle t\rangle^{-2-\kappa_0}
			\langle t\rangle^\sigma
			\frac{dv\,ds}{\langle v\rangle^{2+\kappa_0}}+
			\frac1{t^2}
			\int_0^{t/4}
			\int_{\mathbb R^2}\langle s,w\rangle^{-2-\kappa_0}\langle s\rangle^\sigma\,dw\,ds.
		\end{align*}
		For the first term we obtain
		\begin{align*}
			\int_0^t\int_{\mathbb R^2}
			\langle t\rangle^{-2-\kappa_0}
			\langle t\rangle^\sigma
			\frac{dv\,ds}{\langle v\rangle^{2+\kappa_0}}&\lesssim
			\langle t\rangle^{\sigma-1-\kappa_0}
			\lesssim
			\langle t\rangle^{(\sigma-\kappa_0+1)_+}\langle t\rangle^{-2}.
		\end{align*}
		For the second term, we use $\int_{\mathbb R^2}
			\langle s,w\rangle^{-2-\kappa_0}\,dw\lesssim
			\langle s\rangle^{-\kappa_0}$, so
		\begin{align*}
			\frac1{t^2}\int_0^{t/4}
			\int_{\mathbb R^2}\langle s,w\rangle^{-2-\kappa_0}\langle s\rangle^\sigma\,dw\,ds&\lesssim
			\frac1{t^2}
			\int_0^t
			\langle s\rangle^{\sigma-\kappa_0}\,ds
			\lesssim
			\langle t\rangle^{(\sigma-\kappa_0+1)_+}\langle t\rangle^{-2}.
		\end{align*}
		The desired bounds \eqref{z142a} follow.
		
		The estimate for $\mathbf{T}_{\sigma}(t,x)$ is similar. First we estimate, as before
	\begin{align*}
	\mathbf T_\sigma(t,x)
	&\lesssim
	\int_0^t\int_{\mathbb R^2}
	\langle t\rangle^{-2}
	\langle t\rangle^\sigma
	\frac{dv\,ds}{\langle v\rangle^{2+2\kappa_0}}+
	\frac{1}{t^2}
	\int_0^{t/4}
	\langle s\rangle^\sigma
	\int_{\mathbb R^2}
	\langle s,w\rangle^{-2}
	\left\langle\frac{x-w}{t}\right\rangle^{-2-2\kappa_0}
	\,dw\,ds.
\end{align*}
For the first term we obtain
\begin{align*}
	\int_0^t\int_{\mathbb R^2}
	\langle t\rangle^{-2}
	\langle t\rangle^\sigma
	\frac{dv\,ds}{\langle v\rangle^{2+2\kappa_0}}
	&\lesssim
	\langle t\rangle^{\sigma-1}
	\lesssim
	\langle t\rangle^{(\sigma+1)_+}\langle t\rangle^{-2}.
\end{align*}
For the second term, we notice that
\begin{equation*}
	\int_{\mathbb R^2}
	\langle s,w\rangle^{-2}
	\left\langle\frac{x-w}{t}\right\rangle^{-2-2\kappa_0}
	\,dw
	\lesssim
	\log(2+t),
\end{equation*}
which follows by considering the two regions $|w|\lesssim t$ and $|w|\gg t$ separately. Thus
\begin{align*}
	\frac{1}{t^2}
	\int_0^{t/4}
	\langle s\rangle^\sigma
	\int_{\mathbb R^2}
	\langle s,w\rangle^{-2}
	\left\langle\frac{x-w}{t}\right\rangle^{-2-2\kappa_0}
	\,dw\,ds\lesssim
	\log(2+t)\langle t\rangle^{(\sigma+1)_+}\langle t\rangle^{-2}.
\end{align*}
Combining the two estimates proves \eqref{z142b}.
		\medskip

{\bf{Case 2: $\langle x\rangle\geq 10t$.}}
In this case, $\langle t,x\rangle\approx\langle x\rangle$. Let
$r=t-s$. For fixed $s$, we split the velocity
integration according to whether $r|v|\leq |x|/2$ or $r|v|\geq |x|/2$.

On the first region, $|x-rv|\gtrsim |x|$, and therefore
\[
	\int_{\{r|v|\leq |x|/2\}}
	\langle s,x-rv\rangle^{-2-\kappa_0}
	\frac{dv}{\langle v\rangle^{2+\kappa_0}}
	\lesssim \langle x\rangle^{-2-\kappa_0}.
\]
On the second region, we have
$\langle v\rangle^{-2-\kappa_0}\lesssim (r/|x|)^{2+\kappa_0}$. Using the change of
variables $w=x-rv$, we obtain
\begin{align*}
	\int_{\{r|v|\geq |x|/2\}}\langle s,x-rv\rangle^{-2-\kappa_0}
	\frac{dv}{\langle v\rangle^{2+\kappa_0}}&\lesssim
	r^{-2}\frac{r^{2+\kappa_0}}{|x|^{2+\kappa_0}}\int_{\mathbb R^2}\langle s,w\rangle^{-2-\kappa_0}\,dw\lesssim
	\langle x\rangle^{-2-\kappa_0}
	r^{\kappa_0}\langle s\rangle^{-\kappa_0}.
\end{align*}
Consequently,
\[
	\int_{\mathbb R^2}
	\langle s,x-(t-s)v\rangle^{-2-\kappa_0}
	\frac{dv}{\langle v\rangle^{2+\kappa_0}}
	\lesssim
	\langle x\rangle^{-2-\kappa_0}
	\left(1+(t-s)^{\kappa_0}\langle s\rangle^{-\kappa_0}\right).
\]
and it follows that
\begin{align*}
	\mathbf S_\sigma(t,x)
	&\lesssim
	\langle x\rangle^{-2-\kappa_0}
	\langle t\rangle^{\kappa_0}
	\int_0^t\langle s\rangle^{\sigma-\kappa_0}\,ds\lesssim
	\langle t\rangle^{\kappa_0}
	\langle t\rangle^{(\sigma-\kappa_0+1)_+}
	\langle x\rangle^{-2-\kappa_0}.
\end{align*}
This proves the desired estimates \eqref{z143a} in Case 2.

To estimate $\mathbf T_\sigma$, we use \eqref{z201p}, with $t-s$
in place of $t$, thus
\[
	\int_{\mathbb R^2}
	\langle s,x-(t-s)v\rangle^{-2}
	\frac{dv}{\langle v\rangle^{2+2\kappa_0}}
	\lesssim
	\log(2+t)\langle x\rangle^{-2}.
\]
Therefore
\begin{align*}
	\mathbf T_\sigma(t,x)
	&\lesssim\log(2+t)\langle x\rangle^{-2}
	\int_0^t\langle s\rangle^\sigma\,ds\lesssim
	\log(2+t)
	\langle t\rangle^{(\sigma+1)_+}
	\langle x\rangle^{-2}.
\end{align*}
This gives \eqref{z143b} in Case 2 and completes the proof of the lemma.
	\end{proof}

	\section{Estimates for the characteristics}\label{section4}
	In this section we assume that $(f,E)$ is a solution of the system \eqref{eq2} on the interval $[0,T]$, satisfying the bounds $M_T\leq\epsilon\ll 1$ as in Proposition \ref{propoglobal3d}, and we establish pointwise bounds for the deviation of the
characteristic flow. Recall that, with the notation of Section~2,
	the correction term $Y_{s,t}$ is defined by
	\[
	Y_{s,t}(x,v)=X_{s,t}(x+tv,v)-x-sv,
	\]
        for any $0\le s\le t\le T$ and $x,v\in\R^2$, and satisfies the integral equation
	\begin{equation}\label{Yinteq}
	Y_{s,t}(x,v)=\int_s^t(\tau-s)\,E(\tau,x+\tau v+Y_{\tau,t}(x,v))\,d\tau.
	\end{equation}
	We also recall that $W_{s,t}=\partial_s Y_{s,t}$. For simplicity of notation we define the weights
	\begin{equation}\label{aweights}
	\begin{split}
		\mathfrak a_0(s,v)&:=\langle v\rangle^{-1}\langle s\rangle^{-\frac12-\kappa_0}
		+\langle s,|v|^{\frac32}\rangle^{-\frac23+50\kappa_0},\\
		\mathfrak a_1(s,v)&:=\langle v\rangle^{-1}\langle s\rangle^{-\frac76-\kappa_0}+\langle s,|v|^{\frac32}\rangle^{-\frac23}
		\langle s\rangle^{-\frac23+\kappa_0},\\
		\mathfrak a_2(s,v)&:=\langle s^{\frac12+2\kappa_0},v\rangle^{-1}
		\langle s\rangle^{-\frac43+\kappa_0}.
	\end{split}
	\end{equation}
	and notice that these weights are nonincreasing in both $s$ and $|v|$.
	Lemma~\ref{z216} applies also to the rays $x+\tau v$, by replacing $v$ with $-v$.
	\medskip
	\subsection{Pointwise bounds for the characteristics} We prove now several bounds on the function $Y$ and its derivatives.

	\begin{lemma} [Estimates for the characteristics]\label{keyle}  
	For all $0\le s\le t\le T$ we have
	\begin{align}\label{n13}
		&	|Y_{s,t}|+\langle s\rangle	|\partial_sY_{s,t}|\lesssim   M_T\mathfrak{a}_0(s,v),\\
		&\label{n14}	|\nabla_xY_{s,t}|+\langle s\rangle	|\nabla_x\partial_sY_{s,t}|\lesssim  M_T\mathfrak{a}_1(s,v),\\
		&\label{xf}|\nabla_vY_{s,t}|+\langle s\rangle	|\nabla_v\partial_sY_{s,t}|\lesssim  M_T\big(\langle v\rangle^{-1}\langle s\rangle^{-\frac{1}{6}-\kappa_0}+\langle s,|v|^{\frac{3}{2}}\rangle^{-\frac{1}{3}+\kappa_0}\big),\\
		&\label{n15}	|\nabla_x\nabla_vY_{s,t}|+\langle s\rangle 	|\nabla_x\nabla_v\partial_sY_{s,t}|+\langle s\rangle(	|\nabla_x^2Y_{s,t}|+\langle s\rangle	|\nabla_x^2\partial_sY_{s,t}|)\lesssim  M_T\langle s\rangle\mathfrak{a}_2(s,v),\\
		&\label{n17}|\nabla_v^2Y_{s,t}|\lesssim M_T	\langle t\rangle^{\frac{2}{3}+\kappa_0}	\langle t^{\frac{1}{2}+2\kappa_0},v\rangle^{-1}
		,\\&\label{n17b}
		|\nabla_v^2\partial_sY_{s,t}|\lesssim	M_T\langle t\rangle^{\frac{1}{6}-\kappa_0}	\langle s^{\frac12+2\kappa_0},v\rangle^{-1}
		\langle s\rangle^{-\frac12+2\kappa_0}.
	\end{align}
	Here, $Y_{s,t}=Y_{s,t}(x,v)$. In particular, the following
	rougher bounds also hold:
	\begin{align}\label{n13'}
		&	|Y_{s,t}|+\langle s\rangle	|\partial_sY_{s,t}|\lesssim   M_T
		\langle v\rangle^{-\frac{1}{10}}\langle s\rangle^{-\frac12-\kappa_0}	,\\&\label{n14'}	|\nabla_xY_{s,t}|+\langle s\rangle	|\nabla_x\partial_sY_{s,t}|\lesssim  M_T\langle v\rangle^{-\frac{1}{10}}\langle s\rangle^{-\frac{7}{6}-\kappa_0}
		,\\&\label{xf'}
		|\nabla_vY_{s,t}|+\langle s\rangle	|\nabla_v\partial_sY_{s,t}|\lesssim  M_T\langle v\rangle^{-\frac{1}{10}}\langle s\rangle^{-\frac{1}{6}-\kappa_0},\\
		&\label{n15'}	|\nabla_x\nabla_vY_{s,t}|+\langle s\rangle 	|\nabla_x\nabla_v\partial_sY_{s,t}|+\langle s\rangle(	|\nabla_x^2Y_{s,t}|+\langle s\rangle	|\nabla_x^2\partial_sY_{s,t}|)\lesssim  M_T\langle v\rangle^{-2\kappa_0}  \langle s\rangle^{-\frac{5}{6}+2\kappa_0}.
	\end{align}
\end{lemma}

\begin{proof} We divide the proof into three steps.
	\medskip
	
	\textbf{Step 1: oscillatory integral bounds.} For integers $n_1,n_2\ge0$ with $n_1+n_2\le2$, define
	\begin{align*}
		P_{n_1,n_2}	:=	\Big|\int_s^t(\tau-s) \tau^{n_1}\nabla_{x}^{n_1+n_2}E^{osc}(\tau,x+\tau v)d\tau\Big|+\langle s\rangle\Big|\int_s^t \tau^{n_1}\nabla_{x}^{n_1+n_2}E^{osc}(\tau,x+\tau v)\,d\tau\Big|.
	\end{align*}
	Using \eqref{de1} and integrating by parts on $[s_0,t]$ for any $s_0\in[s,t]$, we obtain
	\begin{equation}\label{Pbounds1}
		\begin{split}
			P_{n_1,n_2}	&\lesssim \sum_{\ast\in\{\cos,\sin\}}\Big\{\int_s^{s_0} \langle \tau\rangle^{n_1+1}|\nabla_{x}^{n_1+n_2}E^{\ast}(\tau,x+\tau v)|\,d\tau\\
			&+\mathbf 1_{s_0<t}\sum_{\tau\in\{s_0,t\}} \langle \tau\rangle^{n_1+1}|\nabla_{x}^{n_1+n_2} E^{\ast}(\tau,x+\tau v)|+\int_{s_0}^{t}\langle \tau\rangle^{n_1}\Big[\langle \tau\rangle|\nabla_{x}^{n_1+n_2}\partial_\tau E^{\ast}(\tau,x+\tau v)|\\
			&+|\nabla_{x}^{n_1+n_2} E^{\ast}(\tau,x+\tau v)|+\langle v\rangle\langle \tau\rangle|\nabla_{x}^{n_1+n_2+1} E^{\ast}(\tau,x+\tau v)|\Big]\,d\tau\Big\}.
		\end{split}
	\end{equation}
	
	Using \eqref{bo1} and \eqref{a23}, for any $t\in[0,T]$, $x\in\R^2$, and $\ast\in\{\cos,\sin\}$ we obtain
	\begin{equation}\label{Ebounds1}
		\begin{split}
			&|E^\ast(t,x)|\lesssim M_T\langle t,x\rangle^{-2}\langle t\rangle^{50\kappa_0},\\
			&|\nabla_x^n E^\ast(t,x)|\lesssim M_T\langle t,x\rangle^{-2-\kappa_0}\langle t\rangle^{-2n/3+2\kappa_0},\qquad n\in\{1,2\},\\
			&|\partial_t\nabla_x^n E^\ast(t,x)|\lesssim M_T\langle t,x\rangle^{-2-\kappa_0}\langle t\rangle^{-1/2-2n/3},\qquad n\in\{0,1\},\\
			&|\nabla_x^3 E^\ast(t,x)|+|\partial_t\nabla_x^2 E^\ast(t,x)|\lesssim M_T\langle t,x\rangle^{-2-\kappa_0}\langle t\rangle^{-11/6}.
		\end{split}
	\end{equation}

	{\bf{Case $n_1+n_2=0$.}} Using \eqref{Pbounds1} and \eqref{Ebounds1}, we have
	\begin{align*}
		P_{0,0}&\lesssim
		M_T\int_s^{s_0}\langle \tau\rangle^{1+50\kappa_0}
		\langle \tau,x+\tau v\rangle^{-2}\,d\tau+M_T\mathbf 1_{s_0<t}\sum_{r\in\{s_0,t\}}
		\langle r\rangle^{1+50\kappa_0}
		\langle r,x+rv\rangle^{-2}\\
		&+M_T\int_{s_0}^{t}\Big[
		\langle \tau\rangle^{\frac12}
		\langle \tau,x+\tau v\rangle^{-2-\kappa_0}+
		\langle \tau\rangle^{50\kappa_0}
		\langle \tau,x+\tau v\rangle^{-2}+
		\langle v\rangle
		\langle \tau\rangle^{\frac13+\kappa_0}
		\langle \tau,x+\tau v\rangle^{-2}
		\Big]\,d\tau .
	\end{align*}
	Applying \eqref{z210}--\eqref{z212}, we obtain
	\begin{align}
		P_{0,0}&\lesssim M_T\big(\mathbf 1_{s_0>s}\langle v\rangle^{-1}
		\langle s_0\rangle^{50\kappa_0}+\mathbf 1_{s_0<t}
		\langle v\rangle^{-1}
		\langle s_0\rangle^{-\frac12-\kappa_0}+\mathbf 1_{s_0<t}
		\langle s_0\rangle^{-\frac23+\kappa_0}\big).
		\label{P00-pre}
	\end{align}
	We choose $s_0\in[s,t]$ as 
	\begin{equation}\label{choice-s0-P00}
		s_0:=\begin{cases}
			s\qquad&\text{ if }\,\,|v|^{3/2}\leq s,\\
			t\qquad&\text{ if }\,\,|v|^{3/2}\geq t,\\
			|v|^{3/2}&\text{ if }\,\,s\leq |v|^{3/2}\leq t.
		\end{cases}		
	\end{equation}
	With this choice,
	\begin{align*}
		\mathbf 1_{s_0>s}
		\langle v\rangle^{-1}
		\langle s_0\rangle^{50\kappa_0}
		&\lesssim
		\langle s,\langle v\rangle^{3/2}\rangle^{-\frac23+50\kappa_0},\\
		\mathbf 1_{s_0<t}\langle s_0\rangle^{-\frac23+\kappa_0}
		&\lesssim
		\langle s,\langle v\rangle^{3/2}\rangle^{-\frac23+\kappa_0}\lesssim
		\langle s,\langle v\rangle^{3/2}\rangle^{-\frac23+50\kappa_0},\\
		\mathbf 1_{s_0<t}
		\langle v\rangle^{-1}
		\langle s_0\rangle^{-\frac12-\kappa_0}
		&\lesssim
		\langle v\rangle^{-1}
		\langle s\rangle^{-\frac12-\kappa_0}.
	\end{align*}
	Therefore
	\begin{equation}\label{q60}
		P_{0,0}\lesssim M_T\big(\langle v\rangle^{-1}\langle s\rangle^{-\frac12-\kappa_0}+
		\langle s,|v|^{3/2}\rangle^{-\frac23+50\kappa_0}\big).
	\end{equation}

	{\bf{Case $n_1+n_2=1$.}} In this case we prove that
	\begin{align}\nonumber
		P_{n_1,n_2}&\lesssim \mathbf{1}_{n_1=0}M_T\big(\langle v\rangle^{-1}\langle s\rangle^{-\frac{7}{6}-\kappa_0}+\langle s,|v|^{\frac{3}{2}}\rangle^{-\frac{2}{3}} 	\langle s\rangle^{-\frac{2}{3}+\kappa_0}\big)
		\\&+\mathbf{1}_{n_1=1}M_T\big(\langle v\rangle^{-1}\langle s\rangle^{-\frac{1}{6}-\kappa_0}+\langle s,|v|^{\frac{3}{2}}\rangle^{-\frac{1}{3}+\kappa_0}\big).\label{q61}
	\end{align}	
	Indeed, from \eqref{Pbounds1} and \eqref{Ebounds1}, we have
	\begin{align}\nonumber
		P_{n_1,n_2}	&\lesssim M_T\int_s^{s_0} \langle \tau\rangle^{\frac{1}{3}+n_1+2\kappa_0} \Phi(\tau,x+\tau v)d\tau+\mathbf 1_{s_0<t}\sum_{\tau\in\{s_0,t\}}M_T \langle \tau\rangle^{\frac{1}{3}+n_1+2\kappa_0} \Phi(\tau,x+\tau v)\\&
		+M_T\int_{s_0}^{t}(\langle \tau\rangle^{n_1-\frac{1}{6}}+\langle v\rangle\langle \tau\rangle^{n_1-\frac{1}{3}+2\kappa_0})\Phi(\tau,x+\tau v)d\tau.	\label{P-one-derivative-start}
	\end{align}
	
	We apply now \eqref{z210}--\eqref{z212}. First, for the integral over $[s,s_0]$, Lemma~\ref{z216} gives
	\begin{align}
		&\int_s^{s_0}
		\langle \tau\rangle^{\frac13+n_1+2\kappa_0}
		\Phi(\tau,x+\tau v)\,d\tau\lesssim\mathbf 1_{s_0>s}\langle v\rangle^{-1}
		\big(\mathbf 1_{n_1=0}\langle s\rangle^{-\frac23+\kappa_0}+\mathbf 1_{n_1=1}\langle s_0\rangle^{\frac13+\kappa_0}\big)
		\label{q61-I0}.
	\end{align}
	Second, for the integrals over $[s_0,t]$, we obtain
	\begin{equation*}
		\begin{split}
			\int_{s_0}^{t}
			\langle \tau\rangle^{n_1-\frac16}
			\Phi(\tau,x+\tau v)\,d\tau
			&\lesssim\mathbf{1}_{s_0<t}\langle v\rangle^{-1}\langle s_0\rangle^{n_1-\frac76-\kappa_0},\\
			\langle v\rangle
			\int_{s_0}^{t}
			\langle \tau\rangle^{n_1-\frac13+2\kappa_0}
			\Phi(\tau,x+\tau v)\,d\tau
			&\lesssim\mathbf{1}_{s_0<t}\langle s_0\rangle^{n_1-\frac43+\kappa_0}.
		\end{split}
	\end{equation*}
	Finally, $\langle r\rangle^{\frac13+n_1+2\kappa_0}\Phi(r,x+rv)\lesssim\langle r\rangle^{n_1-\frac53+\kappa_0}$ for $r\in\{s_0,t\}$. Thus, 
	\begin{align}\nonumber
		P_{n_1,n_2}	&\lesssim M_T	\mathbf 1_{s_0>s}\langle v\rangle^{-1}\big(\mathbf 1_{n_1=0}
		\langle s\rangle^{-\frac23+\kappa_0}+\mathbf 1_{n_1=1}
		\langle s_0\rangle^{\frac13+\kappa_0}\big)\\&
		+M_T\mathbf{1}_{s_0<t}\big(\langle v\rangle^{-1}
		\langle s_0\rangle^{n_1-\frac76-\kappa_0}+	\langle s_0\rangle^{n_1-\frac43+\kappa_0}\big).	\label{P-one-derivative-start'}
	\end{align}
	We now choose  $s_0$ as in \eqref{choice-s0-P00} and the desired bounds \eqref{q61} follow from \eqref{P-one-derivative-start'}.
	
	{\bf{Case $n_1+n_2=2$.}} In this case we prove that
	\begin{align}
		P_{n_1,n_2}&\lesssim\mathbf{1}_{n_1\leq 1}M_T
		\langle s^{\frac12+2\kappa_0},v\rangle^{-1}
		\langle s\rangle^{-\frac43+n_1+\kappa_0}+\mathbf{1}_{n_1=2}M_T
		\langle t\rangle^{\frac23+\kappa_0}
		\langle v,t^{\frac12+2\kappa_0}\rangle^{-1}
		\label{q62}.
	\end{align}
	From \eqref{Pbounds1} and \eqref{Ebounds1}, we have, for $n_1+n_2=2$,
	\begin{equation}\label{P-two-derivative-start}
		\begin{split}
			&P_{n_1,n_2}\lesssim
			M_T\int_s^{s_0}
			\langle \tau\rangle^{-\frac13+n_1+2\kappa_0}\Phi(\tau,x+\tau v)\,d\tau\\
			&+\mathbf{1}_{s_0<t}M_T\sum_{r\in\{s_0,t\}}
			\langle r\rangle^{-\frac13+n_1+2\kappa_0}\Phi(r,x+rv)+M_T\int_{s_0}^{t}
			\langle v\rangle
			\langle \tau\rangle^{n_1-\frac56}
			\Phi(\tau,x+\tau v)\,d\tau.
		\end{split}	
	\end{equation}
	
	Assume first that $n_1\le1$. Applying \eqref{z210}--\eqref{z212} gives
	\begin{align}
		P_{n_1,n_2}
		&\lesssim
		M_T\mathbf 1_{s_0>s}
		\langle v\rangle^{-1}
		\langle s\rangle^{n_1-\frac43+\kappa_0}+
		M_T\mathbf 1_{s_0<t}
		\langle s_0\rangle^{n_1-\frac{11}{6}-\kappa_0}.
		\label{q62-reduced-low}
	\end{align}
	Choose $s_0\in[s,t]$ such that
	\begin{equation*}
		s_0=t
		\quad\text{if}\quad
		\langle v\rangle\ge \langle s\rangle^{\frac12+2\kappa_0},
		\qquad
		s_0=s
		\quad\text{if}\quad
		\langle v\rangle< \langle s\rangle^{\frac12+2\kappa_0}.
	\end{equation*}
	The bounds \eqref{q62} follow in this case from \eqref{q62-reduced-low}.

	It remains to consider $n_1=2$, $n_2=0$. Applying
	\eqref{z210} to \eqref{P-two-derivative-start}, we obtain
	\begin{align*}
		P_{2,0}&\lesssim M_T\mathbf 1_{s_0>s}\langle v\rangle^{-1}
		\langle s_0\rangle^{\frac23+\kappa_0}+M_T\mathbf 1_{s_0<t}
		\big(\langle t\rangle^{\frac16-\kappa_0}+\langle s_0\rangle^{-\frac13+\kappa_0}\big).
	\end{align*}
	We choose $s_0\in[s,t]$ such that
	\[
	s_0=t
	\quad\text{if}\quad
	\langle v\rangle\geq
	\langle t\rangle^{\frac12+2\kappa_0},
	\qquad
	s_0=s
	\quad\text{if}\quad
	\langle v\rangle<
	\langle t\rangle^{\frac12+2\kappa_0},
	\]
	and the desired bounds \eqref{q62} follow in this case as well.
	
	For later use we notice that our argument also proves that
	\begin{align}
		\left|\int_s^t\tau^2\nabla_x^2E^{osc}(\tau,x+\tau v)\,d\tau\right|
		&\lesssim M_T\langle s^{\frac12+2\kappa_0},v\rangle^{-1}\langle s\rangle^{-\frac13+\kappa_0}.
		\label{q62p}
	\end{align}
	Indeed, $\tau^2=(\tau-s)\tau+s\tau$, so the term in the left-hand side of \eqref{q62p} is controlled by $P_{1,1}$.
	\medskip
	
	\textbf{Step 2: bootstrap control.}
	Assume that for some $T'\le T$,
	\begin{equation}\label{n12}
		\begin{split}
			&A(T'):=	\sup_{0\leq s\leq t\leq T'}\big\{\langle s\rangle^{\frac{1}{2}+\kappa_0}\|(Y_{s,t},\langle s\rangle\partial_sY_{s,t})\|_{L^\infty_{x,v}}+\langle s\rangle^{\frac{7}{6}+\kappa_0}\|(\nabla_xY_{s,t},\langle s\rangle\partial_s\nabla_xY_{s,t})\|_{L^\infty_{x,v}}\\
			&+\langle s\rangle^{\frac{1}{6}+\kappa_0}\|(\nabla_vY_{s,t},\langle s\rangle\partial_s\nabla_vY_{s,t})\|_{L^\infty_{x,v}}	+\langle s\rangle^{\frac{5}{6}+\kappa_0}\|(\nabla_x\nabla_vY_{s,t},\langle s\rangle\partial_s\nabla_x\nabla_vY_{s,t})\|_{L^\infty_{x,v}}\\
			&+\langle s\rangle^{\frac{11}{6}+\kappa_0}\|(\nabla_x^2Y_{s,t},\langle s\rangle\partial_s\nabla_x^2Y_{s,t})\|_{L^\infty_{x,v}}+\langle t\rangle^{-\frac{1}{6}+\kappa_0}\|(\nabla_v^2Y_{s,t},\langle s\rangle\partial_s\nabla_v^2Y_{s,t})\|_{L^\infty_{x,v}}\big\}\leq 1.
		\end{split}
	\end{equation}
	
	For simplicity of notation we write $\nabla_x^n\widetilde{E}(\tau):=(\nabla_x^nE)(\tau,x+\tau v+Y_{\tau,t})$, $\nabla_x^n\widetilde{E}^{reg}(\tau):=(\nabla_x^nE^{reg})(\tau,x+\tau v+Y_{\tau,t})$, $\nabla_x^n\widetilde{E}^{osc}(\tau):=(\nabla_x^nE^{osc})(\tau,x+\tau v+Y_{\tau,t})$. Differentiating the characteristic equation \eqref{Yinteq} yields the integral identities
	\begin{equation}\label{nabla1Y}
		\begin{split}
			\nabla_xY_{s,t}&=\int_s^t(\tau-s)(I+\nabla_{x}Y_{\tau,t})\nabla \widetilde{E}(\tau)d\tau,\\
			\nabla_vY_{s,t}&=\int_s^t(\tau-s)(\tau I+\nabla_{v}Y_{\tau,t})\nabla \widetilde{E}(\tau)d\tau,
		\end{split}
	\end{equation}
	\begin{equation}\label{nabla2Y}
		\begin{split}
			\nabla_x^2Y_{s,t}&=\int_s^t(\tau-s)\nabla_{x}^2Y_{\tau,t}\nabla \widetilde{E}(\tau)d\tau+\int_s^t(\tau-s)(I+\nabla_{x}Y_{\tau,t})^{\otimes 2}:\nabla^2 \widetilde{E}(\tau)d\tau,\\
			\nabla_v^2Y_{s,t}&=\int_s^t(\tau-s)\nabla_{v}^2Y_{\tau,t}\nabla \widetilde{E}(\tau)\,d\tau+\int_s^t(\tau-s)(\tau I+\nabla_{v}Y_{\tau,t})^{\otimes 2}:\nabla^2 \widetilde{E}(\tau)d\tau,\\
			\nabla_x\nabla_{v}Y_{s,t}&=\int_s^t(\tau-s)\nabla_{x}\nabla_vY_{\tau,t}\nabla \widetilde{E}(\tau)d\tau\\
			&+\int_s^t(\tau-s)(I+\nabla_{x}Y_{\tau,t})\otimes (\tau I+\nabla_{v}Y_{\tau,t}):\nabla^2 \widetilde{E}(\tau)d\tau.
		\end{split}
	\end{equation}
	
	Using the bootstrap bounds \eqref{n12} together with \eqref{bo1}, we reduce each quantity to an oscillatory part controlled by $P_{n_1,n_2}$ plus a regular contribution estimated by \eqref{z210}--\eqref{z212}. Notice that $|Y_{\tau,t}|\lesssim1$, due to \eqref{n12}. Therefore, using \eqref{bo1}, \eqref{a23}, and \eqref{Ebounds1}, we have
	\begin{equation}\label{Ecomp}
		\begin{split}
			|\nabla_x^q&E^{osc}(\tau,x+\tau v+Y_{\tau,t}(x,v))
			-\nabla_x^qE^{osc}(\tau,x+\tau v)|\\
			&+|\nabla_x^qE^{reg}(\tau,x+\tau v+Y_{\tau,t}(x,v))|\lesssim M_T\langle\tau\rangle^{-\frac12-\frac{2q}{3}}\Phi(\tau,x+\tau v),
		\end{split}
	\end{equation}
	for $q=0,1,2$. 
	
	We would like to use the identities \eqref{Yinteq} and \eqref{nabla1Y}--\eqref{nabla2Y} to close the bootstrap argument. The terms containing only the free factors $I$ and $\tau I$ can be bounded by the corresponding $P_{n_1,n_2}$ (and by
	\eqref{q62p} in the last
	estimate).  Every other term is covered by the displayed bound or contains
	a derivative of $Y_{\tau,t}$ and is bounded by \eqref{n12} and
	\eqref{Ebounds1}; the
	$\nabla_v^2Y_{\tau,t}$ terms are treated separately in the last two bounds. 
	
	More precisely, starting from \eqref{Yinteq} and using \eqref{Ecomp} we can estimate
	\begin{equation}\label{Ysu1}
		\begin{split}
			|Y_{s,t}|+\langle s\rangle	|\partial_sY_{s,t}|&\lesssim \left|\int_s^t(\tau-s) \widetilde{E}^{osc}(\tau)d\tau\right|+\langle s\rangle\left|\int_s^t \widetilde{E}^{osc}(\tau)d\tau\right|+\int_s^t\langle \tau\rangle| \widetilde{E}^{reg}(\tau)|\,d\tau\\&
			\lesssim 	P_{0,0}+M_T\int_s^t\langle \tau\rangle^{\frac{1}{2}}\Phi(\tau,x+\tau v)d\tau.
		\end{split}
	\end{equation}
	
	Similarly, using \eqref{nabla1Y}, \eqref{Ecomp}, and \eqref{Ebounds1} we estimate
	\begin{equation}\label{Ysu2}
		\begin{split}
			|\nabla_xY_{s,t}|+\langle s\rangle|\nabla_x\partial_sY_{s,t}|&\lesssim \Big|\int_s^t(\tau-s)\nabla \widetilde{E}^{osc}(\tau)d\tau\Big|+\langle s\rangle \Big|\int_s^t\nabla \widetilde{E}^{osc}(\tau)d\tau\Big|\\
			&+\int_s^t\langle \tau\rangle|\nabla \widetilde{E}^{reg}(\tau)|\,d\tau+\int_s^t\langle \tau\rangle^{1-\frac{7}{6}-\kappa_0}|\nabla \widetilde{E}(\tau)|\,d\tau\\
			&\lesssim	P_{0,1}+M_T\int_s^t\langle \tau\rangle^{-\frac{1}{6}}\Phi(\tau,x+\tau v)d\tau,
		\end{split}
	\end{equation}
	and
	\begin{equation}\label{Ysu3}
		\begin{split}
			|\nabla_vY_{s,t}|+\langle s\rangle|\nabla_v\partial_sY_{s,t}|&\lesssim \Big|\int_s^t(\tau-s)\tau\nabla \widetilde{E}^{osc}(\tau)d\tau\Big|+\langle s\rangle \Big|\int_s^t \tau\nabla \widetilde{E}^{osc}(\tau)d\tau\Big|\\
			&+\int_s^t\langle \tau\rangle^2|\nabla \widetilde{E}^{reg}(\tau)|\,d\tau+\int_s^t\langle \tau\rangle^{1-\frac{1}{6}-\kappa_0}|\nabla \widetilde{E}(\tau)|\,d\tau\\
			&\lesssim P_{1,0}+M_T\int_s^t\langle \tau\rangle^{\frac{5}{6}}\Phi(\tau,x+\tau v)d\tau.
		\end{split}
	\end{equation}
	
	Similarly, using \eqref{nabla2Y},
	\eqref{Ecomp}, and \eqref{Ebounds1}, we obtain
	\begin{equation}\label{Ysu4}
		\begin{split}
			&|\nabla_x^2Y_{s,t}|+\langle s\rangle|\nabla_x^2\partial_sY_{s,t}|
			\lesssim\Big|\int_s^t(\tau-s)\nabla^2\widetilde E^{osc}(\tau)\,d\tau\Big|
			+\langle s\rangle\Big|\int_s^t\nabla^2\widetilde E^{osc}(\tau)\,d\tau\Big|\\
			&+\int_s^t\langle\tau\rangle|\nabla^2\widetilde E^{reg}(\tau)|\,d\tau+\int_s^t\langle\tau\rangle^{1-\frac{11}{6}-\kappa_0}
			|\nabla\widetilde E(\tau)|\,d\tau+\int_s^t\langle\tau\rangle^{1-\frac{7}{6}-\kappa_0}|\nabla^2\widetilde E(\tau)|\,d\tau\\
			&\lesssim
			P_{0,2}
			+M_T\int_s^t
			\langle\tau\rangle^{-\frac56}
			\Phi(\tau,x+\tau v)\,d\tau,
		\end{split}
	\end{equation}
	and
	\begin{equation}\label{Ysu5}
		\begin{split}
			&|\nabla_x\nabla_vY_{s,t}|+\langle s\rangle|\nabla_x\nabla_v\partial_sY_{s,t}|
			\lesssim\Big|\int_s^t(\tau-s)\tau
			\nabla^2\widetilde E^{osc}(\tau)\,d\tau\Big|+\langle s\rangle
			\Big|\int_s^t\tau\nabla^2\widetilde E^{osc}(\tau)\,d\tau\Big|\\
			&+\int_s^t\langle\tau\rangle^2
			|\nabla^2\widetilde E^{reg}(\tau)|\,d\tau+\int_s^t\langle\tau\rangle^{1-\frac56-\kappa_0}
			|\nabla\widetilde E(\tau)|\,d\tau+\int_s^t
			\langle\tau\rangle^{1-\frac16-\kappa_0}
			|\nabla^2\widetilde E(\tau)|\,d\tau\\
			&\lesssim P_{1,1}+M_T\int_s^t\langle\tau\rangle^{\frac16}\Phi(\tau,x+\tau v)\,d\tau.
		\end{split}
	\end{equation}
	
	To bound the two $v$-derivatives, we notice that $|\nabla_v^2Y_{\tau,t}|\lesssim \langle t\rangle^{\frac16-\kappa_0}$ (due to \eqref{n12}) thus
	\begin{equation}\label{Ysu6}
		\begin{split}
			&|\nabla_v^2Y_{s,t}|\lesssim\Big|\int_s^t(\tau-s)\tau^2\nabla^2\widetilde E^{osc}(\tau)\,d\tau\Big|
			+\int_s^t\langle\tau\rangle^3|\nabla^2\widetilde E^{reg}(\tau)|\,d\tau\\
			&+\int_s^t\langle\tau\rangle^{2-\frac16-\kappa_0}|\nabla^2\widetilde E(\tau)|\,d\tau+
			\langle t\rangle^{\frac16-\kappa_0}\int_s^t\langle\tau\rangle|\nabla\widetilde E(\tau)|\,d\tau\\
			&\lesssim
			P_{2,0}+M_T\int_s^t
			\langle\tau\rangle^{\frac76}
			\Phi(\tau,x+\tau v)\,d\tau+	M_T\langle t\rangle^{\frac16-\kappa_0}
			\int_s^t
			\langle\tau\rangle^{\frac13+2\kappa_0}
			\Phi(\tau,x+\tau v)\,d\tau
		\end{split}
	\end{equation}
	and
	\begin{equation}\label{Ysu7}
		\begin{split}
			&|\nabla_v^2\partial_sY_{s,t}|\lesssim\Big|\int_s^t\tau^2\nabla^2\widetilde E^{osc}(\tau)\,d\tau\Big|
			+\int_s^t\langle\tau\rangle^2|\nabla^2\widetilde E^{reg}(\tau)|\,d\tau\\
			&+\int_s^t\langle\tau\rangle^{1-\frac16-\kappa_0}|\nabla^2\widetilde E(\tau)|\,d\tau+
			\langle t\rangle^{\frac16-\kappa_0}
			\int_s^t|\nabla\widetilde E(\tau)|\,d\tau\lesssim\Big|\int_s^t\tau^2\nabla^2E^{osc}(\tau,x+\tau v)\,d\tau\Big|\\
			&+M_T\int_s^t\langle\tau\rangle^{\frac16}\Phi(\tau,x+\tau v)\,d\tau+M_T\langle t\rangle^{\frac16-\kappa_0}\int_s^t\langle\tau\rangle^{-\frac23+2\kappa_0}\Phi(\tau,x+\tau v)\,d\tau.
		\end{split}
	\end{equation}

	\textbf{Step 3: conclusion.}
	For any $\beta\neq1+\kappa_0$, Lemma~\ref{z216} (ii) gives
	\begin{equation}\label{IntBouRep}
		\int_s^t
		\langle\tau\rangle^\beta
		\Phi(\tau,x+\tau v)\,d\tau
		\lesssim
		\langle v\rangle^{-1}
		\begin{cases}
			\langle s\rangle^{\beta-1-\kappa_0},
			&\beta<1+\kappa_0,\\
			\langle t\rangle^{\beta-1-\kappa_0},
			&\beta>1+\kappa_0.
		\end{cases}
	\end{equation}
	We apply this with $\beta=1/2, \beta=-1/6, \beta=5/6, \beta=-5/6, \beta=1/6$, and use the bounds \eqref{Ysu1}--\eqref{Ysu5}, and the bounds \eqref{q60}, \eqref{q61}, and \eqref{q62} (with $n_1\leq 1$). The desired bounds \eqref{n13}, \eqref{n14}, \eqref{xf}, and \eqref{n15} follow. Moreover, the bounds \eqref{Ysu6}, \eqref{q62} (with $n_1=2$), and \eqref{IntBouRep} show that
	\begin{align*}
		|\nabla_v^2Y_{s,t}|&\lesssim
		M_T\langle t\rangle^{\frac23+\kappa_0}
		\langle v,t^{\frac12+2\kappa_0}\rangle^{-1}+M_T\langle v\rangle^{-1}
		\langle t\rangle^{\frac16-\kappa_0}\lesssim
		M_T\langle t\rangle^{\frac23+\kappa_0}
		\langle v,t^{\frac12+2\kappa_0}\rangle^{-1},
	\end{align*}
	which proves \eqref{n17}. Finally, the bounds \eqref{Ysu7}, \eqref{q62p}, and \eqref{IntBouRep} show that
	\begin{align*}
		|\nabla_v^2\partial_sY_{s,t}|&\lesssim M_T\langle s^{\frac12+2\kappa_0},v\rangle^{-1}
		\langle s\rangle^{-\frac13+\kappa_0}+M_T\langle v\rangle^{-1}
		\langle s\rangle^{-\frac56-\kappa_0}+ M_T\langle t\rangle^{\frac16-\kappa_0}\langle v\rangle^{-1}\langle s\rangle^{-\frac53+\kappa_0}\\
		&\lesssim
		M_T\langle t\rangle^{\frac16-\kappa_0}
		\langle s^{\frac12+2\kappa_0},v\rangle^{-1}
		\langle s\rangle^{-\frac12+2\kappa_0},
	\end{align*}
	which proves \eqref{n17b}.
	
	It remains to close the bootstrap. The bounds \eqref{n13}--\eqref{n17b} and the definition \eqref{n12} show that \(A(T')\leq C M_T\) whenever \(A(T')\leq1\). Since $A(0)=0$ and $A(T')$ is continuous it follows that $A(T)\lesssim M_T$ provided that $M_T$ is sufficiently small. This completes the proof.
\end{proof}
	
\subsection{Additional estimates} For later use we record a few more estimates that follow from our arguments so far. Since $\partial_s^2Y_{s,t}(x,v)=E(s,x+sv+Y_{s,t}(x,v))$, we have
		\begin{align}
			\label{x200}
			|\partial_s^2Y_{s,t}(x,v)|
			&\lesssim M_T\langle s,x+sv\rangle^{-2+50\kappa_0},\\
			|\partial_s^2\nabla_vY_{s,t}(x,v)|
			&\lesssim M_T\langle s,x+sv\rangle^{-\frac{5}{3}+\kappa_0},\notag\\
			\label{x202}
			|\partial_s^2\nabla_v^2Y_{s,t}(x,v)|
			&\lesssim M_T\big(\langle s\rangle^{\frac23}+\langle s\rangle^{\frac13}\langle t\rangle^{\frac16-\kappa_0}\big)\langle s,x+sv\rangle^{-2+\kappa_0}.
		\end{align}
Moreover, it follows from \eqref{Ysu7} and integration by parts that
		\begin{align}\label{x203}
			&|\nabla_v^2W_{s,t}(x,v)|\lesssim M_T\sum_{\tau\in\{s,t\}}\Big(\langle\tau\rangle^{\frac23+2\kappa_0}
			+\langle\tau\rangle^{-\frac23+2\kappa_0}\langle t\rangle^{\frac16-\kappa_0}\Big)\Phi(\tau,x+\tau v) \notag\\
			&\quad
			+M_T\langle v\rangle\int_s^t\langle\tau\rangle^{\frac16}\Phi(\tau,x+\tau v)\,d\tau
			+M_T\langle v\rangle\langle t\rangle^{\frac16-\kappa_0}\int_s^t\langle\tau\rangle^{-\frac23+2\kappa_0}\Phi(\tau,x+\tau v)d\tau .
		\end{align}
For the weighted oscillatory integrals, use the primitives
$(\tau-s)\sin\tau+\cos\tau$ and
$-(\tau-s)\cos\tau+\sin\tau$, respectively.
Integration by parts, together with \eqref{Ebounds1} and
\eqref{Ysu1}, then gives
\begin{align*}
|Y_{s,t}(x,v)|+\langle s\rangle|\partial_sY_{s,t}(x,v)|
&\lesssim M_T\sum_{\tau\in\{s,t\}}
\langle\tau\rangle^{1+50\kappa_0}
\langle\tau,x+\tau v\rangle^{-2}\\
&\quad+M_T\int_s^t
\bigl(1+\langle v\rangle
\langle\tau\rangle^{\frac16+2\kappa_0}\bigr)
\langle\tau\rangle^{\frac12}
\Phi(\tau,x+\tau v)\,d\tau.
\end{align*}

	\subsection{Estimates along the shifted ray and the change of variables} Let
	\begin{equation}\label{newDefY}
	\widetilde Y_{s,t}(x,v):=Y_{s,t}(x-tv,v),\qquad \widetilde W_{s,t}(x,v):=W_{s,t}(x-tv,v).
	\end{equation}
	Recalling the function $X$ defined in \eqref{Lan1}, we notice that
	\begin{equation}\label{shifted-Y-equation}
	\begin{split}
	&X_{s,t}(x,v)=x-(t-s)v+\widetilde Y_{s,t}(x,v),\\
	&\widetilde Y_{s,t}(x,v)=\int_s^t(r-s)E\bigl(r,X_{r,t}(x,v)\bigr)\,dr.
	\end{split}
	\end{equation}
Differentiating \eqref{shifted-Y-equation} with respect to $x$ and $v$, we obtain
\begin{align}
	\nabla_x\widetilde Y_{s,t}&=\int_s^t(r-s)\nabla E(r,X_{r,t})\big(I+\nabla_x\widetilde Y_{r,t}\big)dr,
	\label{shifted-x-system}\\
	\nabla_v\widetilde Y_{s,t}&=\int_s^t(r-s)\nabla E(r,X_{r,t})\big(-(t-r)I+\nabla_v\widetilde Y_{r,t}\big)\,dr.
	\label{shifted-v-system}
\end{align}
Differentiating once more with respect to $v$ gives
\begin{align}
	\nabla_v^2\widetilde Y_{s,t}
	&=\int_s^t(r-s)\Big[\nabla E(r,X_{r,t})\nabla_v^2\widetilde Y_{r,t}+
	\nabla^2E(r,X_{r,t})(-(t-r)I+\nabla_v\widetilde Y_{r,t})^{\otimes 2}\Big]\,dr.
	\label{shifted-vv-system}
\end{align}

We prove now several estimates on the function $\widetilde{Y}$ and its derivatives. For simplicity of notation, for $s\leq t\in[0,T]$ let
	\begin{equation}\label{etast}
		\eta_{s,t}:=\frac{t-s}{\langle t\rangle}\approx\min\Big\{\frac{t-s}{\langle s\rangle},1\Big\}.
	\end{equation}

	\begin{lemma}[Shifted-ray estimates]\label{lem:shifted-ray} With $\widetilde{Y}$ defined as in \eqref{newDefY}, we have
		\begin{align}
			|\widetilde Y_{s,t}(x,v)|
			&\lesssim M_T\mathfrak{a}_0(s,v)\eta_{s,t}.
			\label{z80a}
		\end{align}
		Moreover, if $n_1,n_2\ge0$ and $n_1+n_2=1$, then
	\begin{align}\label{z80'b}
	&\left|\nabla_x^{n_1}\nabla_v^{n_2}\widetilde Y_{s,t}(x,v)\right|\lesssim M_T\mathfrak{a}_1(s,v)(t-s)^{n_2}\eta_{s,t}.
	\end{align}
	Finally, the refined second-derivative estimate is
	\begin{align}\label{z80'c}
		\left|\nabla_v^2\widetilde Y_{s,t}(x,v)\right|&\lesssim M_T\mathfrak{a}_2(s,v)(t-s)^2.
	\end{align}
	\end{lemma}
	
\begin{proof}  Since $\widetilde Y_{t,t}=0$, we have
	\begin{equation*}
		\widetilde Y_{s,t}(x,v)=-\int_s^t\partial_rY_{r,t}(x-tv,v)\,dr.
	\end{equation*}
	Therefore, by \eqref{n13},	$|\widetilde Y_{s,t}(x,v)|\lesssim M_T\mathfrak a_0(s,v)$ and
	\begin{align*}
|\widetilde Y_{s,t}(x,v)|
		&\lesssim M_T\int_s^t\frac{\mathfrak a_0(r,v)}{\langle r\rangle}\,dr
		\lesssim M_T\mathfrak a_0(s,v)\frac{t-s}{\langle s\rangle}.
	\end{align*}
The desired bounds \eqref{z80a} follow. 

The bounds \eqref{z80'b} follow by the same argument, using \eqref{n14} instead of \eqref{n13}, if $(n_1,n_2)=(1,0)$. On the other hand, if $(n_1,n_2)=(0,1)$ then we use first \eqref{shifted-x-system} and integrate by parts to see that
\begin{align*}
&\int_s^t(r-s)(t-r)\nabla E(r,X_{r,t})
\big(I+\nabla_x\widetilde Y_{r,t}\big)\,dr=(t-s)\nabla_x\widetilde Y_{s,t}-2\int_s^t\nabla_x\widetilde Y_{r,t}\,dr.
\end{align*}
Notice that the functions $\eta_{s,t}$, $\mathfrak{a}_1(s,v)$, and $t-s$ are all decreasing in $s$. We use now the bounds \eqref{z80'b} with $(n_1,n_2)=(1,0)$ to conclude that
\[
\left|\int_s^t(r-s)(t-r)\nabla E(r,X_{r,t})\big(I+\nabla_x\widetilde Y_{r,t}\big)\,dr\right|\lesssim M_T\mathfrak a_1(s,v)(t-s)\eta_{s,t}.
\]
Moreover,
\begin{align*}
\left|\int_s^t(r-s)(t-r)\nabla E(r,X_{r,t})\nabla_x\widetilde Y_{r,t}\,dr\right|&\lesssim
\sup_{r\in[s,t]}\big((t-r)|\nabla_x\widetilde Y_{r,t}|\big)\int_s^t(r-s)|\nabla E(r,X_{r,t})|\,dr\\
&\lesssim M_T^2\mathfrak a_1(s,v)(t-s)\eta_{s,t},
\end{align*}
where the last estimates follow because $\|\nabla E(r)\|_{L^\infty_x}\lesssim M_T\langle r\rangle^{-\frac83+\kappa_0}$. Therefore
\begin{equation}\label{sec4eq1}
\left|\int_s^t(r-s)(t-r)\nabla E(r,X_{r,t})\,dr\right|\lesssim M_T\mathfrak a_1(s,v)(t-s)\eta_{s,t}.
\end{equation}

Using \eqref{shifted-v-system}, we obtain
\begin{equation}\label{sec4eq2}
|\nabla_v\widetilde Y_{s,t}(x,v)|\lesssim M_T\mathfrak a_1(s,v)(t-s)\eta_{s,t}+\int_s^tM_T\langle r\rangle^{-\frac32}|\nabla_v\widetilde Y_{r,t}(x,v)|\,dr,
\end{equation}
for any $s\leq t\in[0,T]$ and $x,v\in\R^2$. Notice that the function 
\(
s\mapsto
\mathfrak a_1(s,v)(t-s)\eta_{s,t}
\)
is decreasing in $s$  and the integral kernel in \eqref{sec4eq2} has $L^1$ norm $\lesssim M_T$. A standard Volterra absorption argument gives
\[
|\nabla_v\widetilde Y_{s,t}(x,v)|\lesssim M_T\mathfrak a_1(s,v)(t-s)\eta_{s,t},
\]
which completes the proof of \eqref{z80'b}.

Finally, we prove the bounds \eqref{z80'c}. In view of \eqref{shifted-vv-system} and the estimates $\|\nabla E(r)\|_{L^\infty_x}\lesssim M_T\langle r\rangle^{-\frac52}$ used earlier, for \eqref{z80'c} it suffices to prove that
\begin{equation}\label{sec4eq3}
\Big|\int_s^t(r-s)\nabla^2E(r,X_{r,t})(-(t-r)I+\nabla_v\widetilde Y_{r,t})^{\otimes 2}\,dr\Big|\lesssim M_T\mathfrak{a}_2(s,v)(t-s)^2
\end{equation}
for any $s\leq t\in[0,T]$ and $x,v\in\R^2$.

We differentiate \eqref{shifted-x-system} in $x$ to conclude that 
\begin{equation*}
\nabla_x^2\widetilde{Y}_{s,t}=\int_s^t(r-s)\Big[\nabla E(r,X_{r,t})\nabla_x^2\widetilde Y_{r,t}+
	\nabla^2E(r,X_{r,t})(I+\nabla_x\widetilde Y_{r,t})^{\otimes 2}\Big]\,dr.
\end{equation*}
As in the preceding integration-by-parts calculation, we integrate by parts twice to conclude that
\begin{align*}
&\int_s^t(r-s)(t-r)^2\Big[\nabla E(r,X_{r,t})\nabla_x^2\widetilde Y_{r,t}+\nabla^2E(r,X_{r,t})\big(I+\nabla_x\widetilde Y_{r,t}\big)^{\otimes2}\Big]\,dr\\
&\qquad=(t-s)^2\nabla_x^2\widetilde Y_{s,t}+\int_s^t\partial_r^2\big[(r-s)(t-r)^2\big]\nabla_x^2\widetilde Y_{r,t}\,dr.
\end{align*}
The bounds \eqref{n15} give
\begin{equation*}
|\nabla_x^2\widetilde Y_{r,t}(x,v)|\lesssim M_T\mathfrak a_2(r,v),\qquad r\in[s,t].
\end{equation*}
Since $\left|\partial_r^2\big[(r-s)(t-r)^2\big]\right|\lesssim t-s$ and $\mathfrak a_2$ is decreasing, it follows that
\begin{equation}\label{sec4eq4}
\begin{split}
\Big|\int_s^t(r-s)(t-r)^2\Big[\nabla E(r,X_{r,t})\nabla_x^2\widetilde Y_{r,t}
+\nabla^2E(r,X_{r,t})\big(I+\nabla_x\widetilde Y_{r,t}\big)^{\otimes2}\Big]\,dr\Big|\\
\lesssim M_T\mathfrak a_2(s,v)(t-s)^2.
\end{split}
\end{equation}

In addition,
\begin{align*}
&\Big|\int_s^t(r-s)(t-r)^2\nabla E(r,X_{r,t})\nabla_x^2\widetilde Y_{r,t}\,dr\Big|\\
&\lesssim
\sup_{r\in[s,t]}\big((t-r)^2|\nabla_x^2\widetilde Y_{r,t}|\big)\int_s^t(r-s)|\nabla E(r,X_{r,t})|\,dr\lesssim
M_T^2\mathfrak a_2(s,v)(t-s)^2.
\end{align*}
Using also \eqref{sec4eq4} it follows that
\begin{equation}\label{sec4eq5}
\Big|\int_s^t(r-s)(t-r)^2\nabla^2E(r,X_{r,t})\big(I+\nabla_x\widetilde Y_{r,t}\big)^{\otimes2}\,dr\Big|
\lesssim M_T\mathfrak a_2(s,v)(t-s)^2.
\end{equation}

On the other hand, the bounds \eqref{z80'b} give
\[
\big|\nabla_v\widetilde Y_{r,t}\big|+(t-r)\big|\nabla_x\widetilde Y_{r,t}\big|\lesssim M_T\mathfrak a_1(r,v)(t-r)\eta_{r,t}.
\]
Therefore
\begin{align*}
&\big|\big(-(t-r)I+\nabla_v\widetilde Y_{r,t}\big)^{\otimes2}-(t-r)^2\big(I+\nabla_x\widetilde Y_{r,t}\big)^{\otimes2}\big|\lesssim
M_T\mathfrak a_1(r,v)(t-r)^2.
\end{align*}
Notice that 
\[
\mathfrak a_1(r,v)\lesssim\mathfrak a_2(r,v)\langle r\rangle.
\]
The pointwise field estimates imply $|\nabla^2E(r,X_{r,t})|\lesssim M_T\langle r\rangle^{-\frac{10}{3}+\kappa_0}$
and therefore we obtain
\begin{align*}
&\int_s^t(r-s)|\nabla^2E(r,X_{r,t})|\left|\big(-(t-r)I+\nabla_v\widetilde Y_{r,t}\big)^{\otimes2}
-(t-r)^2\big(I+\nabla_x\widetilde Y_{r,t}\big)^{\otimes2}
\right|\,dr\\
&\qquad\lesssim M_T^2\mathfrak a_2(s,v)(t-s)^2.
\end{align*}
The desired bounds \eqref{sec4eq3} follow using also \eqref{sec4eq5}, which completes the proof.
\end{proof}

	In several spacetime integrals we need to straighten the characteristic map
	\(
	v\mapsto X_{s,t}(x,v)
	\)
	into the free transport map
	\(
	v\mapsto x-(t-s)v
	\).
	To this end, we introduce the following velocity reparametrization.
	
	\begin{lemma}[Approximate free-transport reparametrization]\label{lem:Psi}
		For every $0\leq s<t\leq T$ and $x,v\in\mathbb{R}^2$, there exists a unique
		$\Psi_{s,t}(x,v)\in\mathbb{R}^2$ satisfying
		\begin{equation}\label{X(x,Psi)-new}
			X_{s,t}\!\big(x,\Psi_{s,t}(x,v)\big)=x-(t-s)v.
		\end{equation}
		Equivalently,
		\begin{equation}\label{z16a-new}
			\Psi_{s,t}(x,v)-v
			=
			\frac{1}{t-s}\,\widetilde{Y}_{s,t}\!\big(x,\Psi_{s,t}(x,v)\big).
		\end{equation}
		On the diagonal, we set $\Psi_{t,t}(x,v):=v$.
		Then the map $v\mapsto \Psi_{s,t}(x,v)$ is a $C^2$ diffeomorphism close to the identity, and the following estimates hold.
		
		(i) First-order bounds: for any $n_1+n_2\le 1$,
		\begin{align}\label{z52-new}
			|\nabla_x^{n_1}\nabla_v^{n_2}&(\Psi_{s,t}(x,v)-v)|\lesssim
			\begin{cases}
			M_T\mathfrak{a}_0(s,v)\langle t\rangle^{-1}\qquad&\text{ if }n_1=n_2=0,\\
			M_T\mathfrak{a}_1(s,v)\langle t\rangle^{-1}(t-s)^{n_2}\qquad&\text{ if }n_1+n_2=1.
			\end{cases}
		\end{align}

		(ii) Second $v$-derivative bound:
		\begin{equation}\label{q9-new}
			|\nabla_v^2\Psi_{s,t}(x,v)|\lesssim M_T\mathfrak{a}_2(s,v)(t-s).
		\end{equation}

		(iii) Jacobian control: in particular,
		\begin{equation}\label{z142c-new}
			\big|\det(\nabla_v\Psi_{s,t}(x,v))-1\big|\lesssim
			M_T\mathfrak{a}_1(s,v)\langle t\rangle^{-1}(t-s).
		\end{equation}
	\end{lemma}
	
	\begin{proof} For any $s<t\in[0,T]$ and $x\in\R^2$ define
\[
H_{s,t,x}(w):=w-\frac{1}{t-s}\widetilde Y_{s,t}(x,w).
\]
Then \eqref{X(x,Psi)-new} is equivalent to
\begin{equation}\label{sec4eq8}
H_{s,t,x}(\Psi_{s,t}(x,v))=v.
\end{equation}
By \eqref{z80'b},
\[
\sup_{x,w\in\R^2,\,s<t\in[0,T]}
\frac{1}{t-s}|\nabla_w\widetilde Y_{s,t}(x,w)|
\lesssim M_T.
\]
Therefore, for every $x,v\in\R^2$ and $s<t\in[0,T]$, the map
\[
w\mapsto v+\frac{1}{t-s}\widetilde Y_{s,t}(x,w)
\]
is a contraction of $\R^2$. It has a unique fixed point, denoted by
$\Psi_{s,t}(x,v)$, and this proves \eqref{z16a-new}. Moreover the map $w\mapsto H_{s,t,x}(w)$ is bijective and its Jacobian
\[
\nabla_wH_{s,t,x}(w)=I-\frac{1}{t-s}\nabla_w\widetilde Y_{s,t}(x,w)
\]
is close to the identity in $L^\infty$. The inverse function theorem therefore shows that $v\mapsto\Psi_{s,t}(x,v)$ is a global $C^2$ diffeomorphism satisfying \eqref{sec4eq8}.

The estimates \eqref{z80a}--\eqref{z80'c} and the definition \eqref{aweights} show that
\begin{equation}\label{sec4eq9}
\begin{split}
\frac{1}{t-s}|\widetilde Y_{s,t}(x,w)|
&\lesssim M_T\mathfrak a_0(s,w)\langle t\rangle^{-1},\\
\frac{1}{t-s}|\nabla_x\widetilde Y_{s,t}(x,w)|
&\lesssim M_T\mathfrak a_1(s,w)\langle t\rangle^{-1},\\
\frac{1}{t-s}|\nabla_w\widetilde Y_{s,t}(x,w)|
&\lesssim M_T\mathfrak a_1(s,w)\langle t\rangle^{-1}(t-s),\\
\frac{1}{t-s}|\nabla_w^2\widetilde Y_{s,t}(x,w)|
&\lesssim M_T\mathfrak a_2(s,w)(t-s).
\end{split}
\end{equation}
Thus $|\Psi_{s,t}(x,v)-v|\lesssim M_T\mathfrak a_0(s,\Psi)\langle t\rangle^{-1}\lesssim M_T$, so $\langle\Psi_{s,t}(x,v)\rangle\approx\langle v\rangle$ and $\mathfrak a_j(s,\Psi_{s,t}(x,v))\approx\mathfrak a_j(s,v)$ for any $x,v\in\R^2$, $s<t\in[0,T]$, and  $j\in\{0,1,2\}$.

Differentiating the fixed-point identity \eqref{sec4eq8} in $x$ and $v$, we obtain
\begin{align*}
\nabla_x\Psi_{s,t}(x,v)&=\Big(I-\frac{1}{t-s}\nabla_w\widetilde Y_{s,t}(x,\Psi_{s,t}(x,v))\Big)^{-1}
\frac{1}{t-s}\nabla_x\widetilde Y_{s,t}(x,\Psi_{s,t}(x,v)),\\
\nabla_v(\Psi_{s,t}(x,v)-v)
&=\Big(I-\frac{1}{t-s}\nabla_w\widetilde Y_{s,t}(x,\Psi_{s,t}(x,v))\Big)^{-1}
\frac{1}{t-s}\nabla_w\widetilde Y_{s,t}(x,\Psi_{s,t}(x,v)).
\end{align*}
These identities and the bounds \eqref{sec4eq9} give \eqref{z52-new} and \eqref{z142c-new}. The bounds \eqref{q9-new} follow also from \eqref{sec4eq9} by taking one more $v$-derivative. 
\end{proof}

A consequence is the following useful comparison estimate.
	
	\begin{lemma}[Comparison with free transport]\label{lem:change-var}
		Assume that $M_T$ is sufficiently small.
		Let $h,g\in C(\mathbb{R}^2;\mathbb{R}_+)$ and
		\begin{equation}\label{z36.01}
		g(v+z)\approx g(v)\qquad\text{for all }|z|\leq 1/2.
		\end{equation}
		Then for all $0\leq s\leq t\leq T$ and $x\in\mathbb R^2$,
		\begin{equation}\label{z36}
			\int_{\mathbb{R}^2} g(w)\,h(X_{s,t}(x,w))\,dw\approx \int_{\mathbb{R}^2} g(v)\,h(x-(t-s)v)\,dv.
		\end{equation}
	\end{lemma}
	\begin{proof}
		The case $s=t$ is immediate since $X_{t,t}(x,w)=x$.
		For $s<t$, \eqref{z52-new} and \eqref{z142c-new}, with $M_T$
		sufficiently small, give uniformly in $s,t,x,v$
		\[
		|\Psi_{s,t}(x,v)-v|\leq\frac12,\qquad
		\frac12\leq\det\nabla_v\Psi_{s,t}(x,v)\leq2.
		\]
		We make the change of variables $w=\Psi_{s,t}(x,v)$, which is justified by Lemma~\ref{lem:Psi}. Using \eqref{X(x,Psi)-new},
		\[
		\int_{\mathbb{R}^2} g(w)\,h(X_{s,t}(x,w))\,dw
		=
		\int_{\mathbb{R}^2} g(\Psi_{s,t}(x,v))\,h(x-(t-s)v)\,\det(\nabla_v\Psi_{s,t}(x,v))\,dv.
		\]
		The conclusion follows from the assumption \eqref{z36.01} and the Jacobian bound \eqref{z142c-new}.
	\end{proof}
	
	\subsection{Second $v$-derivative expansion}
	
	For $1\le i,j\le2$, define
	\begin{equation}\label{sec4eq50}
		O_{s,t}^{i,j}(x,v):=\int_s^t(\tau-s)\tau^2\partial_{x_i x_j}^2E^{osc}(\tau, x+\tau v)\,d\tau.
	\end{equation}
	
	\begin{lemma}\label{lem:second-v-expansion}
		For all $0\le s\le t\le T$, the following estimates hold:
		\begin{align}
			\Big|\partial_{v_j}\partial_{v_i}Y_{s,t}(x,v)-O_{s,t}^{i,j}(x,v)\Big|&\lesssim M_T
			\langle t\rangle^{\frac16-\kappa_0}\langle v\rangle^{-1},
			\label{q27}\\
			\langle s\rangle\left|
			\partial_{v_j}\partial_{v_i}\partial_sY_{s,t}(x,v)-\partial_sO_{s,t}^{i,j}(x,v)
			\right|&\lesssim M_T\langle t\rangle^{\frac16-\kappa_0}
			\langle v\rangle^{-1}.\label{q28}
		\end{align}
		Moreover,
		\begin{align}
			|O_{s,t}^{i,j}(x,v)|&\lesssim M_T
			\langle t\rangle^{\frac23+\kappa_0}
			\langle t^{\frac12+2\kappa_0},v\rangle^{-1},
			\label{q62a}\\
			|\partial_sO_{s,t}^{i,j}(x,v)|&\lesssim M_T
			\langle s\rangle^{-\frac13+\kappa_0}\langle s^{\frac12+2\kappa_0},v\rangle^{-1},
			\label{q62b}
		\end{align}
		and
		\begin{align}
			|\nabla_{x}O_{s,t}^{i,j}(x,v)|
			&\lesssim
			M_T
			\langle t\rangle^{\frac16-\kappa_0}
			\langle v\rangle^{-1},
			\label{q62c}
			\\
			|\nabla_{x}\partial_sO_{s,t}^{i,j}(x,v)|
			&\lesssim
			M_T
			\langle s\rangle^{-\frac56-\kappa_0}\langle v\rangle^{-1}.
			\label{q62d}
		\end{align}
		Finally, for $s<t$,
		\begin{align}
			\frac{1}{t-s}\left|O_{s,t}^{i,j}(x-tv,v)\right|&\lesssim M_T\langle s\rangle^{-\frac13+\kappa_0}\langle v\rangle^{-1},\label{q29}\\
			\frac{1}{t-s}\left|\nabla_v\big[O_{s,t}^{i,j}(x-tv,v)\big]\right|&\lesssim
			M_T\langle t\rangle^{\frac16-\kappa_0}\langle v\rangle^{-1}.\label{q30}
		\end{align}
	\end{lemma}
	
	\begin{proof}
The estimate \eqref{q62a} follows from \eqref{q62} with
$n_1=2$ and $n_2=0$. Moreover,
\[
	\partial_sO_{s,t}^{i,j}(x,v)
	=
	-\int_s^t\tau^2
	\partial_{x_ix_j}^2E^{osc}(\tau,x+\tau v)\,d\tau,
\]
so \eqref{q62b} follows from \eqref{q62p}. Using \eqref{Ebounds1},
\begin{align*}
	|\nabla_xO_{s,t}^{i,j}(x,v)|
	&\lesssim
	M_T\int_s^t\langle\tau\rangle^{7/6}
	\Phi(\tau,x+\tau v)\,d\tau,\\
	|\nabla_x\partial_sO_{s,t}^{i,j}(x,v)|
	&\lesssim
	M_T\int_s^t\langle\tau\rangle^{1/6}
	\Phi(\tau,x+\tau v)\,d\tau.
\end{align*}
The bounds \eqref{q62c} and \eqref{q62d} now follow from
\eqref{z210} and \eqref{z212}, respectively.

We next prove \eqref{q27}. Differentiating \eqref{Yinteq} twice
with respect to $v$, we obtain
\begin{align*}
	\partial_{v_j}\partial_{v_i}Y_{s,t}
	&=\int_s^t(\tau-s)\tau^2
	\partial_{x_ix_j}^2E(\tau,x+\tau v+Y_{\tau,t})\,d\tau\\
	&\quad+\int_s^t(\tau-s)\tau
	\Big[
	\partial_{v_i}Y_{\tau,t}\cdot\nabla_x\partial_{x_j}E
	+\partial_{v_j}Y_{\tau,t}\cdot\nabla_x\partial_{x_i}E
	\Big](\tau,x+\tau v+Y_{\tau,t})\,d\tau\\
	&\quad+\int_s^t(\tau-s)
	\partial_{v_i}\partial_{v_j}Y_{\tau,t}\cdot
	\nabla_xE(\tau,x+\tau v+Y_{\tau,t})\,d\tau\\
	&\quad+\int_s^t(\tau-s)
	\big(\partial_{v_i}Y_{\tau,t}\otimes\partial_{v_j}Y_{\tau,t}\big)
	:\nabla_x^2E(\tau,x+\tau v+Y_{\tau,t})\,d\tau.
\end{align*}
The bounds \eqref{n13}, \eqref{xf}, and \eqref{n17} imply that $|Y_{\tau,t}|\lesssim M_T$, $|\nabla_vY_{\tau,t}|\lesssim M_T\langle\tau\rangle^{-1/6-\kappa_0}$, and $|\nabla_v^2Y_{\tau,t}|\lesssim M_T\langle t\rangle^{1/6-\kappa_0}$. 
In particular, for $0\leq\theta\leq1$,
\[
	\Phi(\tau,x+\tau v+\theta Y_{\tau,t})
	\lesssim\Phi(\tau,x+\tau v).
\]
Splitting the leading field into its regular and oscillatory
parts, and applying the mean value theorem to the oscillatory part,
the preceding identity, \eqref{cbc2}, and \eqref{Ebounds1} give
\begin{align*}
	\left|\partial_{v_j}\partial_{v_i}Y_{s,t}-O_{s,t}^{i,j}\right|
	&\lesssim M_T\int_s^t\langle\tau\rangle^{7/6}\Phi(\tau,x+\tau v)\,d\tau+ M_T\langle t\rangle^{1/6-\kappa_0}
	\int_s^t\langle\tau\rangle^{1/3+2\kappa_0}\Phi(\tau,x+\tau v)\,d\tau.
\end{align*}
By \eqref{z210}--\eqref{z212}, the two integrals are bounded by $\langle t\rangle^{1/6-\kappa_0}\langle v\rangle^{-1}$ and $\langle s\rangle^{-2/3+\kappa_0}\langle v\rangle^{-1}$. This proves \eqref{q27}.

Similarly,
\[
	\partial_sY_{s,t}(x,v)
	=-\int_s^tE(\tau,x+\tau v+Y_{\tau,t})\,d\tau.
\]
Differentiating twice in $v$ gives the same error terms as above,
with the factor $\tau-s$ removed, while the leading oscillatory term
is $\partial_sO_{s,t}^{i,j}$. Therefore
\begin{align*}
	&\left|\partial_{v_j}\partial_{v_i}\partial_sY_{s,t}
	-\partial_sO_{s,t}^{i,j}\right|
	\lesssim M_T\int_s^t\langle\tau\rangle^{1/6}
	\Phi(\tau,x+\tau v)\,d\tau\\
	&\qquad+M_T\langle t\rangle^{1/6-\kappa_0}
	\int_s^t\langle\tau\rangle^{-2/3+2\kappa_0}\Phi(\tau,x+\tau v)\,d\tau\lesssim
	M_T\langle t\rangle^{1/6-\kappa_0}
	\langle s\rangle^{-1}\langle v\rangle^{-1},
\end{align*}
which proves \eqref{q28}.

Finally, we prove \eqref{q29}--\eqref{q30}. Since $x-(t-\tau)v=(x-tv)+\tau v$, we have
\begin{align*}
	\frac{|O_{s,t}^{i,j}(x-tv,v)|}{t-s}&\lesssim
	M_T\int_s^t\langle\tau\rangle^{2/3+2\kappa_0}
	\Phi(\tau,x-(t-\tau)v)\,d\tau\lesssim M_T\langle s\rangle^{-1/3+\kappa_0}\langle v\rangle^{-1},
\end{align*}
using \eqref{Ebounds1} and \eqref{z212}. Moreover,
\[
	\nabla_v\!\left[O_{s,t}^{i,j}(x-tv,v)\right]=\int_s^t(\tau-s)\tau^2(\tau-t)\nabla_x\partial_{x_ix_j}^2E^{osc}(\tau,x-(t-\tau)v)\,d\tau.
\]
Consequently, using now \eqref{Ebounds1} and \eqref{z210},
\begin{align*}
	\frac{
	\left|\nabla_v[O_{s,t}^{i,j}(x-tv,v)]\right|}{t-s}&\lesssim M_T\int_s^t\langle\tau\rangle^{7/6}
	\Phi(\tau,x-(t-\tau)v)\,d\tau\lesssim M_T\langle t\rangle^{1/6-\kappa_0}\langle v\rangle^{-1}.
\end{align*}
The bounds \eqref{q29}--\eqref{q30} follow, which completes the proof of the lemma.
\end{proof}

\section{Estimates for the fields $\mathcal{I}_k$ and $\mathcal{R}_k$, $k=0,1,2$}
	
	Our main goal is to prove the bootstrap bounds \eqref{B} in
Proposition~\ref{propoglobal3d}. In view of \eqref{de1.1} and
\eqref{z40.4}, the main input is provided by estimates for the moment
fields $\mathbf R_k$, $k=0,1,2$. Recall from
\eqref{z1.5}--\eqref{z1.6} that $\mathbf R_k=\mathcal I_k+\mathcal R_k$, 
where $\mathcal I_k=\mathcal I_k(E)$ is the initial-data contribution
and $\mathcal R_k=\mathcal R_k(E)$ is the nonlinear remainder. More
explicitly,
\begin{equation}\label{sec5eq1}
\begin{split}
&\mathcal{I}_k(t,x)
=\int_{\mathbb{R}^2}v^{\otimes k}
f_0\bigl(X_{0,t}(x,v),V_{0,t}(x,v)\bigr)\,dv,\\
&\mathcal{R}_k(t,x)
=\int_0^t\int_{\mathbb{R}^2}v^{\otimes k}
\bigl\{
E(s,x-(t-s)v)\cdot\nabla_v\mu(v)
-E(s,X_{s,t}(x,v))\cdot\nabla_v\mu(V_{s,t}(x,v))
\bigr\}\,dv\,ds.
\end{split}
\end{equation}

As before, we assume throughout this section that $(f,E)$ is a solution
of \eqref{eq2} on $[0,T]$ satisfying $M_T\leq\epsilon\ll1$. We first
estimate $\mathcal I_k$ and then $\mathcal R_k$.

	\subsection{Bounds for the initial data components $\mathcal{I}_k$}
	
	\begin{lemma}\label{estI_k} For any $k=0,1,2$, $n=0,1,2$ and $(t,x)\in[0,T]\times\mathbb{R}^2$ we have
		\begin{equation}\label{Z16-new}
			|\nabla_x^n\mathcal{I}_k(t,x)|\lesssim [f_0]\langle t\rangle^{19\kappa_0-n+\mathbf{1}_{n=2}(\frac16-\kappa_0)}
			\langle t,x\rangle^{-2-19\kappa_0}.
		\end{equation}
	\end{lemma}
	
	\begin{proof} If $t\in [0,1]$ then we recall the definition \eqref{a16'} and use \eqref{n13}, \eqref{n14}, and \eqref{n15} to estimate
		\begin{align*}
			\sum_{ n=0}^2|\nabla_{x}^n\mathcal{I}_k(t,x)|\lesssim [f_0]\int_{\mathbb{R}^2}\langle x-tv,v\rangle^{-2-20\kappa_0}\langle v\rangle^{-2}dv\lesssim
			[f_0]\langle x\rangle^{-2-19\kappa_0}.
		\end{align*}
		On the other hand, if $t\geq 1$ then we rewrite
		\begin{align*}
			\mathcal{I}_k(t,x)=\int_{\R^2}\Big(\frac{x-w}{t}\Big)^{\otimes k} f_0\Big(w +Y_{0,t}\Big(w,\frac{x-w}{t}\Big), \frac{x-w}{t}+W_{0,t}\Big(w,\frac{x-w}{t}\Big)\Big)\frac{dw}{t^2}.
		\end{align*}
		Using now the bounds \eqref{n13}, \eqref{xf}, \eqref{n17}, and \eqref{n17b} we can estimate, for $n,k\in\{0,1,2\}$,
		\begin{equation*}
		\begin{split}
		|\nabla_{x}^n\mathcal{I}_k(t,x)|&\lesssim \langle t\rangle^{-n+\mathbf{1}_{n=2}(\frac{1}{6}-\kappa_0)}[f_0]\int_{\mathbb{R}^2}\Big\langle w,\frac{x-w}{t}\Big\rangle^{-2-20\kappa_0}\Big\langle \frac{x-w}{t}\Big\rangle^{-2} t^{-2}dw\\
		&\lesssim \langle t\rangle^{-n+\mathbf{1}_{n=2}(\frac{1}{6}-\kappa_0)}[f_0]\int_{\mathbb{R}^2}\langle x-tv,v\rangle^{-2-20\kappa_0}\langle v\rangle^{-2}dv\\
		&\lesssim \langle t\rangle^{-n+\mathbf{1}_{n=2}(\frac{1}{6}-\kappa_0)}[f_0]\langle t\rangle^{19\kappa_0}\langle t,x\rangle^{-2-19\kappa_0}.
		\end{split}
		\end{equation*}
		This completes the proof of the lemma. 
	\end{proof}

	\subsection{Main bounds for $\mathcal{R}_k$} Our main result in this section is the following:
	
	\begin{proposition}\label{maines} For $k=0,1,2$ we have
		\begin{align}
			\label{q43}
			&	\big|\mathcal{R}_k(t,x)\big|
			\lesssim M_T^2\langle t\rangle^{20\kappa_0}\Phi(t,x),\\
			\label{q44}
			&	 \big|\nabla \mathcal{R}_k(t,x)\big|
			\lesssim M_T^2\langle t\rangle^{-\frac{5}{6}+20\kappa_0}\,\Phi(t,x),\\
			\label{q45}
			&	\big|\nabla^2 \mathcal{R}_k(t,x)\big|
			\lesssim M_T^2\langle t\rangle^{-\frac{3}{2}+30\kappa_0}\,\Phi(t,x),\\
			\label{q46} 
			&\big|\nabla^3(1-\Delta)^{-1}\mathcal{R}_k(t,x)\big|
			\lesssim M_T^2\langle t\rangle^{-\frac{11}{6}}\,\Phi(t,x).
		\end{align}
		Moreover, for any $\varphi\in W^{4,\infty}\cap W^{4,1}$ and $\epsilon_1\in [2\kappa_0,1/6-2\kappa_0]$, we have
		\begin{align}
			\left|\int_{\mathbb{R}^2} \nabla_x^3\varphi(x)\,\mathcal{R}_k(t,x)\,dx\right|
			&\lesssim M_T^2\langle t\rangle^{-\frac{11}{6}}\int_{\mathbb{R}^2}\Big(|\nabla \varphi(x)|
			+\big|\nabla |\nabla|^{-\frac{5}{6}-\epsilon_1}\varphi(x)\big|\Big)\,\Phi(t,x)\,dx.\label{q47}
		\end{align}
\end{proposition}
	
For $b=b(s,v)$ and a weight $\omega=\omega(v)$ we define the more general trilinear operators
	\begin{equation}\label{defN-new}
		\mathcal{N}(F,\omega,b)(t,x)
		:=
		\int_0^t\!\!\int_{\mathbb{R}^2} b(s,v)\Big\{F(s,x-(t-s)v)\omega(v)-F(s,X_{s,t}(x,v))\omega(V_{s,t}(x,v))\Big\}\,dv\,ds.
	\end{equation}
	In particular, recalling the definition \eqref{sec5eq1},
	\begin{equation}\label{sec5eq6}
	\mathcal{R}_k(t,x)=\mathcal{N}\big(E,\nabla\mu,v^{\otimes k}\big)(t,x).
	\end{equation}
	The main advantage of using the more general definition is that $x$-derivatives of $\mathcal{R}_k$ can also be expressed in terms of the trilinear form $\mathcal{N}$ up to lower order terms.
	
	We use the velocity change of variables $\Psi_{s,t}(x,v)$ defined as in Lemma \ref{lem:Psi}.
	Changing variables $v\mapsto \Psi_{s,t}(x,v)$ in the characteristic term yields
	\begin{equation}\label{q1}
		\mathcal{N}(F,\omega,b)(t,x)=\int_0^t\!\!\int_{\mathbb{R}^2} F(s,x-(t-s)v)\,\Sigma_{s,t}(x,v)\,dv\,ds,
	\end{equation}
	where
	\begin{equation}\label{sec5eq7}
	\Sigma_{s,t}(x,v)=b(s,v)\omega(v)-\det(\nabla_v\Psi_{s,t}(x,v))\,b\big(s,\Psi_{s,t}(x,v)\big)\,\omega\big(V_{s,t}(x,\Psi_{s,t}(x,v))\big).
	\end{equation}
	The function $\Sigma_{s,t}$ is to be thought of as small, as it encodes the failure of the characteristic flow to coincide with free transport.
	
	\subsubsection{Outline of the proof of Proposition~\ref{maines}} We prove the main bounds \eqref{q43}--\eqref{q47} in several steps, in the next four subsections:
\smallskip

{\bf{Step 1.}} In Subsection~\ref{subsection5.3}, we estimate the multipliers
$\Sigma_{s,t}$ and their $v$-derivatives. Lemma~\ref{Lem1} then gives
the pointwise bounds \eqref{q43}--\eqref{q45} for $0\leq t\leq1$
and in the exterior region $t\geq1$, $|x|\geq\langle t\rangle^{10}$.
\smallskip

{\bf{Step 2.}} In subsection \ref{zerobound} we prove estimates on expressions of the form
\[
\mathcal N(D_x^\alpha E,\partial_v^\beta\mu,v^{\otimes k}),
\]
$|\alpha|\leq2$.  
These bounds suffice to prove the bounds \eqref{q43} in the interior region
\begin{equation}\label{sec5eq10}
\mathcal{D}_T:=\{(t,x)\in [0,T]\times\R^2:\,t\geq1\text{ and }|x|\leq\langle t\rangle^{10}\}.
\end{equation}
and to control parts of the differentiated terms $\nabla^j_x\mathcal{R}_k$, $j=1,2$. 
The contribution of the regular field is estimated directly, while for the oscillatory field we integrate by parts in time and sometimes also in velocity.
\smallskip

{\bf{Step 3.}} In subsection \ref{firstorderbound} we make a change of variables of the form $w=x-tv$ to transfer one spatial
derivative to the field and to the velocity weights, producing a factor
$t^{-1}$ and a remainder $\mathcal N_j^{r,1}$. The bounds in subsection \ref{zerobound} control the transferred derivatives, while the characteristic estimates control the remainder. In particular, we prove the bounds \eqref{q44} in the interior region.
\smallskip

{\bf{Step 4.}} In subsection \ref{secondorderbound} we differentiate $\mathcal N_j^{r,1}$ to produce a second-order remainder
$\mathcal N_{ji}^{r,2}$. Its lower-order terms are estimated directly,
while its principal term is analyzed using the second $v$-derivative
expansion from Lemma~\ref{lem:second-v-expansion}. The resulting
pointwise bounds prove \eqref{q45} in the interior region. The leading oscillatory term also
satisfies a fractional difference estimate, which yields the bounds \eqref{q46} and
the pairing estimate \eqref{q47}.

\subsection{The multiplier $\Sigma_{s,t}$ and exterior estimates}\label{subsection5.3} We start by proving several bounds on the multipliers $\Sigma_{s,t}$ defined in \eqref{sec5eq7}.
	
\begin{lemma}[Bounds for $\Sigma_{s,t}$]\label{lem:Sigma-new}
	Recall the weights $\mathfrak a_0,\mathfrak a_1,\mathfrak a_2$ defined in \eqref{aweights}, and let
	\[
	b(s,v)=v^{\otimes k},
	\qquad
	\omega(v)=\partial_v^\beta\mu(v),
	\qquad
	k\in\{0,1,2\},
	\qquad
	k-|\beta|\leq 1.
	\]
	Then, for all $0\leq s\leq t\leq T$ and $x,v\in\R^2$, we have
	\begin{align}
		|\Sigma_{s,t}(x,v)|&\lesssim M_T\langle v\rangle^{-3+k-|\beta|}\mathfrak a_1(s,v)(t-s)\langle t\rangle^{-1}+ M_T\langle v\rangle^{-4+k-|\beta|}
		\mathfrak a_0(s,v)\langle s\rangle^{-1},
		\label{q3}\\
		|\nabla_v\Sigma_{s,t}(x,v)|
		&\lesssim
		M_T\langle v\rangle^{-3+k-|\beta|}
		\mathfrak a_2(s,v)(t-s)
		\notag\\
		&+M_T\langle v\rangle^{-4+k-|\beta|}\big(\langle v\rangle^{-1}\langle s\rangle^{-\frac16-\kappa_0}+\langle s,|v|^{\frac32}\rangle^{-\frac13+\kappa_0}\big)\langle s\rangle^{-1}.
		\label{q10}
	\end{align}
	Moreover, we also have the following cruder bounds
	\begin{equation}\label{z2}
		|\Sigma_{s,t}(x,v)|
		\lesssim
		M_T\langle t\rangle^{1+50\kappa_0}
		\left(
		\int_0^t\langle x-\tau v\rangle^{-2}\,d\tau
		\right)
		\langle v\rangle^{-2}.
	\end{equation}
\end{lemma}

\begin{proof} The conclusions are immediate when $s=t$ since $\Psi_{t,t}(x,v)=v$ and $\Sigma_{t,t}(x,v)=0$, so we may assume that $s<t$.
	For simplicity, set
	\begin{equation}\label{sec5eq11}
	J_{s,t}(x,v):=\det(\nabla_v\Psi_{s,t}(x,v)),\qquad W^\sharp_{s,t}(x,v):=W_{s,t}(x-t\Psi_{s,t}(x,v),\Psi_{s,t}(x,v)).
	\end{equation}
	Thus
	\begin{equation}\label{sigma-Q-identity}
	V_{s,t}(x,\Psi_{s,t}(x,v))=\Psi_{s,t}(x,v)+W^\sharp_{s,t}(x,v).
	\end{equation}
	In addition, recalling \eqref{PoEq}, for $m=0,1,2$ we have
	\begin{equation}\label{sigma-vel-decay}
		\big|\nabla_v^m\bigl(v^{\otimes k}\partial_v^\beta\mu(v)\bigr)\big|\lesssim\langle v\rangle^{-3+k-|\beta|-m}.
	\end{equation}

By Lemma~\ref{lem:Psi}, the definition \eqref{aweights}, and
	\eqref{z142c-new}, we have
	\begin{equation}\label{sigma-Psi}
	\begin{split}
		|\Psi_{s,t}(x,v)-v|&\lesssim M_T\mathfrak a_0(s,v)\langle t\rangle^{-1},\\
		|\nabla_v\Psi_{s,t}(x,v)-I|+|J_{s,t}(x,v)-1|&\lesssim M_T\mathfrak a_1(s,v)\langle t\rangle^{-1}(t-s),\\
		|\nabla_v^2\Psi_{s,t}(x,v)|+|\nabla_vJ_{s,t}(x,v)|&\lesssim M_T\mathfrak a_2(s,v)(t-s).
	\end{split}
	\end{equation}
	We will use the following direct consequences of \eqref{aweights}:
	\begin{equation}\label{sigma-weight-comparison}
	\begin{split}
		\langle s\rangle\mathfrak a_1(s,v)+\mathfrak a_0(s,v)
		&\lesssim
		\langle v\rangle^{-1}\langle s\rangle^{-\frac16-\kappa_0}+
		\langle s,|v|^{\frac32}\rangle^{-\frac13+\kappa_0},\\
		\langle v\rangle^{-1}\langle s\rangle^{-1}
		\mathfrak a_1(s,v)
		&\lesssim
		\mathfrak a_2(s,v).
	\end{split}
	\end{equation}
By \eqref{n13} and the comparability of the velocity weights,
	\begin{equation}\label{sigma-W}
		|W^\sharp_{s,t}(x,v)|\lesssim M_T\mathfrak a_0(s,v)\langle s\rangle^{-1}.
	\end{equation}
We next differentiate $W^\sharp_{s,t}$. By the chain rule,
	\begin{align*}
	\nabla_vW^\sharp_{s,t}(x,v)
	&=
	\Big[
	-t\nabla_xW_{s,t}
	+\nabla_vW_{s,t}
	\Big]
	\bigl(x-t\Psi_{s,t}(x,v),\Psi_{s,t}(x,v)\bigr)
	\nabla_v\Psi_{s,t}(x,v).
	\end{align*}
	Using \eqref{n14}, \eqref{xf}, and
	$|\nabla_v\Psi_{s,t}(x,v)|\lesssim1$, we obtain
	\begin{align*}
	|\nabla_vW^\sharp_{s,t}(x,v)|\lesssim
	M_T
	\left(\langle v\rangle^{-1}\langle s\rangle^{-\frac16-\kappa_0}+
	\langle s,|v|^{\frac32}\rangle^{-\frac13+\kappa_0}\right)
	\langle s\rangle^{-1}+M_Tt\,\mathfrak a_1(s,v)\langle s\rangle^{-1}.
	\end{align*}
	Writing $t=s+(t-s)$ and using
	\eqref{sigma-weight-comparison} we conclude that
	\begin{align}
	|\nabla_vW^\sharp_{s,t}(x,v)|
	&\lesssim M_T\left(
	\langle v\rangle^{-1}\langle s\rangle^{-\frac16-\kappa_0}+
	\langle s,|v|^{\frac32}\rangle^{-\frac13+\kappa_0}
	\right)
	\langle s\rangle^{-1}+
	M_T\langle v\rangle\mathfrak a_2(s,v)(t-s).
	\label{sigma-Wv}
	\end{align}

	We prove now \eqref{q3}. By \eqref{sigma-Q-identity},
	\eqref{sigma-Psi}, and \eqref{sigma-W},
	\begin{align}
	&\left|
	V_{s,t}\bigl(x,\Psi_{s,t}(x,v)\bigr)-v
	\right|
	\lesssim
	M_T\mathfrak a_0(s,v)\langle s\rangle^{-1}.
	\label{sigma-V}
	\end{align}
	Therefore, the definition of $\Sigma_{s,t}$ in \eqref{sec5eq7}, the mean value theorem,
	and \eqref{sigma-vel-decay} give
	\begin{align}
	|\Sigma_{s,t}(x,v)|
	&\lesssim
	\langle v\rangle^{-3+k-|\beta|}
	|J_{s,t}(x,v)-1|+\langle v\rangle^{-4+k-|\beta|}
	\Big(|\Psi_{s,t}(x,v)-v|+|W^\sharp_{s,t}(x,v)|
	\Big).\label{sigma-basic}
	\end{align}
	The bounds \eqref{q3} follow from \eqref{sigma-Psi} and \eqref{sigma-W}.

	To prove \eqref{q10}, we differentiate
	\eqref{sigma-Q-identity} with respect to $v$ and use
	\eqref{sigma-Psi} and \eqref{sigma-Wv}, thus
	\begin{equation}\label{sigma-Vv}
	\begin{split}
	&\left|\nabla_v\left[V_{s,t}\bigl(x,\Psi_{s,t}(x,v)\bigr)\right]-I\right|\\
	&\qquad\lesssim M_T\left(\langle v\rangle^{-1}\langle s\rangle^{-\frac16-\kappa_0}+
	\langle s,|v|^{\frac32}\rangle^{-\frac13+\kappa_0}\right)\langle s\rangle^{-1}+
	M_T\langle v\rangle\mathfrak a_2(s,v)(t-s).
	\end{split}
	\end{equation}
	In particular,
	\begin{align}
	\left|
	\nabla_v\left[
	V_{s,t}\bigl(x,\Psi_{s,t}(x,v)\bigr)
	\right]
	\right|
	\lesssim
	1+M_T\langle v\rangle\mathfrak a_2(s,v)(t-s).
	\label{sigma-Vv-rough}
	\end{align}
We differentiate also the definition \eqref{sec5eq7} of $\Sigma_{s,t}$. Using the mean value theorem, \eqref{sigma-vel-decay}, \eqref{sigma-Psi}, and the comparabilities $\langle\Psi_{s,t}(x,v)\rangle\approx \left\langle V_{s,t}\bigl(x,\Psi_{s,t}(x,v)\bigr)\right\rangle\approx\langle v\rangle$, we obtain
	\begin{align*}
	|\nabla_v\Sigma_{s,t}(x,v)|
	&\lesssim
	\langle v\rangle^{-3+k-|\beta|}
	|\nabla_vJ_{s,t}(x,v)|\\
	&+\langle v\rangle^{-4+k-|\beta|}
	|J_{s,t}(x,v)-1|
	{\left(1+\left|\nabla_v\left[V_{s,t}\bigl(x,\Psi_{s,t}(x,v)\bigr)\right]\right|\right)}\\
	&+\langle v\rangle^{-4+k-|\beta|}\left(|\nabla_v\Psi_{s,t}(x,v)-I|
	\right.\left.+\left|
	\nabla_v\left[
	V_{s,t}\bigl(x,\Psi_{s,t}(x,v)\bigr)
	\right]-I
	\right|
	\right)\\
	&+\langle v\rangle^{-5+k-|\beta|}
	\left(|\Psi_{s,t}(x,v)-v|
	\right.\left.+\left|
	V_{s,t}\bigl(x,\Psi_{s,t}(x,v)\bigr)-v
	\right|
	\right)\\
	&\hspace{18mm}\times
	\left(
	1+
	\left|
	\nabla_v\left[
	V_{s,t}\bigl(x,\Psi_{s,t}(x,v)\bigr)
	\right]
	\right|
	\right).
	\end{align*}
	Using \eqref{sigma-Psi}, \eqref{sigma-V},
	\eqref{sigma-Vv}, and \eqref{sigma-Vv-rough}, it follows that
	\begin{align*}
	|\nabla_v\Sigma_{s,t}(x,v)|
	&\lesssim
	M_T\langle v\rangle^{-3+k-|\beta|}
	\mathfrak a_2(s,v)(t-s)\\
	&+M_T\langle v\rangle^{-4+k-|\beta|}
	\left(
	\langle v\rangle^{-1}\langle s\rangle^{-\frac16-\kappa_0}
	+
	\langle s,|v|^{\frac32}\rangle^{-\frac13+\kappa_0}
	\right)
	\langle s\rangle^{-1}\\
	&+M_T^2\langle v\rangle^{-4+k-|\beta|}
	\mathfrak a_0(s,v)\langle s\rangle^{-1}
	\mathfrak a_2(s,v)(t-s).
	\end{align*}
	This gives the bounds \eqref{q10}.

	Finally, we prove \eqref{z2}. The bounds \eqref{bo1},
	\eqref{a23}, \eqref{Ebounds1} give
	\begin{equation}\label{sigma-crude-field}
	|E(r,z)|+|\nabla_xE(r,z)|
	\lesssim
	M_T\langle t\rangle^{50\kappa_0}\langle z\rangle^{-2},
	\qquad 0\leq r\leq t.
	\end{equation}
	For any $x,w\in\R^2$, it follows from \eqref{z80a} that $\big\langle x-(t-r)w+\widetilde Y_{r,t}(x,w)
	\big\rangle\approx\langle x-(t-r)w\rangle$. The integral equations for $\widetilde Y$ and $W$ therefore imply
	\begin{align}
	&\frac{|\widetilde Y_{s,t}(x,w)|}{t-s}+|W_{s,t}(x-tw,w)|\lesssim
	M_T\langle t\rangle^{50\kappa_0}
	\int_0^t\langle x-\tau w\rangle^{-2}\,d\tau.
	\label{sigma-crude-YW}
	\end{align}

	Differentiating the integral equation for $\widetilde Y$ with respect to $w$, we have
	\begin{align*}
	\nabla_w\widetilde Y_{s,t}(x,w)&=\int_s^t(r-s)\nabla_xE\bigl(r,x-(t-r)w+\widetilde Y_{r,t}(x,w)
	\bigr)\bigl(-(t-r)I+\nabla_w\widetilde Y_{r,t}(x,w)\bigr)\,dr.
	\end{align*}
	By \eqref{z80'b}, $|\nabla_w\widetilde Y_{r,t}(x,w)|\lesssim M_T(t-r)$. 
	Combining this with \eqref{sigma-crude-field}, we obtain
	\begin{align}
	\frac{|\nabla_w\widetilde Y_{s,t}(x,w)|}{t-s}&\lesssim M_T\langle t\rangle^{1+50\kappa_0}
	\int_0^t\langle x-\tau w\rangle^{-2}\,d\tau.
	\label{sigma-crude-Yv}
	\end{align}

	We apply \eqref{sigma-crude-YW}--\eqref{sigma-crude-Yv} with
	$w=\Psi_{s,t}(x,v)$. By \eqref{sigma-Psi}, $t|\Psi_{s,t}(x,v)-v|\lesssim M_T$, therefore
	\[
	\int_0^t
	\left\langle
	x-\tau\Psi_{s,t}(x,v)
	\right\rangle^{-2}\,d\tau
	\lesssim
	\int_0^t\langle x-\tau v\rangle^{-2}\,d\tau.
	\]
	Differentiation of the fixed-point identity \eqref{z16a-new} gives
	\[
	\nabla_v\Psi_{s,t}(x,v)
	=
	\left(
	I-\frac{1}{t-s}
	\nabla_w\widetilde Y_{s,t}
	\bigl(x,\Psi_{s,t}(x,v)\bigr)
	\right)^{-1}.
	\]
	Consequently, using \eqref{sigma-crude-YW}--\eqref{sigma-crude-Yv} and \eqref{z16a-new}
	\begin{align}
	&|\Psi_{s,t}(x,v)-v|+|W^\sharp_{s,t}(x,v)|+|J_{s,t}(x,v)-1|\lesssim
	M_T\langle t\rangle^{1+50\kappa_0}
	\int_0^t\langle x-\tau v\rangle^{-2}\,d\tau.
	\label{sigma-crude}
	\end{align}
	Finally, since $k-|\beta|\leq1$, $\langle v\rangle^{-3+k-|\beta|}+\langle v\rangle^{-4+k-|\beta|}\lesssim
	\langle v\rangle^{-2}$, and the desired bounds \eqref{z2} follow from \eqref{sigma-basic} and \eqref{sigma-crude}.
\end{proof}

The same argument shows that we can restrict the symbol $b$ to high velocities:

\begin{remark}[Large-velocity bounds for $\Sigma_{s,t}$]
	Assume that
	\[
	b(s,v)=\big[1-\chi(|v|/\langle s\rangle)\big]v^{\otimes k},
	\qquad
	\omega(v)=\partial_v^\beta\mu(v),
	\qquad
	k\in\{0,1,2\},
	\qquad
	k-|\beta|\leq1.
	\]
	Then
	\begin{align}
	|\Sigma_{s,t}(x,v)|
	&\lesssim
	M_T\mathbf 1_{\{|v|\geq\langle s\rangle/10\}}
	\langle v\rangle^{-3+k-|\beta|}\big[
	\mathfrak a_1(s,v)(t-s)\langle t\rangle^{-1}+
	\langle v\rangle^{-1}
	\mathfrak a_0(s,v)\langle s\rangle^{-1}
	\big],\label{q3-largev}\\
	|\nabla_v\Sigma_{s,t}(x,v)|
	&\lesssim
	M_T\mathbf 1_{\{|v|\geq\langle s\rangle/10\}}
	\Big[
	\langle v\rangle^{-3+k-|\beta|}
	\mathfrak a_2(s,v)(t-s)
	\notag\\
	&\hspace{25mm}
	+
	\langle v\rangle^{-4+k-|\beta|}
	\Big(
	\langle v\rangle^{-1}\langle s\rangle^{-\frac16-\kappa_0}
	+
	\langle s,|v|^{\frac32}\rangle^{-\frac13+\kappa_0}
	\Big)
	\langle s\rangle^{-1}
	\Big].
	\label{q10-largev}
	\end{align}
	Since $\chi=1$ on $[0,1]$ and
	$|\Psi_{s,t}(x,v)-v|\lesssim M_T\ll1$ by \eqref{sigma-Psi},
	$\Sigma_{s,t}(x,v)=0$ for $|v|<\langle s\rangle/10$.
	Each velocity derivative of the cutoff contributes a factor
	$\langle v\rangle^{-1}$ on its support, so the proof of
	\eqref{q3}--\eqref{q10} gives the stated bounds.
\end{remark}

We prove now several bounds on $\nabla_x^m\mathcal{R}_k(t,x)$, $m=0,1,2$.

	\begin{lemma}\label{Lem1}
	For any $t\in[0,T]$, $x\in\R^2$, and $k=0,1,2$, one has
		\begin{equation}\label{z34b}
			\sum_{m=0}^2
			|\nabla_x^m\mathcal R_k(t,x)|
			\lesssim
			M_T^2\,t^{\frac14}\langle t\rangle^4
			\langle x\rangle^{-\frac{11}{4}}.
		\end{equation}
	\end{lemma}

\begin{proof}
	Using \eqref{bo1}, \eqref{a23}, and the decomposition \eqref{de1}, we have
	\begin{equation}\label{a2}
		\sum_{j=0}^2|\nabla_x^jE(t,x)|
		\lesssim
		M_T\langle t\rangle^{50\kappa_0}\langle x\rangle^{-2}.
	\end{equation}

	Recall that
	\[
	\mathcal R_k(t,x)
	=
	\sum_{i=1}^2
	\mathcal N\bigl(E_i,\partial_{v_i}\mu,v^{\otimes k}\bigr)(t,x).
	\]
	Since $k-1\leq1$, Lemma~\ref{lem:Sigma-new}, \eqref{q1}, and
	\eqref{a2} show that, for $m\in\{0,1,2\}$,
	\begin{align*}
	&\left|
	\mathcal N\bigl(\nabla_x^mE_i,\partial_{v_i}\mu,
	v^{\otimes k}\bigr)(t,x)
	\right|\\
	&\quad\lesssim
	M_T\langle t\rangle^{1+50\kappa_0}
	\int_0^t\int_{\mathbb R^2}
	|\nabla_x^mE_i(s,x-(t-s)v)|
	\left(\int_0^t\langle x-\tau v\rangle^{-2}\,d\tau\right)
	\frac{dv\,ds}{\langle v\rangle^2}\\
	&\quad\lesssim
	M_T^2\langle t\rangle^{1+100\kappa_0}
	\int_{\mathbb R^2}
	\left(\int_0^t\langle x-\tau v\rangle^{-2}\,d\tau\right)^2
	\frac{dv}{\langle v\rangle^2}.
	\end{align*}

	It remains only to control the terms in which an $x$-derivative falls
	on the characteristics. Differentiating the integral equations for
	$Y_{s,t}$ and $W_{s,t}$ at most twice in $x$, and using
	\eqref{a2} and \eqref{n13}--\eqref{n15}, we obtain
	\begin{equation}\label{sec5eq4}
	\sum_{j=0}^2
	\left(
	|\nabla_x^jY_{s,t}(x-tv,v)|
	+
	|\nabla_x^jW_{s,t}(x-tv,v)|
	\right)
	\lesssim
	M_T\langle t\rangle^{1+50\kappa_0}
	\int_0^t\langle x-\tau v\rangle^{-2}\,d\tau.
	\end{equation}
	Indeed, the factors $I+\nabla_xY_{\tau,t}$ are uniformly bounded,
	while all the remaining derivatives of $Y_{\tau,t}$ are
	$O(M_T)$.

	When $\mathcal R_k$ is differentiated at most twice, its principal
	terms are precisely the expressions estimated above. Every remaining
	term contains at least one factor
	$\nabla_x^jY_{s,t}$ or $\nabla_x^jW_{s,t}$, with $1\leq j\leq2$, and can be estimated using \eqref{a2}, \eqref{sec5eq4}, and the decay of $\mu$. Thus
	\begin{equation}\label{Rk-crude-master}
	\sum_{m=0}^2|\nabla_x^m\mathcal R_k(t,x)|
	\lesssim
	M_T^2\langle t\rangle^{1+100\kappa_0}
	\int_{\mathbb R^2}
	\left(
	\int_0^t\langle x-\tau v\rangle^{-2}\,d\tau
	\right)^2
	\frac{dv}{\langle v\rangle^2}.
	\end{equation}

	It remains to prove
	\begin{equation}\label{a40}
	\int_{\mathbb R^2}
	\left(
	\int_0^t\langle x-\tau v\rangle^{-2}\,d\tau
	\right)^2
	\frac{dv}{\langle v\rangle^2}
	\lesssim
	t^{\frac14}\langle t\rangle^2\langle x\rangle^{-\frac{11}{4}}.
	\end{equation}
	We split the velocity space according to whether
	$|x|\geq2(1+t)|v|$ or $|x|\leq2(1+t)|v|$. In the first region,
	$|x-\tau v|\gtrsim|x|$ for $0\leq\tau\leq t$, and hence
	\[
	\int_{\{|x|\geq2(1+t)|v|\}}
	\left(
	\int_0^t\langle x-\tau v\rangle^{-2}\,d\tau
	\right)^2
	\frac{dv}{\langle v\rangle^2}
	\lesssim
	t^2\langle x\rangle^{-4}\log(2+|x|)
	\lesssim
	t^{\frac14}\langle t\rangle^2
	\langle x\rangle^{-\frac{11}{4}}.
	\]

	We notice that for every $x,v\in\R^2$, $v\neq0$,
	\begin{equation}\label{a40.1}
	\int_0^t\langle x-\tau v\rangle^{-2}\,d\tau
	\lesssim
	\min\left\{t,\,
	|v|^{-1}\langle D_v(x)\rangle^{-\frac12}\right\},\qquad
	D_v(x):=|x|^2-\frac{(x\cdot v)^2}{|v|^2}.
	\end{equation}
	Interpolating the two bounds gives
	\[
	\left(
	\int_0^t\langle x-\tau v\rangle^{-2}\,d\tau
	\right)^2
	\lesssim
	t^{\frac14}|v|^{-\frac74}
	\langle D_v(x)\rangle^{-\frac78}
	\lesssim
	t^{\frac14}|v|^{-\frac74}
	\langle D_v(x)\rangle^{-\frac34}.
	\]
	Using polar coordinates $v=r\theta$ and
	\begin{equation}\label{a40.2}
	\int_{\mathbb S^1}
	\langle D_\theta(x)\rangle^{-\frac34}\,d\theta
	\lesssim\langle x\rangle^{-1},
	\end{equation}
	we obtain
	\begin{align*}
	&\int_{\{|x|\leq2(1+t)|v|\}}
	\left(
	\int_0^t\langle x-\tau v\rangle^{-2}\,d\tau
	\right)^2
	\frac{dv}{\langle v\rangle^2}\lesssim
	t^{\frac14}\langle x\rangle^{-1}
	\int_{r\gtrsim |x|/\langle t\rangle}
	r^{-\frac34}\langle r\rangle^{-2}\,dr
	\lesssim
	t^{\frac14}\langle t\rangle^2
	\langle x\rangle^{-\frac{11}{4}}.
	\end{align*}
	This proves \eqref{a40}. Combining \eqref{Rk-crude-master} and
	\eqref{a40}, and using $3+100\kappa_0<4$, gives \eqref{z34b}.
\end{proof}

By \eqref{z34b}, \eqref{q43}--\eqref{q45} hold for
$0\leq t\leq1$ or $|x|\geq\langle t\rangle^{10}$.
For $t\leq1$, \eqref{q46} follows from the integrable, exponentially
decaying kernel of $\nabla(1-\Delta)^{-1}$, and \eqref{q47}
from two integrations by parts.
It remains to prove \eqref{q43}--\eqref{q45} in $\mathcal D_T$;
for $t\geq1$, \eqref{q46}--\eqref{q47} follow from the global
decomposition in Subsection~\ref{secondorderbound}.

	\subsection{Zeroth-order bound for $\mathcal N$}\label{zerobound}
	
	Throughout this subsection, we assume that $k\in\{0,1,2\}$ and $k-|\beta|\le 1$. We use the
	representation \eqref{q1}, the symbol bounds \eqref{q3}--\eqref{q10}, and
	the transport comparison inequality \eqref{z36}.
	
	\subsubsection{A basic estimate}
	For simplicity of notation, we sometimes write
	\begin{equation*}
		F(s):=F(s,x-(t-s)v).
	\end{equation*}
	From \eqref{q1}, \eqref{q3}, and \eqref{aweights}, we obtain
	\begin{equation}\label{x103}
	\begin{split}
		&\left|\mathcal N(F,\partial_v^\beta\mu,v^{\otimes k})(t,x)\right|\\
		&\lesssim M_T
		\int_0^t\int_{\mathbb R^2}
		|F(s,x-(t-s)v)|	\left(\langle v\rangle^{-1}\langle s\rangle^{-\frac76-\kappa_0}+
		\langle s,|v|^{\frac32}\rangle^{-\frac23+50\kappa_0}
		\langle s\rangle^{-\frac23}\right)
		\frac{dv\,ds}{\langle v\rangle^{2}}.
	\end{split}
	\end{equation}
	In particular, applying \eqref{x103} with $F=E$ and using Lemma~\ref{Le3},
	we obtain
	\begin{equation}\label{q14}
		\left|\mathcal N(E,\partial_v^\beta\mu,v^{\otimes k})(t,x)\right|
		\lesssim
		M_T^2 \log(2+t)\langle t,x\rangle^{-2}\lesssim
		M_T^2\langle t\rangle^{20\kappa_0}\Phi(t,x),
	\end{equation}
	for $(t,x)\in\mathcal{D}_T$. This gives the desired estimate \eqref{q43}.
	
	\subsubsection{Higher-order derivatives} To estimate higher spatial derivatives of $\mathcal N$, we need to
	control expressions of the form
	\begin{equation*}
		\mathcal N(F',\partial_v^\beta\mu,v^{\otimes k}),
		\qquad
		F'=s^{|\alpha|}\partial_x^\alpha F.
	\end{equation*}
	For a multi-index $\alpha$, recall the convention
	\begin{equation}\label{sec5eq20}
		D_x^\alpha F(s,\cdot)
		:=s^{|\alpha|}\partial_x^\alpha F(s,\cdot)
		\quad (|\alpha|\geq1),~~
		D_jF:=D_x^{e_j}F=s\partial_{x_j}F.
	\end{equation}
	Applying \eqref{x103}, \eqref{bo1}, and Lemma~\ref{Le3}, we get
	\begin{align*}
		\left|\mathcal N(D_jE^{reg},\partial_v^\beta\mu,v^{\otimes k})(t,x)\right|
		&\lesssim M_T\int_0^t\int_{\mathbb R^2}
		|\nabla_xE^{reg}(s,x-(t-s)v)|
		\langle s\rangle^{-\frac16-\kappa_0}
		\frac{dv\,ds}{\langle v\rangle^{2+2\kappa_0}}\\
		&\lesssim M_T^2
		\int_0^t\int_{\mathbb R^2}
		\langle s,x-(t-s)v\rangle^{-2-\kappa_0}
		\langle s\rangle^{-\frac43-\kappa_0}
		\frac{dv\,ds}{\langle v\rangle^{2+2\kappa_0}}\\
		&\lesssim M_T^2
		\langle t,x\rangle^{-2-\kappa_0}
		\langle t\rangle^{\kappa_0}.
	\end{align*}
	Similarly,
	\begin{align*}
		\left|\mathcal N(D_jD_{j'}E^{reg},\partial_v^\beta\mu,v^{\otimes k})(t,x)\right|
		&\lesssim M_T
		\int_0^t\int_{\mathbb R^2}
		|\nabla_x^2E^{reg}|(s,x-(t-s)v)
		\langle s\rangle^{\frac56-\kappa_0}
		\frac{dv\,ds}{\langle v\rangle^{2+2\kappa_0}}\\
		&\lesssim M_T^2\langle t,x\rangle^{-2-\kappa_0}
		\langle t\rangle^{\kappa_0}.
	\end{align*}
	Consequently,
	\begin{align}
		&\left|
		\mathcal N(D_jE^{reg},\partial_v^\beta\mu,v^{\otimes k})(t,x)
		\right|
		+
		\left|
		\mathcal N(D_jD_{j'}E^{reg},\partial_v^\beta\mu,v^{\otimes k})(t,x)
		\right|\lesssim
		M_T^2
		\langle t,x\rangle^{-2-\kappa_0}
		\langle t\rangle^{\kappa_0}.
		\label{q5}
	\end{align}
	
Let $b_{hi}(s,v):=\left(1-\chi\left(\frac{|v|}{\langle s\rangle}\right)\right)v^{\otimes k}.$
	Using the large-velocity version of the symbol bound, namely
	\eqref{q3-largev}, and arguing as in the proof of \eqref{x103}, we obtain
	\begin{equation}\label{largev-start}
	\begin{split}
		&\left|\mathcal N(F,\partial_v^\beta\mu,b_{hi})(t,x)\right|\lesssim M_T\int_0^t
		\int_{\{|v|\ge \langle s\rangle/10\}}
		|F(s,x-(t-s)v)|\\
		&\qquad\qquad\qquad\qquad\times\left(\langle v\rangle^{-1}
		\langle s\rangle^{-\frac76-\kappa_0}+
		\langle s,|v|^{\frac32}\rangle^{-\frac23+50\kappa_0}
		\langle s\rangle^{-\frac23}\right)
		\frac{dv\,ds}{\langle v\rangle^2}.
	\end{split}
	\end{equation}
	On the region $|v|\ge \langle s\rangle/10$, we have
	\begin{equation*}
		\langle v\rangle^{-1}
		\langle s\rangle^{-\frac76-\kappa_0}
		+
		\langle s,|v|^{\frac32}\rangle^{-\frac23+50\kappa_0}
		\langle s\rangle^{-\frac23}
		\lesssim
		\langle s\rangle^{-\frac53+100\kappa_0}
		\langle v\rangle^{-2\kappa_0}.
	\end{equation*}
	Therefore \eqref{largev-start} implies
	\begin{equation}\label{q6}
		\left|
		\mathcal N(F,\partial_v^\beta\mu,b_{hi})(t,x)\right|
		\lesssim
		M_T\int_0^t\int_{\mathbb R^2}|F(s,x-(t-s)v)|\langle s\rangle^{-\frac53+100\kappa_0}
		\frac{dv\,ds}{\langle v\rangle^{2+2\kappa_0}}.
	\end{equation}

	Assume now that $F$ has the oscillatory form
	$F(s,x)=e^{\ast is}F_*(s,x),$
	where $\ast\in\{+,-\}$. We consider the small-velocity contribution
	\begin{align}
		&\mathcal{N}\left(F,\partial_{v}^\beta\mu,\chi\left(\frac{|v|}{\langle s\rangle}\right)v^{\otimes k}\right)(t,x)
		\notag\\
		&=\int_0^t\int_{\mathbb R^2} e^{\ast is}
		\chi\left(\frac{|v|}{\langle s\rangle}\right)
		v^{\otimes k}
		\Big\{
		F_*(s,x-(t-s)v)\partial_v^\beta\mu(v)-F_*(s,X_{s,t})\partial_v^\beta\mu(V_{s,t})
		\Big\}\,dv\,ds.
		\label{small-v-N}
	\end{align}
	The bracket in \eqref{small-v-N} vanishes at $s=t$. Therefore, integrating
	by parts in $s$, we obtain
	\begin{equation}		\label{small-v-decomp}
	\begin{split}
		&\left|\mathcal N\left(F,\partial_v^\beta\mu,
		\chi\left(\frac{|v|}{\langle s\rangle}\right)v^{\otimes k}\right)(t,x)
		\right|\lesssim I_1+I_2+I_3+I_4+I_5+I_6,\\
		&I_1:=\left|
		\int_{\mathbb R^2}
		\chi(|v|)v^{\otimes k}
		\Big\{F_*(0,x-tv)\partial_v^\beta\mu(v)-
		F_*(0,X_{0,t})\partial_v^\beta\mu(V_{0,t})
		\Big\}\,dv\right|,\\
		&I_2:=\int_0^t\left|\int_{\mathbb R^2}\partial_s\left[
		\chi\left(\frac{|v|}{\langle s\rangle}\right)\right]
		v^{\otimes k}\Big\{F_*(s,x-(t-s)v)\partial_v^\beta\mu(v)-F_*(s,X_{s,t})\partial_v^\beta\mu(V_{s,t})
		\Big\}\,dv\right|\,ds,\\
		&I_3:=\int_0^t\left|
		\int_{\mathbb R^2}
		\chi\left(\frac{|v|}{\langle s\rangle}\right)
		v^{\otimes k}\Big\{\partial_sF_*(s,x-(t-s)v)\partial_v^\beta\mu(v)-
		\partial_sF_*(s,X_{s,t})\partial_v^\beta\mu(V_{s,t})\Big\}\,dv\right|\,ds,\\
		&I_4:=\int_0^t\left|\int_{\mathbb R^2}
		\chi\left(\frac{|v|}{\langle s\rangle}\right)v^{\otimes k}v\cdot
		\Big\{\nabla_xF_*(s,x-(t-s)v)\partial_v^\beta\mu(v)-\nabla_xF_*(s,X_{s,t})\partial_v^\beta\mu(V_{s,t})
		\Big\}\,dv\right|ds,\\
		&I_5:=\int_0^t\int_{\mathbb R^2}
		\chi\left(\frac{|v|}{\langle s\rangle}\right)
		|v|^k
		|\nabla_xF_*(s,X_{s,t})|
		|\partial_sY_{s,t}|
		|\partial_v^\beta\mu(V_{s,t})|\,dv\,ds,\\
		&I_6:=\int_0^t\int_{\mathbb R^2}
		\chi\left(\frac{|v|}{\langle s\rangle}\right)|v|^k|F_*(s,X_{s,t})|
		|\partial_v^\beta\nabla_v\mu(V_{s,t})|
		|\partial_sW_{s,t}|\,dv\,ds.
	\end{split}
	\end{equation}
	
	We now estimate these six terms. Using \eqref{q1}, \eqref{x103}, and \eqref{z36}, we obtain
	\begin{equation}\label{I123-bound}
	\begin{split}
	&I_1\lesssim M_T\int_{|v|\lesssim 1}|F_*(0,x-tv)|\frac{dv}{\langle v\rangle^2},\\
	&I_2+I_3\lesssim M_T\int_0^t\int_{\mathbb R^2}
		\left(\langle s\rangle^{-1}|F_*(s,x-(t-s)v)|+
		|\partial_sF_*(s,x-(t-s)v)|\right)\\
		&\qquad\qquad\qquad\times\left(
		\langle v\rangle^{-1}\langle s\rangle^{-\frac76-\kappa_0}+
		\langle s,|v|^{\frac32}\rangle^{-\frac23+50\kappa_0}
		\langle s\rangle^{-\frac23}\right)
		\frac{dv\,ds}{\langle v\rangle^2}.
	\end{split}
	\end{equation}
	For the transport derivative term, using the support of the cutoff
	$|v|\lesssim\langle s\rangle$, we get
	\begin{align}
	I_4&\lesssim M_T\int_0^t
		\int_{\{|v|\le 4\langle s\rangle\}}
		|\nabla_xF_*(s,x-(t-s)v)|\left(
		\langle v\rangle^{-1}\langle s\rangle^{-\frac76-\kappa_0}+
		\langle s,|v|^{\frac32}\rangle^{-\frac23}
		\langle s\rangle^{-\frac23+\kappa_0}
		\right)
		\frac{dv\,ds}{\langle v\rangle}
		\notag\\
		&\lesssim
		M_T\int_0^t\int_{\mathbb R^2}|\nabla_xF_*(s,x-(t-s)v)|
		\langle s\rangle^{-\frac23+3\kappa_0}
		\frac{dv\,ds}{\langle v\rangle^{2+2\kappa_0}}.
		\label{q20}
	\end{align}
	Moreover, by the characteristic estimates \eqref{n13}
	\begin{align}
		I_5&\lesssim M_T
		\int_0^t\int_{\mathbb R^2}
		|\nabla_xF_*(s,x-(t-s)v)|\mathfrak{a}_0(s,v)
		\langle s\rangle^{-1}
		\frac{dv\,ds}{\langle v\rangle^2}.
		\label{I5-bound}
	\end{align}
	Finally, since
	\begin{equation*}
		|\partial_sW_{s,t}(x-tv,v)|=|E(s,x-(t-s)v+Y_{s,t})|\lesssim M_T
		\langle s,x-(t-s)v\rangle^{-2}
		\langle s\rangle^{50\kappa_0},
	\end{equation*}
	we have
	\begin{align}
		I_6&\lesssim M_T
		\int_0^t\int_{\mathbb R^2}
		|F_*(s,x-(t-s)v)|\langle s,x-(t-s)v\rangle^{-2}
		\langle s\rangle^{50\kappa_0}
		\frac{dv\,ds}{\langle v\rangle^3}.
		\label{I6-bound}
	\end{align}
	
	Combining the small-velocity estimates \eqref{I123-bound}--\eqref{I6-bound} with the
	large-velocity contribution \eqref{q6}, we obtain, for
	$F(s,x)=e^{\ast is}F_*(s,x)$,
	\begin{equation}\label{q21}
	\begin{split}
		\left|\mathcal N(F,\partial_v^\beta\mu,v^{\otimes k})(t,x)\right|
		&\lesssim M_T\int_{\mathbb R^2}
		|F_*(0,x-tv)|\frac{dv}{\langle v\rangle^3}\\
		&+M_T\int_0^t\int_{\mathbb R^2}
		|\partial_sF_*(s,x-(t-s)v)|\langle s\rangle^{-\frac76-\kappa_0}
		\frac{dv\,ds}{\langle v\rangle^{2+2\kappa_0}}\\
		&+M_T\int_0^t\int_{\mathbb R^2}|\nabla_xF_*(s,x-(t-s)v)|\langle s\rangle^{-\frac23+3\kappa_0}
		\frac{dv\,ds}{\langle v\rangle^{2+2\kappa_0}}\\
		&+M_T\int_0^t\int_{\mathbb R^2}
		|F_*(s,x-(t-s)v)|\langle s\rangle^{-\frac53+100\kappa_0}
		\frac{dv\,ds}{\langle v\rangle^{2+2\kappa_0}}.
	\end{split}
	\end{equation}
	
\subsubsection{An improved estimate for the $I_4$ term} For $t\geq1$, we sharpen the estimate of $I_4$ by integrating by parts
in $v$ away from the endpoint $s=t$. We split	
\begin{align}
		I_4&\lesssim M_T\int_{t-\frac12}^t\int_{\mathbb R^2}
		|\nabla_xF_*(s,x-(t-s)v)|\langle s\rangle^{-\frac23+3\kappa_0}
		\frac{dv\,ds}{\langle v\rangle^{2+2\kappa_0}}\notag\\
		&\quad+\int_0^{t-\frac12}
		\left|\int_{\mathbb R^2}
		\nabla_xF_*(s,x-(t-s)v)\cdot
		\Sigma'_{s,t}(x,v)\,dv
		\right|\,ds.
		\label{I4-split}
	\end{align}
	Here
	\begin{align}
		\Sigma'_{s,t}(x,v)&:=
		\chi\left(\frac{|v|}{\langle s\rangle}\right)
		v^{\otimes(k+1)}
		\partial_v^\beta\mu(v)
		\notag\\
		&-J_{s,t}(x,v)
		\chi\left(\frac{|\Psi_{s,t}(x,v)|}{\langle s\rangle}\right)
		\Psi_{s,t}(x,v)^{\otimes(k+1)}
		\partial_v^\beta\mu
		\left(
		V_{s,t}(x,\Psi_{s,t}(x,v))
		\right).
		\label{Sigma-prime-def}
	\end{align}
	For $s\le t-\frac12$, we use the identity
	\begin{equation*}
		\nabla_xF_*(s,x-(t-s)v)=-\frac{1}{t-s}
		\nabla_v\bigl[F_*(s,x-(t-s)v)\bigr].
	\end{equation*}
	Integrating by parts in $v$, we get
	\begin{align}
		&\int_0^{t-\frac12}\left|
		\int_{\mathbb R^2}
		\nabla_xF_*(s,x-(t-s)v)\cdot
		\Sigma'_{s,t}(x,v)\,dv\right|\,ds
		\notag\\
		&\qquad\lesssim
		\int_0^{t-\frac12}
		\int_{\mathbb R^2}
		|F_*(s,x-(t-s)v)|
		\frac{|\nabla_v\Sigma'_{s,t}(x,v)|}{t-s}\,dv\,ds .
		\label{I4-ibp-v}
	\end{align}
	The boundary term in this identity vanishes because $\Sigma'_{s,t}(x,\cdot)$ is compactly supported. 
	
	Let $b'(s,v):=\chi\big(|v|/\langle s\rangle\big)v^{\otimes(k+1)}$, so, for $m\in\{0,1,2\}$,
	\begin{equation}
		\left|
		\nabla_v^m\bigl(b'(s,v)\partial_v^\beta\mu(v)\bigr)
		\right|
		\lesssim\mathbf 1_{\{|v|\leq4\langle s\rangle\}}\langle v\rangle^{-2+k-|\beta|-m}.
		\label{Sigma-prime-symbol-bound}
	\end{equation}
The proof of \eqref{q10} uses only velocity derivative
bounds through order two. Each derivative of the cutoff contributes
$\langle v\rangle^{-1}$ on its support, so \eqref{Sigma-prime-symbol-bound}
allows the same argument with $k$ replaced by $k+1$, giving
	\begin{align*}
		\frac{|\nabla_v\Sigma'_{s,t}(x,v)|}{t-s}&\lesssim M_T\mathbf 1_{\{|v|\leq4\langle s\rangle\}}
		\langle v\rangle^{-1}\langle s^{\frac12+2\kappa_0},v\rangle^{-1}
		\langle s\rangle^{-\frac43+\kappa_0}\\
		&+M_T\mathbf 1_{\{|v|\leq4\langle s\rangle\}}\langle v\rangle^{-2}\left(
		\langle v\rangle^{-1}\langle s\rangle^{-\frac16-\kappa_0}+
		\langle s,|v|^{\frac32}\rangle^{-\frac13+\kappa_0}\right)
		\langle s\rangle^{-1}(t-s)^{-1}.
	\end{align*}
	On the support of the cutoff, we have
	\begin{align*}
		\langle v\rangle^{-1}
		\langle s^{\frac12+2\kappa_0},v\rangle^{-1}&\lesssim
		\langle s\rangle^{2\kappa_0}
		\langle v\rangle^{-2-2\kappa_0},\\
		\langle v\rangle^{-2}\left(
		\langle v\rangle^{-1}\langle s\rangle^{-\frac16-\kappa_0}+
		\langle s,|v|^{\frac32}\rangle^{-\frac13+\kappa_0}\right)
		&\lesssim
		\langle s\rangle^{-\frac16-\kappa_0}
		\langle v\rangle^{-2-2\kappa_0}.
	\end{align*}
	Therefore, if $t-s\geq1/2$ then the last three bounds show that
	\begin{equation}
		\frac{|\nabla_v\Sigma'_{s,t}(x,v)|}{t-s}
		\lesssim
		M_T\left(
		\langle s\rangle^{-\frac43+3\kappa_0}
		+
		\langle s\rangle^{-\frac76-\kappa_0}
		\langle t-s\rangle^{-1}
		\right)
		\langle v\rangle^{-2-2\kappa_0}.
		\label{Sigma-prime-v-bound}
	\end{equation}
	Combining \eqref{I4-split}, \eqref{I4-ibp-v}, and \eqref{Sigma-prime-v-bound} we have
	\begin{align}
		I_4&\lesssim
		M_T\int_0^t\int_{\mathbb R^2}|F_*(s,x-(t-s)v)|
		\left(\langle s\rangle^{-\frac43+3\kappa_0}+
		\langle s\rangle^{-\frac76-\kappa_0}
		\langle t-s\rangle^{-1}\right)
		\frac{dv\,ds}{\langle v\rangle^{2+2\kappa_0}}
		\notag\\
		&+M_T\int_{t-\frac12}^t\int_{\mathbb R^2}
		|\nabla_xF_*(s,x-(t-s)v)|
		\langle s\rangle^{-\frac23+3\kappa_0}
		\frac{dv\,ds}{\langle v\rangle^{2+2\kappa_0}}.
		\label{I4-improved}
	\end{align}
	
	Substituting \eqref{I4-improved} into the proof of \eqref{q21}, we obtain
	\begin{equation}\label{q21-improved}
	\begin{split}
		&\left|\mathcal N(F,\partial_v^\beta\mu,v^{\otimes k})(t,x)\right|
		\lesssim
		M_T\int_{\mathbb R^2}
		|F_*(0,x-tv)|\frac{dv}{\langle v\rangle^3}\\
		&\qquad+M_T\int_0^t\int_{\mathbb R^2}
		|\partial_sF_*(s,x-(t-s)v)|
		\langle s\rangle^{-\frac76-\kappa_0}
		\frac{dv\,ds}{\langle v\rangle^{2+2\kappa_0}}\\
		&\qquad+M_T\int_0^t\int_{\mathbb R^2}
		|\nabla_xF_*(s,x-(t-s)v)|\mathfrak{a}_0(s,v)
		\langle s\rangle^{-1}
		\frac{dv\,ds}{\langle v\rangle^2}\\
		&\qquad+M_T\int_{t-\frac12}^t\int_{\mathbb R^2}
		|\nabla_xF_*(s,x-(t-s)v)|
		\langle s\rangle^{-\frac23+3\kappa_0}
		\frac{dv\,ds}{\langle v\rangle^{2+2\kappa_0}}\\
		&\qquad+M_T\int_0^t\int_{\mathbb R^2}
		|F_*(s,x-(t-s)v)|\left(\langle s\rangle^{-\frac43+3\kappa_0}+
		\langle s\rangle^{-\frac76-\kappa_0}
		\langle t-s\rangle^{-1}\right)
		\frac{dv\,ds}{\langle v\rangle^{2+2\kappa_0}}.
	\end{split}
	\end{equation}
	
	\subsubsection{Bounds on the derivatives of $E^{osc}$} We prove the following:
	
	\begin{lemma}\label{lem1Sec5}
	For any $t\in[0,T]$, $x\in\mathbb R^2$, and $j,j'\in\{1,2\}$ we have
	\begin{align}
		&\left|\mathcal N(D_jE^{osc},\partial_v^\beta\mu,v^{\otimes k})(t,x)\right|\lesssim
		M_T^2\langle t,x\rangle^{-2-\kappa_0}
		\langle t\rangle^{5\kappa_0},\label{q8}\\
		&\left|\mathcal N(D_jD_{j'}E^{osc},\partial_v^\beta\mu,v^{\otimes k})(t,x)\right|\lesssim
	M_T^2\langle t\rangle^{\frac13+5\kappa_0}\Phi(t,x).\label{q11}
\end{align}
	\end{lemma}

\begin{proof} We first apply \eqref{q21} to
	$F=D_jE^{osc}$. Since $D_j$ contains a factor of $s$, the
	boundary term at $s=0$ vanishes. Using Lemma~\ref{Le3} and the bounds \eqref{Ebounds1}, we obtain
		\begin{align*}
		\left|
		\mathcal N(D_jE^{osc},\partial_v^\beta\mu,v^{\otimes k})(t,x)
		\right|
		&\lesssim
		M_T^2
		\int_0^t\int_{\mathbb R^2}
		\langle s,x-(t-s)v\rangle^{-2-\kappa_0}
		\langle s\rangle^{-1+5\kappa_0}
		\frac{dv\,ds}{\langle v\rangle^{2+2\kappa_0}}
		\notag\\
		&\lesssim
		M_T^2
		\langle t,x\rangle^{-2-\kappa_0}
		\langle t\rangle^{5\kappa_0}.
	\end{align*}

It remains to prove \eqref{q11}. For $\Box\in\{\cos,\sin\}$, we decompose
\[
	D_jD_{j'}E^{osc}(s,x)=
	\cos(s)F^{\cos}(s,x)+\sin(s)F^{\sin}(s,x),\qquad 	F^\Box(s,x):=s^2\partial_{x_jx_{j'}}E^\Box(s,x).
\]
It follows from \eqref{Ebounds1} and \eqref{aweights} that
\begin{equation}\label{sec5eq21}
\begin{split}
	&|F^\Box(s,x)|\lesssim M_T\langle s\rangle^{\frac23+2\kappa_0}\Phi(s,x),\\
	&|\partial_sF^\Box(s,x)|+|\nabla_xF^\Box(s,x)|\lesssim M_T\langle s\rangle^{\frac16}\Phi(s,x),\\
	&\mathfrak a_0(s,v)\langle s\rangle^{-1}\langle v\rangle^{-2}\lesssim\langle s\rangle^{-\frac76}\langle v\rangle^{-2-2\kappa_0}.
\end{split}
\end{equation}
The bounds \eqref{q11} follow from \eqref{q21},  \eqref{sec5eq21}, and Lemma \ref{Le3} if $t\in[0,2]$. On the other hand, if $t\geq 2$ then we use \eqref{q21-improved}, \eqref{sec5eq21}, and the observation $F^\Box(0,x)=0$ to obtain
\begin{align*}
&\left|
\mathcal N(D_jD_{j'}E^{osc},
\partial_v^\beta\mu,v^{\otimes k})(t,x)
\right|\\
&\qquad\lesssim
M_T^2\int_0^t\int_{\mathbb R^2}
\Phi(s,x-(t-s)v)
\left(\langle s\rangle^{-\frac23+5\kappa_0}+
\langle s\rangle^{-\frac12+\kappa_0}
\langle t-s\rangle^{-1}\right)
\frac{dv\,ds}{\langle v\rangle^{2+2\kappa_0}}\\
&\qquad+
M_T^2\int_{t-\frac12}^t\int_{\mathbb R^2}
\Phi(s,x-(t-s)v)\langle s\rangle^{-\frac12+3\kappa_0}
\frac{dv\,ds}{\langle v\rangle^{2+2\kappa_0}}.
\end{align*}
Notice that if $0\leq s\leq t-t^{\frac14}$ then $\langle s\rangle^{-\frac12+\kappa_0}\langle t-s\rangle^{-1}\lesssim\langle s\rangle^{-\frac23+5\kappa_0}$. In view of Lemma~\ref{Le3} with
$\sigma=-2/3+5\kappa_0$, for \eqref{q11} it remains to prove that if $t\geq 2$ then
\begin{equation}\label{sec5eq22}
\begin{split}
\int_{t-t^{1/4}}^t\int_{\mathbb R^2}&\Phi(s,x-(t-s)v)\langle s\rangle^{-\frac12+\kappa_0}
\langle t-s\rangle^{-1}\frac{dv\,ds}{\langle v\rangle^{2+2\kappa_0}}\\
&+\int_{t-1/2}^t\int_{\mathbb R^2}\Phi(s,x-(t-s)v)\langle s\rangle^{-\frac12+3\kappa_0}
\frac{dv\,ds}{\langle v\rangle^{2+2\kappa_0}}\lesssim \langle t\rangle^{\frac13+5\kappa_0}\Phi(t,x).
\end{split}
\end{equation}

To prove \eqref{sec5eq22} we notice that if $t\geq 2$ and $0\leq r\leq t^{1/4}$ then
\begin{equation}\label{sec5eq23}
	\int_{\mathbb R^2}
	\Phi(t,x-rv)\frac{dv}{\langle v\rangle^{2+2\kappa_0}}
	\lesssim
	\langle r\rangle^{\kappa_0}\Phi(t,x).
\end{equation}
Indeed, if $|x|\leq\langle t\rangle$ then this follows from the time component
of $\Phi(t,x-rv)$. If $|x|\geq\langle t\rangle$, then
$r\lesssim t\lesssim|x|$ and the bounds \eqref{sec5eq23} follow from Lemma~\ref{Le1}.
For $t\geq2$ and $s\in[t-t^{1/4},t]$, we have
$\langle s\rangle\approx\langle t\rangle$, hence $\Phi(s,z)\lesssim\Phi(t,z)$
for all $z\in\mathbb R^2$.
The bounds \eqref{sec5eq22} then follow from \eqref{sec5eq23} and the change of variables $r=t-s$. This completes the proof of the bounds \eqref{q11}.
\end{proof}

\subsection{The first-order bound for $\mathcal N$}\label{firstorderbound}

Throughout this subsection, we assume that $k\in\{0,1,2\}$ and $k-|\beta|\leq1$. We first
prove the desired first-order bound and then record several estimates that
will be used in the proof of the second-order bound.

\subsubsection{Proof of \eqref{q44}} Assume that $t\geq1$. We make the change of variables
\[
v=\frac{x-w}{t},\qquad dv=t^{-2}\,dw,
\]
in the definition of $\mathcal N$ to obtain
\begin{equation}\label{z51}
\begin{split}
&\mathcal N(F,\omega,b)(t,x)
=\frac1{t^2}\int_0^t\!\!\int_{\mathbb R^2}
b\Big(s,\frac{x-w}{t}\Big)
\Big[F\left(s,w+\frac{s}{t}(x-w)\right)
\omega\Big(\frac{x-w}{t}\Big)\\
&\qquad-F\Big(s,w+\frac{s}{t}(x-w)
+Y_{s,t}\Big(w,\frac{x-w}{t}\Big)\Big)\omega\Big(\frac{x-w}{t}
+W_{s,t}\Big(w,\frac{x-w}{t}\Big)\Big)\Big]\,dwds.
\end{split}
\end{equation}
Differentiating with respect to $x_j$ gives
\begin{equation}\label{z60}
\partial_{x_j}\mathcal N(F,\omega,b)=t^{-1}\mathcal N(D_jF,\omega,b)
+t^{-1}\mathcal N(F,\omega,\partial_{v_j}b)
+t^{-1}\mathcal N(F,\partial_{v_j}\omega,b)+t^{-1}\mathcal N_j^{r,1}(F,\omega,b),
\end{equation}
where $D_jF(s)=s\partial_{x_j}F(s)$ as before, and
\begin{align}
\mathcal N_j^{r,1}(F,\omega,b)(t,x)
&:=-\int_0^t\!\!\int_{\mathbb R^2}b(s,v)
\Big[
\partial_{v_j}Y_{s,t}(x-tv,v)\cdot
\nabla_xF(s,X_{s,t}(x,v))\omega(V_{s,t}(x,v))
\notag\\
&+F(s,X_{s,t}(x,v))
\nabla_v\omega(V_{s,t}(x,v))\cdot
\partial_{v_j}W_{s,t}(x-tv,v)
\Big]\,dvds.
\label{z61}
\end{align}

By \eqref{xf'}, \eqref{z36}, and the decay of $\mu$, we have
\begin{align}
\left|\mathcal N_j^{r,1}
(F,\partial_v^\beta\mu,v^{\otimes k})(t,x)\right|&\lesssim
M_T\int_0^t\!\!\int_{\mathbb R^2}
|\nabla_xF(s,x-(t-s)v)|
\langle s\rangle^{-\frac16-\kappa_0}
\frac{dv\,ds}{\langle v\rangle^{2+2\kappa_0}}
\notag\\
&+M_T\int_0^t\!\!\int_{\mathbb R^2}
|F(s,x-(t-s)v)|
\langle s\rangle^{-\frac76-\kappa_0}
\frac{dv\,ds}{\langle v\rangle^3}.
\label{x104}
\end{align}

We apply this estimate first to the two components of the field. From the
bounds for $E^{reg}$ and Lemma~\ref{Le3},
\begin{align}
\left|\mathcal N_j^{r,1}
(E^{reg},\partial_v^\beta\mu,v^{\otimes k})(t,x)\right|
&\lesssim M_T^2\mathbf S_{-\frac43-\kappa_0}(t,x)\lesssim
M_T^2\langle t\rangle^{\kappa_0}\Phi(t,x).
\label{q22}
\end{align}
For the oscillatory component, we separate the linear and nonlinear
contributions in the term containing $E^{osc}$ itself. Using
\eqref{a23}, \eqref{Ebounds1}, and Lemma~\ref{Le3}, we obtain
\begin{align}
\left|\mathcal N_j^{r,1}
(E^{osc},\partial_v^\beta\mu,v^{\otimes k})(t,x)\right|
&\lesssim M_T^2\big(
\mathbf S_{-\frac56+\kappa_0}(t,x)
+\mathbf T_{-\frac76}(t,x)
\big)\lesssim M_T^2\langle t\rangle^{\frac16}\langle t,x\rangle^{-2}.
\label{q19}
\end{align}

We now apply \eqref{z60} with $F=E^{reg}+E^{osc}$,
$\omega=\partial_v^\beta\mu$, and $b=v^{\otimes k}$. The term in which
$D_j$ falls on the field is controlled by \eqref{q5} and \eqref{q8}, the
term containing a $v$-derivative on $\omega$ or $b$ is controlled by
\eqref{q14}, and the remainder is controlled by \eqref{q22} and
\eqref{q19}. Thus
\begin{align}
\left|\nabla_x\mathcal N
(E,\partial_v^\beta\mu,v^{\otimes k})(t,x)\right|
&\lesssim
M_T^2\langle t\rangle^{-\frac56}\langle t,x\rangle^{-2}.
\label{q15}
\end{align}
which proves the desired bounds \eqref{q44} for $(t,x)\in\mathcal D_T$.

\subsubsection{Auxiliary first-order estimates}
We prove now some estimates needed in the second-order argument.
Differentiating the original definition \eqref{defN-new} of $\mathcal N$, and then using
\eqref{x103}, \eqref{n14}, and \eqref{z36}, gives
\begin{align}
&\left|\nabla_x\mathcal N
(F,\partial_v^\beta\mu,v^{\otimes k})(t,x)\right|\lesssim
M_T\int_0^t\!\!\int_{\mathbb R^2}
|\nabla_xF(s,x-(t-s)v)|
\langle s\rangle^{-\frac76-\kappa_0}
\frac{dv\,ds}{\langle v\rangle^{2+2\kappa_0}}
\notag\\
&\qquad+
M_T\int_0^t\!\!\int_{\mathbb R^2}
|F(s,x-(t-s)v)|
\langle s\rangle^{-\frac{13}{6}-\kappa_0}
\frac{dv\,ds}{\langle v\rangle^3}.
\label{q12}
\end{align}
Here we used the elementary bounds
\begin{align*}
\langle s,|v|^{\frac32}\rangle^{-\frac23+50\kappa_0}
\langle s\rangle^{-\frac23}
+\langle s,|v|^{\frac32}\rangle^{-\frac23}
\langle s\rangle^{-\frac23+\kappa_0}
&\lesssim\langle s\rangle^{-\frac76-\kappa_0}
\langle v\rangle^{-2\kappa_0}.
\end{align*}

For $F=D_jD_{j'}E^{osc}$, \eqref{Ebounds1} gives
\begin{align*}
|\nabla_xD_jD_{j'}E^{osc}(s,z)|
&\lesssim M_T\langle s\rangle^{\frac16}\Phi(s,z),\\
|D_jD_{j'}E^{osc}(s,z)|
&\lesssim M_T\langle s\rangle^{\frac23+2\kappa_0}\Phi(s,z).
\end{align*}
Thus \eqref{q12} and Lemma~\ref{Le3} imply
\begin{align}
\left|\nabla_x\mathcal N
(D_jD_{j'}E^{osc},\partial_v^\beta\mu,v^{\otimes k})(t,x)\right|
&\lesssim M_T^2\mathbf S_{-1-\kappa_0}(t,x)\lesssim
M_T^2\langle t\rangle^{\kappa_0}\Phi(t,x).
\label{q13}
\end{align}

We also apply \eqref{x104} to the first derivatives of
the field. The bounds for $E^{reg}$ give
\begin{align}
\left|\mathcal N_j^{r,1}
(D_{j'}E^{reg},\partial_v^\beta\mu,v^{\otimes k})(t,x)\right|
&\lesssim M_T^2\mathbf S_{-1-\kappa_0}(t,x)\lesssim
M_T^2\langle t\rangle^{\kappa_0}\Phi(t,x),
\label{q25}
\end{align}
whereas \eqref{Ebounds1} gives
\begin{align}
\left|\mathcal N_j^{r,1}
(D_{j'}E^{osc},\partial_v^\beta\mu,v^{\otimes k})(t,x)\right|
&\lesssim M_T^2\mathbf S_{-\frac12+\kappa_0}(t,x)\lesssim
M_T^2\langle t\rangle^{\frac12+\kappa_0}\Phi(t,x).
\label{q26}
\end{align}

We differentiate now \eqref{z61} once in $x$. The leading terms are bounded by
\[
|\nabla_x\nabla_vY_{s,t}|\,|\nabla_xF|,\qquad |\nabla_vW_{s,t}|\,|\nabla_xF|,\qquad
|\nabla_vY_{s,t}|\,|\nabla_x^2F|,\qquad |\nabla_x\nabla_vW_{s,t}|\,|F|.
\]
All other terms contain a product of two derivatives of the characteristic corrections and are smaller. Therefore,
using \eqref{n14'}--\eqref{n15'}, and
the comparison estimate \eqref{z36}, we obtain
\begin{equation}\label{q16}
\begin{split}
\left|\nabla_x\mathcal N_j^{r,1}
(F,\partial_v^\beta\mu,v^{\otimes k})(t,x)\right|&\lesssim
M_T\int_0^t\!\!\int_{\mathbb R^2}
|\nabla_x^2F(s,x-(t-s)v)|\langle s\rangle^{-\frac16-\kappa_0}
\frac{dv\,ds}{\langle v\rangle^{2+2\kappa_0}}\\
&+M_T\int_0^t\!\!\int_{\mathbb R^2}
|\nabla_xF(s,x-(t-s)v)|\langle s\rangle^{-\frac56+2\kappa_0}
\frac{dv\,ds}{\langle v\rangle^{2+2\kappa_0}}\\
&+M_T\int_0^t\!\!\int_{\mathbb R^2}|F(s,x-(t-s)v)|
\langle s\rangle^{-\frac{11}{6}+2\kappa_0}
\frac{dv\,ds}{\langle v\rangle^3}.
\end{split}
\end{equation}
Then we split $E^{osc}=E^{osc}_{lin}+E^{osc}_{non}$ as before and use Lemma~\ref{Le3}, so
\begin{align}
\left|\nabla_x\mathcal N_j^{r,1}
(E^{osc},\partial_v^\beta\mu,v^{\otimes k})(t,x)\right|
&\lesssim M_T^2\Big(
\mathbf S_{-\frac32+4\kappa_0}(t,x)
+\mathbf T_{-\frac{11}{6}+2\kappa_0}(t,x)
\Big)
\notag\\
&\lesssim
M_T^2\log(2+t)\langle t,x\rangle^{-2}.
\label{q23}
\end{align}
Similarly, taking $F=D_{j'}E^{osc}$ in \eqref{q16} and using \eqref{Ebounds1}, we obtain
\begin{align}
\left|\nabla_x\mathcal N_j^{r,1}
(D_{j'}E^{osc},\partial_v^\beta\mu,v^{\otimes k})(t,x)\right|
&\lesssim M_T^2\mathbf S_{-1-\kappa_0}(t,x)\lesssim
M_T^2\langle t\rangle^{\kappa_0}\Phi(t,x).
\label{q24}
\end{align}

\subsubsection{Some second-order bounds} We differentiate in $x$ \eqref{z60} again to see that
\begin{equation}\label{sec5eq30}
\begin{split}
\partial_{x_j}\partial_{x_{j'}}&\mathcal N(F,\omega,b)
=t^{-2}\Big[\mathcal N(D_jD_{j'}F,\omega,b)
+\mathcal N(D_{j'}F,\omega,\partial_{v_j}b)+\mathcal N(D_{j'}F,\partial_{v_j}\omega,b)\\
&+\mathcal N_j^{r,1}(D_{j'}F,\omega,b)
+\mathcal N(D_jF,\omega,\partial_{v_{j'}}b)
+\mathcal N(F,\omega,\partial_{v_j}\partial_{v_{j'}}b)+\mathcal N(F,\partial_{v_j}\omega,\partial_{v_{j'}}b)\\
&+\mathcal N_j^{r,1}(F,\omega,\partial_{v_{j'}}b)
+\mathcal N(D_jF,\partial_{v_{j'}}\omega,b)+\mathcal N(F,\partial_{v_{j'}}\omega,\partial_{v_j}b)
+\mathcal N(F,\partial_{v_j}\partial_{v_{j'}}\omega,b)\\
&+\mathcal N_j^{r,1}(F,\partial_{v_{j'}}\omega,b)
\Big]+t^{-1}\partial_{x_j}\mathcal N_{j'}^{r,1}(F,\omega,b).
\end{split}
\end{equation}

We first consider the regular component. Using \eqref{x103}, \eqref{q5},
\eqref{q22}, and \eqref{q25} we have 
\begin{align}
&\left|
\partial_{x_j}\partial_{x_{j'}}
\mathcal N(E^{reg},\partial_v^\beta\mu,v^{\otimes k})
-t^{-1}\partial_{x_j}
\mathcal N_{j'}^{r,1}(E^{reg},\partial_v^\beta\mu,v^{\otimes k})
\right|\lesssim
M_T^2\langle t\rangle^{-2+\kappa_0}\Phi(t,x).
\label{q38}
\end{align}

For the oscillatory component, define
\begin{align*}
\mathbf J_{jj'}(\partial_v^\beta\mu,v^{\otimes k})
&:=\mathcal N(D_jD_{j'}E^{osc},\partial_v^\beta\mu,v^{\otimes k})
+\mathcal N_j^{r,1}(D_{j'}E^{osc},\partial_v^\beta\mu,v^{\otimes k})
\notag\\
&+\mathcal N_j^{r,1}(E^{osc},\partial_{v_{j'}}\partial_v^\beta\mu,v^{\otimes k})
+\mathcal N_j^{r,1}(E^{osc},\partial_v^\beta\mu,\partial_{v_{j'}}v^{\otimes k}).
\end{align*}
The remaining terms are controlled by \eqref{q8} and by the argument used
to prove \eqref{q14}, applied separately to $E^{osc}$ and to the
differentiated symbols. Hence
\begin{equation}\label{q39}
\begin{split}
\left|\partial_{x_j}\partial_{x_{j'}}
\mathcal N(E^{osc},\partial_v^\beta\mu,v^{\otimes k})-t^{-1}\partial_{x_j}
\mathcal N_{j'}^{r,1}(E^{osc},\partial_v^\beta\mu,v^{\otimes k})-t^{-2}\mathbf J_{jj'}(\partial_v^\beta\mu,v^{\otimes k})\right|\\
\lesssim M_T^2\langle t\rangle^{-2+5\kappa_0}\langle t,x\rangle^{-2}.
\end{split}
\end{equation}
Moreover, \eqref{q11}, \eqref{q26}, and \eqref{q19} imply
\begin{align}
\big|\mathbf J_{jj'}(\partial_v^\beta\mu,v^{\otimes k})(t,x)\big|
&\lesssim M_T^2\langle t\rangle^{\frac12+\kappa_0}\langle t,x\rangle^{-2},
\label{x1}
\end{align}
while \eqref{q13}, \eqref{q24}, and \eqref{q23} give
\begin{align}
\big|\nabla_x\mathbf J_{jj'}(\partial_v^\beta\mu,v^{\otimes k})(t,x)\big|&\lesssim M_T^2\langle t\rangle^{\kappa_0}\langle t,x\rangle^{-2}.
\label{x2}
\end{align}
In particular, using \eqref{q38}, \eqref{q39}, and \eqref{x1}, we have
\begin{align}
&\left|\partial_{x_j}\partial_{x_{j'}}
\mathcal N(E,\partial_v^\beta\mu,v^{\otimes k})
-t^{-1}\partial_{x_j}
\mathcal N_{j'}^{r,1}(E,\partial_v^\beta\mu,v^{\otimes k})\right|\lesssim
M_T^2\langle t\rangle^{-\frac32+\kappa_0}\langle t,x\rangle^{-2}.
\label{q39.5}
\end{align}
Therefore, to prove the bounds \eqref{q45} it remains to control the term $\partial_{x_j}
\mathcal N_{j'}^{r,1}(E,\partial_v^\beta\mu,v^{\otimes k})$. This is treated in the next subsection.

We can also interpolate between the bounds \eqref{x1} and \eqref{x2} to show that for $\sigma\in[1/2,1)$,
\begin{align}
\big||\nabla|^\sigma\mathbf J_{jj'}(\partial_v^\beta\mu,v^{\otimes k})(t,x)\big|
&\lesssim_\sigma M_T^2\langle t\rangle^{\frac{1-\sigma}{2}+\kappa_0}
\langle t,x\rangle^{-2}.
\label{x3}
\end{align}
Indeed, we use the integral representation of the operator $|\nabla|^\sigma$,
\begin{equation}\label{fracnabla}
|\nabla|^\sigma f(x)=c_\sigma 
\int_{\mathbb R^2}\frac{f(x)-f(x-h)}{|h|^{2+\sigma}}\,dh,\qquad 0<\sigma<1.
\end{equation}
Let $Q(t,x):=\mathbf J_{jj'}(\partial_v^\beta\mu,v^{\otimes k})(t,x)$ and set
$L=\langle t\rangle^{1/2}/4$. The bounds \eqref{x2} and the comparability
$\langle t,x+\lambda h\rangle\approx\langle t,x\rangle$ for $|h|\leq L$ and $\lambda\in[0,1]$, show that
\begin{align*}
\int_{|h|\leq L}\frac{|Q(t,x)-Q(t,x+h)|}{|h|^{2+\sigma}}\,dh
&\lesssim_\sigma M_T^2\langle t\rangle^{\kappa_0}\langle t,x\rangle^{-2}L^{1-\sigma}\lesssim_\sigma
M_T^2\langle t\rangle^{\frac{1-\sigma}{2}+\kappa_0}\langle t,x\rangle^{-2}.
\end{align*}
For the complementary region, we use \eqref{x1}, so
\begin{equation*}
\begin{split}
\int_{|h|\geq L}&\frac{|Q(t,x)-Q(t,x+h)|}{|h|^{2+\sigma}}\,dh\lesssim M_T^2\langle t\rangle^{\frac12+\kappa_0}\int_{|h|\geq L}
\frac{\langle t,x+h\rangle^{-2}+\langle t,x\rangle^{-2}}{|h|^{2+\sigma}}\,dh\\
&\lesssim_\sigma L^{-\sigma}M_T^2\langle t\rangle^{\frac12+\kappa_0}\langle t,x\rangle^{-2}
+M_T^2\langle t\rangle^{\frac12+\kappa_0}\log(2+\langle t,x\rangle)\langle t,x\rangle^{-2-\sigma}.
\end{split}
\end{equation*}
Since
$\log(2+\langle t,x\rangle)\langle t,x\rangle^{-\sigma}
\lesssim_\sigma\langle t\rangle^{-\sigma/2}$,
the desired bounds \eqref{x3} follow from the integral representation \eqref{fracnabla}.
	
\subsection{The second-order bound for $\mathcal N$}\label{secondorderbound}

In this subsection we complete the proofs of the bounds \eqref{q45}--\eqref{q46}, and the
interior part of \eqref{q47}. In view of the estimates in subsection
\ref{subsection5.3}, we may assume throughout that $(t,x)\in\mathcal D_T$. We will also assume in this subsection that $\omega=\partial_v^\beta\mu$, $b(s,v)=v^{\otimes k}$, $k\in\{0,1,2\}$, $k-|\beta|\leq1$. 

We make the change of variables $v=(x-w)/t$ to rewrite \eqref{z61} in the form
\begin{align*}
&\mathcal N_j^{r,1}(F,\omega,b)(t,x)
=\frac{-1}{t^2}\int_0^t\!\!\int_{\mathbb R^2}b\Big(s,\frac{x-w}{t}\Big)
\Big[\partial_{v_j}Y_{s,t}\Big(w,\frac{x-w}{t}\Big)\omega\Big(\frac{x-w}{t}+W_{s,t}\Big(w,\frac{x-w}{t}\Big)\Big)\\
&\times\nabla_xF\Big(s,w+\frac{s}{t}(x-w)+Y_{s,t}\Big(w,\frac{x-w}{t}\Big)\Big)+\partial_{v_j}W_{s,t}\Big(w,\frac{x-w}{t}\Big)
\notag\\
&\times \nabla\omega\Big(\frac{x-w}{t}+W_{s,t}\Big(w,\frac{x-w}{t}\Big)\Big)\cdot F\Big(s,w+\frac{s}{t}(x-w)+Y_{s,t}\Big(w,\frac{x-w}{t}\Big)\Big)
\Big]\,dwds.
\end{align*}
Taking another $x$ derivative gives
\begin{equation}\label{n25}
\begin{split}
\partial_{x_i}\mathcal N_j^{r,1}(F,\omega,b)
&=t^{-1}\big[
\mathcal N_j^{r,1}(D_iF,\omega,b)
+\mathcal N_j^{r,1}(F,\partial_{v_i}\omega,b)
+\mathcal N_j^{r,1}(F,\omega,\partial_{v_i} b)\big]\\
&+t^{-1}\mathcal N_{ji}^{r,2}(F,\omega,b),
\end{split}
\end{equation}
where $D_iF(s)=s\partial_{x_i}F(s)$ and
\[
\mathcal N_{ji}^{r,2}(F,\omega,b)=Z_1(F)+Z_2(F)+Z_3(F)+Z_4(F)+Z_5(F),
\]
\begin{align*}
Z_1(F)&:=-\int_0^t\!\!\int_{\R^2}b(s,v)\,\partial_{v_j}\partial_{v_i}Y_{s,t}(x-tv,v)\nabla F(s,X_{s,t}(x,v))\,\omega(V_{s,t}(x,v))\,dv\,ds,\\
Z_2(F)&:=-\int_0^t\!\!\int_{\R^2} b(s,v)\,F(s,X_{s,t}(x,v))\nabla\omega(V_{s,t}(x,v))\cdot
\partial_{v_j}\partial_{v_i}W_{s,t}(x-tv,v)\,dv\,ds,\\
Z_3(F)&:=-\int_0^t\!\!\int_{\R^2} b(s,v)\,F(s,X_{s,t}(x,v))
\nabla^2\omega(V_{s,t}(x,v))\big[\partial_{v_i}W_{s,t}\otimes\partial_{v_j}W_{s,t}\big](x-tv,v)\,dv\,ds,\\
Z_4(F)&:=-\int_0^t\!\!\int_{\R^2} b(s,v)\,\nabla^2F(s,X_{s,t}(x,v))\big[\partial_{v_i}Y_{s,t}\otimes\partial_{v_j}Y_{s,t}\big](x-tv,v)
\omega(V_{s,t}(x,v))\,dv\,ds,\\
Z_5(F)&:=-\sum_{\ell=1}^2\int_0^t\!\!\int_{\R^2}b(s,v)\big(\nabla_v\omega_\ell(V_{s,t}(x,v))
\otimes\nabla_xF_\ell(s,X_{s,t}(x,v))\big)\\
&\qquad\qquad\qquad :\big[
\partial_{v_i}W_{s,t}\otimes\partial_{v_j}Y_{s,t}
+\partial_{v_j}W_{s,t}\otimes\partial_{v_i}Y_{s,t}
\big](x-tv,v)\,dv\,ds.
\end{align*}

It follows from \eqref{n25}, \eqref{q22}, \eqref{q25}, \eqref{q19}, and
\eqref{q26} that
\begin{align}
\left|\partial_{x_i}\mathcal N_j^{r,1}(E^{reg},\partial_v^\beta\mu,v^{\otimes k})-t^{-1}\mathcal N_{ji}^{r,2}(E^{reg},\partial_v^\beta\mu,v^{\otimes k})\right|
&\lesssim M_T^2\langle t\rangle^{-1+\kappa_0}\Phi(t,x),
\label{q33}\\
\left|\partial_{x_i}\mathcal N_j^{r,1}(E^{osc},\partial_v^\beta\mu,v^{\otimes k})
-t^{-1}\mathcal N_{ji}^{r,2}(E^{osc},\partial_v^\beta\mu,v^{\otimes k})\right|
&\lesssim M_T^2\langle t\rangle^{-1/2}\langle t,x\rangle^{-2}.
\label{q34}
\end{align}
Moreover, by differentiating \eqref{n25} and using \eqref{q23}--\eqref{q24}, we have
\begin{align}
\left|\nabla_x\left(
\partial_{x_i}\mathcal N_j^{r,1}(E^{osc},\partial_v^\beta\mu,v^{\otimes k})
-t^{-1}\mathcal N_{ji}^{r,2}(E^{osc},\partial_v^\beta\mu,v^{\otimes k})\right)\right|
&\lesssim M_T^2\langle t\rangle^{-1+\kappa_0}
\langle t,x\rangle^{-2}.
\label{q35}
\end{align}
The interpolation argument used in the proof of \eqref{x3} then gives,
for any $\sigma\in[1/2,1)$,
\begin{equation}\label{q34-35}
\begin{split}
&\left||\nabla_x|^\sigma\left(
\partial_{x_i}\mathcal N_j^{r,1}(E^{osc},\partial_v^\beta\mu,v^{\otimes k})-t^{-1}\mathcal N_{ji}^{r,2}(E^{osc},\partial_v^\beta\mu,v^{\otimes k})
\right)\right|\\
&\hspace{27mm}\lesssim_\sigma
M_T^2\langle t\rangle^{-(1+\sigma)/2}\langle t\rangle^{2\kappa_0}\langle t,x\rangle^{-2}.
\end{split}
\end{equation}

\subsubsection{Estimates on $Z_j(E)$, $j\in\{2,\ldots,5\}$} The bounds \eqref{xf'} and \eqref{n17b}, the comparison estimate \eqref{z36}, and the field estimates give
\begin{align*}
\sum_{*=reg,osc}\sum_{m=4}^5|Z_m(E^*)|
&\lesssim M_T^2\!\!\int_0^t\int_{\R^2}
\langle s\rangle^{-5/3}\Phi(s,x-(t-s)v)
\frac{dv\,ds}{\langle v\rangle^{11/5}},\\
\sum_{*=reg,osc}|Z_3(E^*)|
&\lesssim M_T^2\int_0^t\!\!\int_{\R^2}
\langle s\rangle^{-2}\langle s,x-(t-s)v\rangle^{-2}
\frac{dv\,ds}{\langle v\rangle^3},\\
|Z_2(E^{reg})|
&\lesssim M_T^2\langle t\rangle^{1/6-\kappa_0}
\int_0^t\!\!\int_{\R^2}
\langle s\rangle^{-3/2}\Phi(s,x-(t-s)v)
\frac{dv\,ds}{\langle v\rangle^3}.
\end{align*}
Consequently, Lemma \ref{Le3} implies that for $(t,x)\in\mathcal{D}_T$
\begin{align}
\sum_{*=reg,osc}\sum_{m=3}^5|Z_m(E^*)|+|Z_2(E^{reg})|\lesssim M_T^2\langle t\rangle^{1/6}\Phi(t,x).
\label{q31}
\end{align}

We prove now a similar estimate on $|Z_2(E^{osc})|$, 
\begin{equation}\label{sec5eq40}
|Z_2(E^{osc})|\lesssim M_T^2\langle t\rangle^{1/6}\Phi(t,x).
\end{equation}
For this we first record the elementary bound
\begin{equation}\label{second-order-convolution}
\int_{\mathbb R^2}\langle x-tv\rangle^{-2}\langle v\rangle^{-2}
\Big(\int_0^t\langle\tau\rangle^a
\Phi(\tau,x-(t-\tau)v)\,d\tau\Big)dv\lesssim\langle t\rangle^{\kappa_0}\Phi(t,x),
\end{equation}
for $a\leq 1/2$.
\begingroup
It suffices to take $a=1/2$ and set
\[
A(v):=\int_0^t\langle\tau\rangle^{1/2}\Phi(\tau,x-(t-\tau)v)\,d\tau.
\]
Since $\langle x-tv\rangle\leq\sqrt2\langle v\rangle
\langle\tau,x-(t-\tau)v\rangle$, we have
\[
\langle x-tv\rangle^{-2}\Phi(\tau,x-(t-\tau)v)
\lesssim\langle v\rangle^{\kappa_0}\langle x-tv\rangle^{-2-\kappa_0}
\langle\tau,x-(t-\tau)v\rangle^{-2}.
\]
Thus Lemma~\ref{z216} with $(a_1,a_2)=(2,1/2)$ and \eqref{z201} give
\begin{align*}
\int_{\mathbb R^2}\langle x-tv\rangle^{-2}A(v)\frac{dv}{\langle v\rangle^2}
&\lesssim\int_{\mathbb R^2}\langle x-tv\rangle^{-2-\kappa_0}
\frac{dv}{\langle v\rangle^{3-\kappa_0}}\\
&\lesssim\langle t\rangle^{\kappa_0}\Phi(t,x),
\end{align*}
since $3-\kappa_0\geq2+\kappa_0$.
\endgroup

To prove \eqref{sec5eq40} we integrate by parts in time, using the decomposition \eqref{de1}. The upper
endpoint vanishes because $\nabla_v^2W_{t,t}=0$. At the lower endpoint we use
\eqref{x203}. The contribution of
the $\tau=t$ endpoint in \eqref{x203} is bounded directly by Lemma \ref{Le1};
the remaining terms are bounded by
\begin{align*}
M_T^2\langle t\rangle^{1/6-\kappa_0}
\int_{\R^2}\langle x-tv\rangle^{-3}\frac{dv}{\langle v\rangle^3}
&+M_T^2\langle t\rangle^{1/6-\kappa_0}
\int_{\R^2}\langle x-tv\rangle^{-2}\langle v\rangle^{-2}A(v)\,dv.
\end{align*}
By \eqref{second-order-convolution} and Lemma \ref{Le1}, the boundary term is
bounded by $M_T^2\langle t\rangle^{1/6}\Phi(t,x)$, as desired.

When the time derivative falls on the remaining factors, the resulting terms
contain 
\[
\partial_sE^*\cdot \nabla_v^2W_{s,t},\quad(v+\partial_sY_{s,t})\cdot\nabla E^*\cdot \nabla_v^2W_{s,t},\quad E^*\partial_sW_{s,t}\cdot \nabla_v^2W_{s,t},\quad E^*\partial_s\nabla_v^2W_{s,t},
\]
for $*\in\{\cos,\sin\}$. Using \eqref{n13'}, \eqref{n17b},
\eqref{x200}, \eqref{x202}, and \eqref{Ebounds1}, all these contributions are
bounded by
\[
M_T^2\langle t\rangle^{1/6-\kappa_0}
\int_0^t\!\!\int_{\R^2}
\langle s\rangle^{-1-2\kappa_0}
\Phi(s,x-(t-s)v)\frac{dv\,ds}{\langle v\rangle^{2+2\kappa_0}}.
\]
The bounds \eqref{sec5eq40} follow from Lemma \ref{Le3}. Using also \eqref{q31} we obtain
\begin{align}\label{q36}
\left|\mathcal N_{ji}^{r,2}(E^*,\omega,b)-Z_1(E^*)\right|
\lesssim M_T^2\langle t\rangle^{1/6}\Phi(t,x),
\end{align}
for $*\in\{reg,osc\}$ and $(t,x)\in\mathcal D_T$.

\subsubsection{The nonlinear term $Z_1(E)$} With $O^{i,j}_{s,t}$ defined as in \eqref{sec4eq50}, set
\begin{align*}
\mathcal J_{ji}(F,\omega,b)(t,x)
:=-\int_0^t\!\!\int_{\R^2} b(s,v)\,
O_{s,t}^{i,j}(x-tv,v)\cdot\nabla F(s,x-(t-s)v)\,\omega(v)\,dv\,ds.
\end{align*}
We show first that if $(t,x)\in\mathcal{D}_T$ then
\begin{align}\label{n19}
\left|Z_1(E^*)-\mathcal J_{ji}(E^*,\omega,b)\right|
\lesssim M_T^2\langle t\rangle^{1/6}\Phi(t,x),
\qquad *\in\{reg,osc\}.
\end{align}

Indeed, by the mean value theorem, \(\eqref{n13'}\), \(\eqref{n17}\), the field bounds, and Lemma \(\ref{Le3}\), the errors produced by replacing  $X_{s,t}(x,v)$ and $V_{s,t}(x,v)$ by $x-(t-s)v$ and $v$ contain an additional factor $Y_{s,t}$ or $W_{s,t}$ and
are estimated directly. It remains to replace
$\partial_{v_j}\partial_{v_i}Y_{s,t}$ by $O_{s,t}^{i,j}$.

For the regular component, \eqref{q27}--\eqref{q28} and Lemma \ref{Le3} give
\begin{align*}
\left|Z_1(E^{reg})-\mathcal J_{ji}(E^{reg},\omega,b)\right|
&\lesssim M_T^2\langle t\rangle^{1/6-\kappa_0}\!\!\int_0^t\int_{\R^2}
\langle s\rangle^{-7/6}\Phi(s,x-(t-s)v)
\frac{dv\,ds}{\langle v\rangle^3}\\
&\lesssim M_T^2\langle t\rangle^{1/6}\Phi(t,x).
\end{align*}
For the oscillatory component, we split according to
$\chi(|v|/\langle s\rangle)$. On the support of $1-\chi$,
\eqref{q27} and Lemma \ref{Le3} give
\begin{align*}
M_T^2\langle t\rangle^{1/6-\kappa_0}
\int_0^t\!\!\int_{|v|\geq\langle s\rangle/2}
\langle s\rangle^{-2/3+2\kappa_0}
\Phi(s,x-(t-s)v)\frac{dv\,ds}{\langle v\rangle^3}
\lesssim M_T^2\langle t\rangle^{1/6}\Phi(t,x).
\end{align*}
Therefore, for \eqref{n19} it remains to prove that if $(t,x)\in\mathcal{D}_T$ then
\begin{equation}\label{n19.2}
\begin{split}
\Big|\int_0^t\!\!\int_{\R^2} b(s,v)&\chi(|v|/\langle s\rangle)\,[O_{s,t}^{i,j}(x-tv,v)-\partial_{v_i}\partial_{v_j}Y_{s,t}(x-tv,v)]\\
&\times\nabla E^{osc}(s,x-(t-s)v)\,\omega(v)\,dvds\Big|\lesssim M_T^2\langle t\rangle^{1/6}\Phi(t,x).
\end{split}
\end{equation}

We use \eqref{de1} and integrate by parts in time.
The upper endpoint vanishes. Moreover, the lower endpoint and the bulk terms are bounded by
\begin{align*}
&M_T^2\langle t\rangle^{1/6-\kappa_0}
\!\!\int_{|v|\leq4}\langle x-tv\rangle^{-2-\kappa_0}\,dv\\
&+M_T^2\langle t\rangle^{1/6-\kappa_0}
\int_0^t\!\!\int_{\R^2}\langle s\rangle^{-7/6}\Phi(s,x-(t-s)v)
\frac{dv\,ds}{\langle v\rangle^{2+2\kappa_0}}\lesssim M_T^2\langle t\rangle^{1/6}\Phi(t,x),
\end{align*}
using Lemmas \ref{Le1} and \ref{Le3}. This completes the proofs of \eqref{n19.2} and \eqref{n19}.

We show now that if $(t,x)\in\mathcal{D}_T$ then
\begin{align}\label{n20}
|\mathcal J_{ji}(E^{reg},\omega,b)(t,x)|\lesssim M_T^2\langle t\rangle^{1/6}\Phi(t,x).
\end{align}
On the low-velocity region we use \eqref{q62a}; on the high-velocity region we use
\[
\nabla_xE^{reg}(s,x-(t-s)v)=-(t-s)^{-1}\nabla_v[E^{reg}(s,x-(t-s)v)]
\]
and integrate by parts in $v$. By \eqref{q29}--\eqref{q30},
\begin{align*}
|\mathcal J_{ji}(E^{reg},\omega,b)(t,x)|
&\lesssim M_T^2\langle t\rangle^{1/6-\kappa_0}
\int_0^t\!\!\int_{|v|\leq2\langle s\rangle}
\langle s\rangle^{-7/6}\Phi(s,x-(t-s)v)\frac{dv\,ds}{\langle v\rangle^2}\\
&+M_T^2\langle t\rangle^{1/6-\kappa_0}
\int_0^t\!\!\int_{|v|\geq\langle s\rangle/2}
\langle s\rangle^{-1/2}\Phi(s,x-(t-s)v)\frac{dv\,ds}{\langle v\rangle^3}\\
&\lesssim M_T^2\langle t\rangle^{1/6-\kappa_0}
\int_0^t\!\!\int_{\R^2}
\langle s\rangle^{-13/12}\Phi(s,x-(t-s)v)
\frac{dv\,ds}{\langle v\rangle^{25/12}}.
\end{align*}
The bounds \eqref{n20} then follow from Lemma \ref{Le3}.

Combining \eqref{q33}, \eqref{q36}, \eqref{n19}, and \eqref{n20}, we obtain
\begin{align}\label{q37}
\left|\partial_{x_i}\mathcal N_j^{r,1}(E^{reg},\omega,b)(t,x)\right|
\lesssim M_T^2\langle t\rangle^{-5/6}\Phi(t,x).
\end{align}
Together with \eqref{q38}, this yields
\begin{align}\label{q38a}
\left|\nabla_x^2\mathcal N(E^{reg},\omega,b)(t,x)\right|
\lesssim M_T^2\langle t\rangle^{-11/6}\Phi(t,x).
\end{align}

For the oscillatory field, \eqref{q62a} and \eqref{Ebounds1} imply
\begin{align*}
&|\mathcal J_{ji}(E^{osc},\omega,b)(t,x)|\lesssim M_T^2\langle t\rangle^{2/3+\kappa_0}
\int_0^t\!\!\int_{\R^2}
\langle t^{1/2+2\kappa_0},v\rangle^{-1}
\langle s\rangle^{-2/3+2\kappa_0}\Phi(s,x-(t-s)v)
\frac{dv\,ds}{\langle v\rangle^2}.
\end{align*}
Since $\langle t^{1/2+2\kappa_0},v\rangle^{-1}
\lesssim\langle v\rangle^{-\kappa_0}
\langle t\rangle^{-(1/2+2\kappa_0)(1-\kappa_0)}$, 
Lemma \ref{Le3} gives
\begin{align}\label{n21}
|\mathcal J_{ji}(E^{osc},\omega,b)(t,x)|
\lesssim M_T^2\langle t\rangle^{1/2+4\kappa_0}\Phi(t,x).
\end{align}
Therefore \eqref{q34}, \eqref{q36}, and \eqref{n19} imply
\begin{align}\label{q36a}
\left|\partial_{x_i}\mathcal N_j^{r,1}(E^{osc},\omega,b)(t,x)\right|
\lesssim M_T^2\langle t\rangle^{-1/2+4\kappa_0}
\langle t,x\rangle^{-2}.
\end{align}
Using \eqref{q39} and \eqref{x1}, we conclude that
\begin{align}\label{q36b}
\left|\nabla_x^2\mathcal N(E^{osc},\omega,b)(t,x)\right|
\lesssim M_T^2\langle t\rangle^{-3/2+4\kappa_0}
\langle t,x\rangle^{-2}.
\end{align}
In particular, the bounds \eqref{q45} for $(t,x)\in\mathcal D_T$ follow from \eqref{q36b} and \eqref{q38a}.

\subsubsection{The fractional gain and completion of the proof}

The bounds \eqref{q34}, \eqref{x1}, \eqref{n21}, and \eqref{sec5eq55}
hold for every $t\geq1$ and $x\in\mathbb R^2$: their proofs use only
global field, characteristic, and ray estimates.
In particular, \eqref{n21} may be used at $x-h$ outside $\mathcal D_T$ below.

We now prove the additional regularity needed for \eqref{q46} and \eqref{q47}. For $h\in\mathbb R^2$, let
\[
\delta_hu(x):=u(x)-u(x-h).
\]
We claim that for $(t,x)\in\mathcal{D}_T$
\begin{align}\label{q17}
\sup_{0<|h|\leq\langle t,x\rangle/2}
\frac{|\delta_h\mathcal J_{ji}(E^{osc},\omega,b)(t,x)|
}{|h|^{5/6}\min\{|h|,1\}^{1/6}}
\lesssim M_T^2\langle t\rangle^{1/6}\Phi(t,x).
\end{align}

The field estimates and \eqref{a23} give
\begin{equation}\label{sec5eq55}
\begin{split}
|\delta_h\nabla_xE^{osc}(s,z)|
&\lesssim M_T
\min\{1,|h|\langle s\rangle^{-2/3}\}
\langle s\rangle^{-2/3+2\kappa_0}
\big(\Phi(s,z)+\Phi(s,z-h)\big),\\
|\delta_hE^{osc}(s,z)|
&\lesssim M_T|h|^{\kappa_0}\langle s\rangle^{50\kappa_0}
\min\{1,|h|\langle s\rangle^{-2/3}\}^{1-\kappa_0}
\big(\Phi(s,z)+\Phi(s,z-h)\big).
\end{split}
\end{equation}
Indeed, the first bound follows by taking the minimum of the pointwise
bound for $\nabla_xE^{osc}$ and the mean-value bound involving
$|h|\nabla_x^2E^{osc}$. The second bound follows in the same way, by analyzing separately the cases $|h|\gtrsim \langle s,z\rangle$ and $|h|\ll\langle s,z\rangle$, and using the decomposition $E^{osc}=E^{osc}_{lin}+E^{osc}_{non}$, \eqref{a23}, and \eqref{bo1} in the second case.

Similarly, using \eqref{q62a}--\eqref{q62d}, we obtain
\begin{equation}\label{sec5eq56}
\begin{split}
|\delta_hO_{s,t}^{i,j}(x-tv,v)|&\lesssim
M_T\langle t\rangle^{\frac16-\kappa_0}
\min\{1,|h|\langle v\rangle^{-1}\},\\
|\delta_h\partial_sO_{s,t}^{i,j}(x-tv,v)|
&\lesssim M_T\langle s\rangle^{-\frac13+\kappa_0}
\langle v\rangle^{-1}
\min\{1,|h|\langle s\rangle^{-\frac12}\}.
\end{split}
\end{equation}

Notice that 
\begin{equation}\label{sec5eq57}
\begin{split}
&\delta_h\mathcal J_{ji}(E^{osc},\omega,b)(t,x)\\
&=-\int_0^t\!\!\int_{\R^2}
\delta_hO_{s,t}^{i,j}(x-tv,v)\cdot \nabla_x E^{osc}(s,x-(t-s)v)
b(s,v)\omega(v)\,dv\,ds\\
&-\int_0^t\!\!\int_{\R^2}O_{s,t}^{i,j}(x-h-tv,v)\cdot 
\nabla_x\delta_hE^{osc}(s,x-(t-s)v) b(s,v)\omega(v)\,dv\,ds.
\end{split}
\end{equation}
We split these two integrals according to $\chi(|v|/\langle s\rangle)$, so we denote the low-velocity and
high-velocity parts of the first integral by $Z^\ast_1$ and $Z^\ast_2$, 
and the low-velocity and high-velocity parts of the second integral by $Z^\ast_3$ and $Z^\ast_4$. We integrate by parts in $v$ in $Z_4^\ast$, thus
\begin{align*}
Z_4^\ast&:=-\int_0^t\!\!\int_{\R^2}
\nabla_v\Big[
b(s,v)\omega(v)(1-\chi)(|v|/\langle s\rangle)\frac{O_{s,t}^{i,j}(x-h-tv,v)}{t-s}\Big]
\delta_hE^{osc}(s,x-(t-s)v)\,dv\,ds.
\end{align*}

We can prove now \eqref{q17}. The bounds \eqref{sec5eq55}--\eqref{sec5eq56}, \eqref{q62a}, and Lemma \ref{Le3} show that
\begin{align*}
|Z_2^\ast|
&\lesssim M_T^2\langle t\rangle^{1/6}\Phi(t,x)
|h|^{1/3+5\kappa_0}\min\{|h|,1\}^{2/3-5\kappa_0},\\
|Z_3^\ast|
&\lesssim M_T^2\langle t\rangle^{1/6}
\big(\Phi(t,x)+\Phi(t,x-h)\big)
|h|^{2/3}\min\{|h|,1\}^{1/3}.
\end{align*}
Then we use the formula above, the bounds \eqref{sec5eq55}, \eqref{q29}--\eqref{q30}, and Lemma \ref{Le3} to estimate
\begin{align*}
|Z_4^\ast|&\lesssim M_T^2\langle t\rangle^{1/6-\kappa_0}
|h|^{1/10}\min\{|h|,1\}^{9/10}\\
&\qquad\times\int_0^t\!\!\int
\langle s\rangle^{-16/15+53\kappa_0}
\sum_{\ell=0}^1\Phi(s,x-\ell h-(t-s)v)
\frac{dv\,ds}{\langle v\rangle^{2+2\kappa_0}}\\
&\lesssim M_T^2\langle t\rangle^{1/6}
\big(\Phi(t,x)+\Phi(t,x-h)\big)
|h|^{1/10}\min\{|h|,1\}^{9/10}.
\end{align*}
 
Finally, to estimate $Z_1^\ast$ we use \eqref{de1} and integrate by parts in time. The upper
endpoint vanishes since $O_{t,t}^{i,j}=0$. Using \eqref{sec5eq56} and Lemmas \ref{Le1} and \ref{Le3}, the lower endpoint and all bulk
terms are bounded by
\begin{align*}
|Z_1^\ast|
&\lesssim M_T^2\langle t\rangle^{1/6-\kappa_0}
|h|^{5/6}\min\{|h|,1\}^{1/6}\\
&\qquad\times\Big[
\int_{\R^2}\langle x-tv\rangle^{-2-\kappa_0}\frac{dv}{\langle v\rangle^{5/2}}
+\int_0^t\!\!\int_{\R^2}\langle s\rangle^{-1-\kappa_0}
\Phi(s,x-(t-s)v)\frac{dv\,ds}{\langle v\rangle^{2+2\kappa_0}}\Big]\\
&\lesssim M_T^2\langle t\rangle^{1/6}\Phi(t,x)
|h|^{5/6}\min\{|h|,1\}^{1/6}.
\end{align*}
If $|h|\leq\langle t,x\rangle/2$, then
$\Phi(t,x-h)\lesssim\Phi(t,x)$. The bounds \eqref{q17} follows from the estimates on $Z_j^\ast$, $j\in\{1,2,3,4\}$, above.

To summarize, using \eqref{q36}, \eqref{n19}, \eqref{n20}, \eqref{n21}, and \eqref{q17}, we have
\begin{align}
|\mathcal N_{ji}^{r,2}(E^{reg},\omega,b)|
+|\mathcal N_{ji}^{r,2}(E^{osc},\omega,b)
-\mathcal J_{ji}(E^{osc},\omega,b)|
&\lesssim M_T^2\langle t\rangle^{1/6}\Phi(t,x),
\label{q40}\\
|\mathcal J_{ji}(E^{osc},\omega,b)(t,x)|
&\lesssim M_T^2\langle t\rangle^{1/2+4\kappa_0}\Phi(t,x),
\label{q41}\\
\sup_{0<|h|\leq\langle t,x\rangle/2}
\frac{|\delta_h\mathcal J_{ji}(E^{osc},\omega,b)(t,x)|}
{|h|^{5/6}\min\{|h|,1\}^{1/6}}
&\lesssim M_T^2\langle t\rangle^{1/6}\Phi(t,x).
\label{q42}
\end{align}

Let $\sigma=5/6+\epsilon_1$, where $\epsilon_1\in[\kappa_0,1/6-\kappa_0]$. Using the integral 
representation \eqref{fracnabla} of $|\nabla|^\sigma$, split the difference integral at
$|h|=\langle t,x\rangle/2$. The local part is bounded by \eqref{q42}, while
the complementary part is bounded by \eqref{q41}. Since
$\int_{\mathbb R^2}\Phi(t,y)\,dy\lesssim\langle t\rangle^{-\kappa_0}$, we obtain
\begin{equation}\label{q42.5}
\begin{split}
&\left||\nabla|^\sigma\mathcal J_{ji}(E^{osc},\omega,b)(t,x)\right|
\lesssim M_T^2\langle t\rangle^{1/6}\Phi(t,x)+M_T^2\langle t\rangle^{1/2+4\kappa_0}\Phi(t,x)\langle t,x\rangle^{-\sigma}\\
&\qquad+M_T^2\langle t\rangle^{1/2+4\kappa_0}\langle t,x\rangle^{-2-\sigma}\int_{\mathbb R^2}\Phi(t,y)\,dy\lesssim M_T^2\langle t\rangle^{1/6}\Phi(t,x).
\end{split}
\end{equation}

To conclude, set componentwise
\begin{align*}
Q_{ij}:=\partial_{x_i}\mathcal N_j^{r,1}(E^{osc},\omega,b)
-t^{-1}\mathcal N_{ji}^{r,2}(E^{osc},\omega,b),\\
\mathbb S_{ij}:=t^{-1}Q_{ij}
+t^{-2}\mathbf J_{ij}(\omega,b)
+t^{-2}\mathcal J_{ji}(E^{osc},\omega,b).
\end{align*}
The bounds \eqref{q42.5}, 
\eqref{q34-35}, \eqref{x3}, \eqref{q39}, \eqref{q40}, and \eqref{q38a} show that
\begin{equation}\label{second-order-structure}
\left|\nabla_x^2\mathcal N(E,\omega,b)-\mathbb S\right|
+\left||\nabla|^\sigma\mathbb S\right|
\lesssim M_T^2\langle t\rangle^{-11/6}\Phi(t,x),
\end{equation}
for $\sigma\in[5/6+\kappa_0,1-\kappa_0]$. Indeed, the only additional terms to check are
\begin{align*}
t^{-1}\bigl|\,|\nabla|^\sigma Q\,\bigr|
&\lesssim M_T^2\langle t\rangle^{-(3+\sigma)/2}
\langle t\rangle^{2\kappa_0}
\langle t,x\rangle^{-2},\\
t^{-2}\bigl|\,|\nabla|^\sigma\mathbf J\,\bigr|
&\lesssim M_T^2
\langle t\rangle^{-(3+\sigma)/2+\kappa_0}
\langle t,x\rangle^{-2}.
\end{align*}
and both are bounded by the right-hand side of
\eqref{second-order-structure} in $\mathcal D_T$.

For $t\geq1$, choose $\psi(x)=\psi_0(x/\langle t\rangle^{10})$, with
$\psi_0\in C_c^\infty(\mathbb R^2)$ equal to $1$ for $|x|\leq1/2$
and $0$ for $|x|\geq1$.
For $\mathcal R_k$, sum the tensors above as in \eqref{sec5eq6}, retaining
the notation $\mathbb S$. By \eqref{q34}, \eqref{x1}, and \eqref{n21},
\[
|\mathbb S(t,x)|\lesssim M_T^2\langle t\rangle^{-3/2+4\kappa_0}\langle t,x\rangle^{-2},
\qquad x\in\mathbb R^2.
\]
Using \eqref{fracnabla}, the support of $\psi$, and
$|\psi(x)-\psi(x-h)|\lesssim\min\{1,|h|\langle t\rangle^{-10}\}$,
splitting at $|h|=\langle t,x\rangle/2$ gives, for $\sigma\in[5/6+\kappa_0,1-\kappa_0]$,
\begin{align*}
|[|\nabla|^\sigma,\psi]\mathbb S(t,x)|
&\lesssim M_T^2\langle t\rangle^{-3/2+5\kappa_0}
\begin{cases}
\langle t\rangle^{-\sigma}\langle t,x\rangle^{-2},&|x|\leq2\langle t\rangle^{10},\\
|x|^{-2-\sigma},&|x|>2\langle t\rangle^{10}
\end{cases}\\
&\lesssim M_T^2\langle t\rangle^{-11/6}\Phi(t,x).
\end{align*}
Here the logarithmic spatial integrals are absorbed by $\langle t\rangle^{\kappa_0}$;
the last comparison uses $\sigma\geq5/6$ and $\kappa_0=10^{-4}$.
Together with \eqref{second-order-structure}
on $|x|\leq\langle t\rangle^{10}$, this controls
$|\nabla|^\sigma(\psi\mathbb S)
=\psi|\nabla|^\sigma\mathbb S+[|\nabla|^\sigma,\psi]\mathbb S$ globally.
Also, \eqref{z34b} gives, on $|x|\geq\langle t\rangle^{10}/2$,
\[
|\nabla^2\mathcal R_k(t,x)|
\lesssim M_T^2\langle t\rangle^{-13/4+10\kappa_0}\Phi(t,x)
\lesssim M_T^2\langle t\rangle^{-11/6}\Phi(t,x).
\]
Thus, with
\[
\mathbb H:=\psi(\nabla^2\mathcal R_k-\mathbb S)
+(1-\psi)\nabla^2\mathcal R_k,
\]
we have the global decomposition and bounds
\[
\nabla^2\mathcal R_k=\mathbb H+\psi\mathbb S,
\qquad |\mathbb H|+\big||\nabla|^\sigma(\psi\mathbb S)\big|
\lesssim M_T^2\langle t\rangle^{-11/6}\Phi(t,x).
\]
For \eqref{q46}, fix $\sigma_0\in[5/6+\kappa_0,1-2\kappa_0]$ and write,
with the componentwise tensor contractions understood,
\begin{align*}
\nabla^3(1-\Delta)^{-1}\mathcal R_k=\nabla(1-\Delta)^{-1}\mathbb H+\nabla(1-\Delta)^{-1}|\nabla|^{-\sigma_0}
\big[|\nabla|^{\sigma_0}(\psi\mathbb S)\big].
\end{align*}
Both kernels are integrable and preserve $\Phi$; the second decays
like $|x|^{-3+\sigma_0}$ at infinity, with $3-\sigma_0>2+\kappa_0$.
This proves \eqref{q46} for all $x$.
Finally, taking $\sigma=5/6+\epsilon_1$ as in \eqref{q47},
two integrations by parts and the self-adjointness of $|\nabla|^\sigma$ give
\begin{align*}
\int\nabla^3\varphi\,\mathcal R_k\,dx
&=\int\nabla\varphi\,\mathbb H\,dx
+\int |\nabla|^{-\sigma}\nabla\varphi\,
|\nabla|^\sigma(\psi\mathbb S)\,dx.
\end{align*}
The global bounds above imply \eqref{q47}; the identities are justified
first for smooth compactly supported test functions and then by approximation.
This completes the proof of Proposition~\ref{maines}.

\section{Proof of Proposition \ref{propoglobal3d}}
In this section we combine the estimates for the initial-data components
$\mathcal I_k$ and the nonlinear components $\mathcal R_k$ and complete the
proof of Proposition \ref{propoglobal3d}.
\subsection{Consequences of the moment estimates}

We first record the estimates for the full moment fields
$\mathbf R_k=\mathcal I_k+\mathcal R_k$ that will be used below.

\begin{lemma}\label{lem:full-moment-estimates}
For $k\in\{0,1,2\}$, the estimates \eqref{q43}--\eqref{q47} remain valid
with $\mathcal R_k$ replaced by $\mathbf R_k$ and $M_T^2$ replaced by
$[f_0]+M_T^2$.

Assume in addition that $m_0:\mathbb R^2\to\mathbb C$ is supported in
$\{|\xi|\leq2\}$ and satisfies
\begin{equation}\label{sec5eq2}
|\nabla_\xi^j m_0(\xi)|\lesssim |\xi|^{-j},
\qquad 0\leq j\leq10.
\end{equation}
Then
\begin{align}
|\nabla m_0(\nabla)\mathbf R_k(t,x)|
&\lesssim ([f_0]+M_T^2)
\langle t\rangle^{-\frac56+21\kappa_0}\Phi(t,x),
\label{q44b}\\
|\nabla^2m_0(\nabla)\mathbf R_k(t,x)|
&\lesssim ([f_0]+M_T^2)
\langle t\rangle^{-\frac32+31\kappa_0}\Phi(t,x),
\label{q45b}\\
\big||\nabla|^{\frac{17}{6}+\frac1{100}}
m_0(\nabla)\mathbf R_k(t,x)\big|
&\lesssim ([f_0]+M_T^2)
\langle t\rangle^{-\frac{11}{6}}\Phi(t,x).
\label{q46b}
\end{align}
\end{lemma}

\begin{proof}
The bounds \eqref{q43}--\eqref{q45} for $\mathbf R_k$ follow immediately
from Proposition \ref{maines} and Lemma \ref{estI_k}. For the two remaining
bounds, notice that \eqref{Z16-new} gives
\[
|\nabla_x^2\mathcal I_k(t,x)|
\lesssim [f_0]\langle t\rangle^{-\frac{11}{6}}\Phi(t,x).
\]
The kernel of $\nabla(1-\Delta)^{-1}$ is integrable and preserves the
weight $\Phi$. Therefore
\[
|\nabla^3(1-\Delta)^{-1}\mathcal I_k(t,x)|
\lesssim [f_0]\langle t\rangle^{-\frac{11}{6}}\Phi(t,x).
\]
Moreover, integrating by parts twice gives
\[
\left|\int_{\mathbb R^2}\nabla^3\varphi(x)\mathcal I_k(t,x)\,dx\right|
\lesssim [f_0]\langle t\rangle^{-\frac{11}{6}}
\int_{\mathbb R^2}|\nabla\varphi(x)|\Phi(t,x)\,dx.
\]
This proves the full-moment versions of \eqref{q46} and \eqref{q47}.

We next prove \eqref{q44b}--\eqref{q45b}. Let $K_0$ be the kernel of
$m_0(\nabla)$. The symbol bounds \eqref{sec5eq2} imply
\[
|\nabla^jK_0(z)|\lesssim\langle z\rangle^{-2-j},
\qquad 0\leq j\leq8.
\]
To estimate $K_0*\nabla^j\mathbf R_k$, $j=1,2$, we split the kernel at
the scale $\langle t\rangle$. On $|z|\lesssim\langle t\rangle$ we retain
all derivatives on $\mathbf R_k$ and use
\[
\int_{|z|\lesssim\langle t\rangle}|K_0(z)|\,dz
\lesssim\log(2+t).
\]
On the complementary region we move one derivative from $\mathbf R_k$
to the cutoff kernel. The resulting kernel is bounded by
$\langle z\rangle^{-3}$ and has $L^1$ norm
$O(\langle t\rangle^{-1})$. The near and far contributions are therefore
both bounded as claimed. 

It remains to prove \eqref{q46b}. Set $\sigma=5/6+1/100$, $\sigma_0=5/6+1/200$. For fixed $x$, apply the full-moment version of \eqref{q47}, with exponent
$\sigma_0$, to regularized kernels $\varphi_x^{abc}$ such that
\[
\sum_{a,b,c\in\{1,2\}}\partial_{y_a}\partial_{y_b}\partial_{y_c}\varphi^{abc}_x(y)
=\big[|\nabla|^{2+\sigma}m_0(\nabla)\big](x-y).
\]
The two kernels on the right-hand side of \eqref{q47} satisfy
\begin{align*}
|\nabla\varphi^{abc}_x(y)|
&\lesssim\langle x-y\rangle^{-2-\sigma},\qquad \big|\nabla|\nabla|^{-\sigma_0}\varphi^{abc}_x(y)\big|\lesssim\langle x-y\rangle^{-2-(\sigma-\sigma_0)}.
\end{align*}
Since $\sigma-\sigma_0=1/200>\kappa_0$, both kernels preserve the weight
$\Phi$. Thus \eqref{q47} gives \eqref{q46b}. One may first regularize the
symbols at $\xi=0$ and truncate the kernels in physical space; the stated
bound then follows by passing to the limit.
\end{proof}

\subsection{The regular component}

We now estimate the two terms in \eqref{de1.1}. By Lemma \ref{keres}, the
full-moment version of \eqref{q46} for zero and one spatial derivatives,
and the full-moment version of \eqref{q47} for two spatial derivatives,
we have, for $n\in\{0,1,2\}$,
\begin{align*}
&\left|\nabla_x^n\int_0^t\cos(t-s)(1-\chi)(|\nabla|)
\nabla\Delta^{-1}e^{-(t-s)|\nabla|}
\big(|\nabla|\mathbf R_0-\operatorname{div}\mathbf R_1\big)(s)(x)\,ds\right|\\
&\qquad\lesssim([f_0]+M_T^2)
\int_0^t\!\!\int_{\mathbb R^2}
\frac{\langle s\rangle^{-\frac{11}{6}}\Phi(s,y)}
{(t-s+|x-y|)^2\langle t-s,x-y\rangle^6}\,dy\,ds.
\end{align*}
Here the use of \eqref{q47} when $n=2$ transfers two derivatives to the
moment field without creating a nonintegrable kernel.

On the other hand, \eqref{bo1}, \eqref{a23}, and the definition of the
bootstrap norms imply
\[
|E(s,y)|^2
\lesssim M_T^2\langle s\rangle^{100\kappa_0}
\langle s,y\rangle^{-4}.
\]
Consequently,
\begin{align*}
&\left|\nabla_x^n\int_0^t\sin(t-s)\chi(|\nabla|)e^{-(t-s)|\nabla|}
\nabla\Delta^{-1}
\left(\operatorname{div}^2(E\otimes E)-\frac12\Delta|E|^2\right)(s)(x)\,ds\right|\\
&\qquad\lesssim M_T^2
\int_0^t\!\!\int_{\mathbb R^2}
\frac{\langle s\rangle^{100\kappa_0}\langle s,y\rangle^{-4}}
{\langle t-s,x-y\rangle^{3+n}}\,dy\,ds.
\end{align*}
Standard dyadic decompositions in $(s,y)$ give
\begin{align*}
\int_0^t\!\!\int_{\mathbb R^2}
\frac{\langle s\rangle^{-\frac{11}{6}}\Phi(s,y)}
{(t-s+|x-y|)^2\langle t-s,x-y\rangle^6}\,dy\,ds
&\lesssim\langle t\rangle^{-\frac{11}{6}}\Phi(t,x),\\
\int_0^t\!\!\int_{\mathbb R^2}
\frac{\langle s\rangle^{100\kappa_0}\langle s,y\rangle^{-4}}
{\langle t-s,x-y\rangle^{3+n}}\,dy\,ds
&\lesssim\langle t\rangle^{-\frac12-\frac{2n}{3}}\Phi(t,x).
\end{align*}
It follows that
\begin{equation}\label{section6-regular-bound}
\|E^{reg}\|_{1,T}\lesssim [f_0]+M_T^2.
\end{equation}

\subsection{The nonlinear oscillatory components}

Let $\Box\in\{\cos,\sin\}$. Since the trigonometric factors in
\eqref{z40.4} are bounded, Lemma \ref{keres} and the full-moment versions
of \eqref{q43}--\eqref{q45} give
\begin{align*}
|E^\Box_{non}(t,x)|
&\lesssim([f_0]+M_T^2)
\int_0^t\!\!\int_{\mathbb R^2}
\frac{\langle s\rangle^{20\kappa_0}\Phi(s,y)}
{\langle t-s,x-y\rangle^3}\,dy\,ds,\\
|\nabla_xE^\Box_{non}(t,x)|
&\lesssim([f_0]+M_T^2)
\int_0^t\!\!\int_{\mathbb R^2}
\frac{\langle s\rangle^{-\frac56+20\kappa_0}\Phi(s,y)}
{\langle t-s,x-y\rangle^3}\,dy\,ds.
\end{align*}
The kernel $\langle t-s,x-y\rangle^{-3}$ is borderline integrable in
time-space and produces at most a factor $\log(2+t)$. Therefore the
preceding estimates imply
\begin{align}
|E^\Box_{non}(t,x)|
&\lesssim([f_0]+M_T^2)
\langle t\rangle^{50\kappa_0}\Phi(t,x),
\label{section6-osc-zero}\\
|\nabla_xE^\Box_{non}(t,x)|
&\lesssim([f_0]+M_T^2)
\langle t\rangle^{-\frac23+2\kappa_0}\Phi(t,x).
\label{section6-osc-one}
\end{align}

For the second spatial derivative, we split the time integral into
$[0,t/2]$ and $[t/2,t]$. On the first interval, we place the two
additional derivatives on the kernel and use the full-moment version
of \eqref{q43}. On the second interval, we transfer these derivatives
to the moment fields and use the full-moment version of \eqref{q45}.
Thus, by Lemma \ref{keres},
\begin{align*}
|\nabla_x^2E^\Box_{non}(t,x)|
&\lesssim ([f_0]+M_T^2)
\int_0^t\!\!\int_{\mathbb R^2}
\Bigg[
\frac{\mathbf{1}_{s\leq t/2}\langle s\rangle^{20\kappa_0}}
{\langle t-s,x-y\rangle^5}+
\frac{\mathbf{1}_{s>t/2}
\langle s\rangle^{-\frac32+30\kappa_0}}
{\langle t-s,x-y\rangle^3}
\Bigg]\Phi(s,y)\,dy\,ds\\
&=:([f_0]+M_T^2)
\bigl(I_{\mathrm{early}}(t,x)+I_{\mathrm{late}}(t,x)\bigr).
\end{align*}
We first estimate $I_{\mathrm{early}}$. Set
$R:=\langle t,x\rangle$. For $0\leq s\leq t/2$, splitting the
spatial integral into $|y|\geq R/4$ and $|y|<R/4$ gives
\begin{equation*}
\int_{\mathbb R^2}
\frac{\Phi(s,y)}{\langle t-s,x-y\rangle^5}\,dy
\lesssim
\langle t\rangle^{-3}\Phi(t,x)
+R^{-5}\langle s\rangle^{-\kappa_0}.
\end{equation*}
Indeed, on the first region,
$\Phi(s,y)\lesssim R^{-2-\kappa_0}=\Phi(t,x)$, and
the $L^1$ norm of the kernel is bounded by
$C\langle t-s\rangle^{-3}\lesssim\langle t\rangle^{-3}$.
On the second region, $\langle t-s,x-y\rangle\gtrsim R$,
while
$\int_{\mathbb R^2}\Phi(s,y)\,dy
\lesssim\langle s\rangle^{-\kappa_0}$.
Consequently,
\begin{align*}
I_{\mathrm{early}}(t,x)
&\lesssim
\langle t\rangle^{-3}\Phi(t,x)
\int_0^{t/2}\langle s\rangle^{20\kappa_0}\,ds
+R^{-5}\int_0^{t/2}\langle s\rangle^{19\kappa_0}\,ds\\
&\lesssim
\langle t\rangle^{-2+20\kappa_0}\Phi(t,x),
\end{align*}
where the last inequality follows from $R\geq\langle t\rangle$.
For $t/2\leq s\leq t$, we have
$\langle s\rangle\approx\langle t\rangle$ and
\begin{equation*}
\int_{\mathbb R^2}
\frac{\Phi(s,y)}{\langle t-s,x-y\rangle^3}\,dy
\lesssim
\langle t-s\rangle^{-1}\Phi(t,x).
\end{equation*}
To see this, if $|x|\leq2\langle t\rangle$, use
$\Phi(s,y)\lesssim\langle t\rangle^{-2-\kappa_0}
\lesssim\Phi(t,x)$ and the $L^1$ norm of the kernel.
If $|x|>2\langle t\rangle$, split into
$|y|\geq|x|/2$ and $|y|<|x|/2$.
The first region is estimated in the same way. The second contributes
at most
\begin{equation*}
|x|^{-3}\langle t\rangle^{-\kappa_0}
\lesssim
\langle t-s\rangle^{-1}|x|^{-2-\kappa_0}
\lesssim
\langle t-s\rangle^{-1}\Phi(t,x),
\end{equation*}
since $\langle t-s\rangle\leq\langle t\rangle$ and
$\kappa_0<1$.
It follows that
\begin{align*}
I_{\mathrm{late}}(t,x)
&\lesssim
\langle t\rangle^{-\frac32+30\kappa_0}\Phi(t,x)
\int_{t/2}^t\langle t-s\rangle^{-1}\,ds\lesssim
\langle t\rangle^{-\frac32+30\kappa_0}
\log(2+t)\Phi(t,x).
\end{align*}
Combining the two estimates, we obtain
\begin{align}
|\nabla_x^2E^\Box_{non}(t,x)|
&\lesssim ([f_0]+M_T^2)
\left(
\langle t\rangle^{-2+20\kappa_0}
+\langle t\rangle^{-\frac32+30\kappa_0}\log(2+t)
\right)\Phi(t,x)\nonumber\\
&\lesssim ([f_0]+M_T^2)
\langle t\rangle^{-\frac43+2\kappa_0}\Phi(t,x).
\label{section6-osc-two}
\end{align}
Here we used $\kappa_0=10^{-4}$, so that $18\kappa_0<2/3$ and
$28\kappa_0<1/6$; the latter inequality allows us to absorb the
logarithmic factor.

We next estimate the third spatial derivative. It is enough to consider
$t\geq2$, since the bound for bounded times follows directly from the
preceding kernel estimates. We split $[0,t]=[0,t/2]\cup[t/2,t]$. On the
first interval we use the full-moment version of \eqref{q43} and place all
four derivatives on the low-frequency kernel. Thus
\begin{align*}
&\int_0^{t/2}\!\!\int_{\mathbb R^2}
\frac{([f_0]+M_T^2)\langle s\rangle^{20\kappa_0}\Phi(s,y)}
{\langle t-s,x-y\rangle^6}\,dy\,ds\lesssim([f_0]+M_T^2)
\langle t\rangle^{-3+20\kappa_0}\Phi(t,x).
\end{align*}
On $[t/2,t]$ we use \eqref{q46b}. After extracting
$|\nabla|^{17/6+1/100}$ from the moment field, the remaining kernel has
order $7/6-1/100$ and satisfies
\[
|K(t-s,x-y)|
\lesssim\langle t-s,x-y\rangle^{-\frac{19}{6}+\frac1{100}}.
\]
Since $19/6-1/100>3$ and $s\geq t/2$, convolution with this kernel gives
\begin{align*}
&\int_{t/2}^t\!\!\int_{\mathbb R^2}
\frac{([f_0]+M_T^2)\langle s\rangle^{-\frac{11}{6}}\Phi(s,y)}
{\langle t-s,x-y\rangle^{\frac{19}{6}-\frac1{100}}}\,dy\,ds\lesssim([f_0]+M_T^2)
\langle t\rangle^{-\frac{11}{6}}\Phi(t,x).
\end{align*}
We conclude that
\begin{equation}\label{section6-osc-three}
|\nabla_x^3E^\Box_{non}(t,x)|
\lesssim([f_0]+M_T^2)
\langle t\rangle^{-\frac{11}{6}}\Phi(t,x).
\end{equation}

It remains to estimate time derivatives. For this paragraph, set
\[
\mathcal P
:=|\nabla|^2\mathbf R_0-2|\nabla|\operatorname{div}\mathbf R_1
+\operatorname{div}^2\mathbf R_2.
\]
Differentiating \eqref{z40.4}, for $n\in\{0,1,2\}$, gives the bound
\begin{align*}
|\partial_t\nabla_x^nE^\Box_{non}(t,x)|
&\lesssim
\left|\nabla_x^n\chi(|\nabla|)\nabla\Delta^{-1}
\mathcal P(t,x)\right|\\
&\quad+\int_0^t\left|\nabla_x^n|\nabla|\chi(|\nabla|)
e^{-(t-s)|\nabla|}\nabla\Delta^{-1}\mathcal P(s)(x)\right|\,ds.
\end{align*}
The endpoint term is bounded by \eqref{q44b}, \eqref{q45b}, and
\eqref{q46b}, for $n=0,1,2$, respectively. In the last case the remaining
multiplier has order $1/6-1/100$ and therefore has an integrable kernel
that preserves $\Phi$. Hence the endpoint contribution is bounded by $([f_0]+M_T^2)\langle t\rangle^{-\frac12-\frac{2n}{3}}\Phi(t,x)$. 
The integral term is estimated by the same arguments used for one, two,
and three spatial derivatives above; the harmless order-zero angular
multipliers do not affect the kernel bounds. Using
\eqref{section6-osc-one}--\eqref{section6-osc-three}, we obtain
\begin{equation}\label{section6-time-bound}
|\partial_t\nabla_x^nE^\Box_{non}(t,x)|
\lesssim([f_0]+M_T^2)
\langle t\rangle^{-\frac12-\frac{2n}{3}}\Phi(t,x),
\qquad n\in\{0,1,2\}.
\end{equation}

Combining \eqref{section6-osc-zero}--\eqref{section6-time-bound} gives
\begin{equation}\label{section6-osc-bound}
\|(E^{\cos}_{non},E^{\sin}_{non})\|_{2,T}
\lesssim [f_0]+M_T^2.
\end{equation}
Proposition \ref{propoglobal3d} follows using also \eqref{section6-regular-bound}.

	\section{Proof of Proposition \ref{Localwell}} 

	In this section we assume that $T$ is fixed and $\delta$ is sufficiently small with respect to $T$. In both cases $T=0$ and $T>0$, we have
	\begin{align}\label{a35}
		\sum_{n=0,1,2}|\nabla_x^{n} E(t,x)|\lesssim_T\epsilon \langle x\rangle^{-2}
	\end{align}
	for any $t\in [0,T]$ and $x\in\R^2$. For $T>0$, this follows from
	Lemma~\ref{eslinear} and \eqref{bo1'}; for $T=0$, it follows from
	$E(0)=E_{lin}^{\cos}(0)$, Lemma~\ref{eslinear}, and $[f_0]\leq\epsilon$.
	Here the notation $X\lesssim_T Y$ means in this section that the inequality holds with implied constants that depend only on powers of $\langle T\rangle$. 
	
	We divide the proof into several steps.
	\medskip

	\textbf{Step 1.} We construct the solution by Picard iteration. For this we define the Banach spaces $G^{k,\delta}$, $k\in\{0,1,2\}$, of vector-fields $(h_1,h_2,h_3): [T,T+\delta]\times\R^2\to(\R^2)^3$ by the norms
	\begin{equation}\label{sec7eq1}
		\|(h_1,h_2,h_3)\|_{G^{k,\delta}}:= \sum_{n\leq k}\sup_{t\in[T,T+\delta],\,x\in \mathbb{R}^2}\langle x\rangle^{2+2\kappa_0}|\nabla_x^n(h_1,h_2,h_3,\nabla_{t,x}h_2,\nabla_{t,x}h_3)(t,x)|.
	\end{equation}
	For a triple $H=(H_1,H_2,H_3)\in G^{k,\delta}$, set $\mathcal U(H):=
	(H_1,H_2,H_3,\nabla_{t,x}H_2,\nabla_{t,x}H_3)$
	and define the terminal-variation seminorm
	\begin{equation}\label{sec7eq2}
	[H]_{k,\delta;T}:=\sum_{n\leq k}
	\sup_{\substack{t\in[T,T+\delta],\,x\in\mathbb R^2}}
	\langle x\rangle^{2+2\kappa_0}
	\left|\nabla_x^n\mathcal U(H)(t,x)-
	\nabla_x^n\mathcal U(H)(T,x)\right|.
	\end{equation}
For $T=0$ we construct the solution on $[0,\delta]$; for
	$T>0$ we extend the given solution from $[0,T]$ to $[0,T+\delta]$.
	In both cases the unknown components on $[T,T+\delta]$ are constructed by
	a Picard iteration.
	
	Precisely, for any suitable $\tilde E=(\tilde{E}^{reg},\tilde{E}^{\cos}_{non},\tilde{E}^{\sin}_{non})$ defined on $[T,T+\delta]$, we define the vector-field $E^\ast$ on $[0,T+\delta]$ by
	\begin{equation}\label{sec7eq3}
	\begin{split}
	&E^\ast(t):=\widetilde{E}^{reg}(t)+\cos(t)\bigl(\widetilde{E}^{\cos}_{non}+E_{lin}^{\cos}\bigr)(t)
		+\sin(t)\bigl(\widetilde{E}^{\sin}_{non}+E_{lin}^{\sin}\bigr)(t),
		\qquad t\in (T,T+\delta],\\
	&E^\ast(t):=E(t),\qquad t\in [0,T].
	\end{split}
	\end{equation}
		Let $(X_{s,t}^\ast,V_{s,t}^\ast,Y_{s,t}^\ast,W_{s,t}^\ast)$, for $0\le s\le t\le T+\delta$, denote the characteristics defined in \eqref{Lan1} and \eqref{Lan6} associated with the field $E^\ast$. By construction,
	\begin{equation}\label{sec7eq5}
		(X_{s,t}^\ast,V_{s,t}^\ast,Y_{s,t}^\ast,W_{s,t}^\ast)=(X_{s,t},V_{s,t},Y_{s,t},W_{s,t}),
		\qquad 0\le s\le t\le T.
	\end{equation}
	By \eqref{a31} and \eqref{a32}, for $0\le s\le T\le t\le T+\delta$, we have
	\begin{equation}\label{a31'}
	\begin{split}
		Y_{s,t}^\ast(x,v)
		&=Y_{T,t}^\ast(x,v)-(T-s)W_{T,t}^\ast(x,v)\\
		&+Y_{s,T}(x+Y_{T,t}^\ast(x,v)-TW_{T,t}^\ast(x,v),\,v+W_{T,t}^\ast(x,v)),\\
		W_{s,t}^\ast(x,v)&=W_{T,t}^\ast(x,v)+W_{s,T}(x+Y_{T,t}^\ast(x,v)-TW_{T,t}^\ast(x,v),\,v+W_{T,t}^\ast(x,v)).
		\end{split}
	\end{equation}

	As in \eqref{z1.5}--\eqref{z1.6}, for $k\in\{0,1,2\}$ we define
	\begin{equation}\label{sec7eq7}
		\mathbf{R}_k(E^\ast)(t,x):=\mathcal{I}_k(E^\ast)(t,x)+\mathcal{R}_k(E^\ast)(t,x)
	\end{equation}
	on $[0,T+\delta]\times\mathbb{R}^2$. Then we define (compare with \eqref{de1.1}--\eqref{z40.4})
	\begin{equation}\label{sec7eq8}
	\begin{split}
		\mathcal{B}^{reg}(\tilde E)&:=
		\int_0^t \cos(t-s)(1-\chi)(|\nabla|)\nabla\Delta^{-1}e^{-(t-s)|\nabla|}
		\bigl(|\nabla|\mathbf{R}_0(E^\ast)-\operatorname{div}\mathbf{R}_1(E^\ast)\bigr)(s)\,ds\\
		&-\int_0^t \sin(t-s)\chi(|\nabla|)e^{-(t-s)|\nabla|}\nabla\Delta^{-1}
		\Bigl(\operatorname{div}^2(E^\ast\otimes E^\ast)-\frac12\Delta |E^\ast|^2\Bigr)(s)\,ds\\
		&=:	\mathcal{B}_1^{reg}(\tilde E)+	\mathcal{B}^{reg}_2(\tilde E),
		\end{split}
	\end{equation}
	and
	\begin{equation}\label{sec7eq9}
		\begin{split}
			\mathcal{B}^{\cos}(\tilde E)&:=-\int_{0}^{t}\sin(s)\chi(|\nabla|)e^{-(t-s)|\nabla|} \nabla\Delta^{-1}\\
			&\qquad\qquad\bigl(|\nabla|^2\mathbf{R}_0( E^\ast)
			-2|\nabla|\operatorname{div}\mathbf{R}_1( E^\ast)
			+\operatorname{div}^2\mathbf{R}_2( E^\ast)\bigr)(s)\,ds,\\
			\mathcal{B}^{\sin}(\tilde E)&:=\int_{0}^{t}\cos(s)\chi(|\nabla|)e^{-(t-s)|\nabla|} \nabla\Delta^{-1}\\
			&\qquad\qquad\bigl(|\nabla|^2\mathbf{R}_0( E^\ast)
			-2|\nabla|\operatorname{div}\mathbf{R}_1( E^\ast)
			+\operatorname{div}^2\mathbf{R}_2( E^\ast)\bigr)(s)\,ds.
		\end{split}
	\end{equation}
	Accordingly, we define $\mathcal{B}(\tilde E):=\bigl(\mathcal{B}^{reg}(\tilde E),\mathcal{B}^{\cos}(\tilde E),\mathcal{B}^{\sin}(\tilde E)\bigr)$. Every classical solution of \eqref{eq2} satisfies the following fixed-point equation via \eqref{de1}; Step~3 proves the converse.
	\begin{align}\label{a14}
		\mathcal{B}(\tilde E)=\tilde E\qquad\text{ on }\qquad [T,T+\delta]\times\mathbb{R}^2.
	\end{align}

	\textbf{Step 2.} We construct now the solution $\widetilde{E}$ to the fixed-point problem \eqref{a14}. We initialize the iteration by setting
	\begin{align*}
		\tilde E^{(0)}=0,\qquad\text{i.e.}\qquad
		\tilde{E}^{reg,(0)}=\tilde{E}^{\cos,(0)}_{non}=\tilde{E}^{\sin,(0)}_{non}=0.
	\end{align*}
	The field $E^{\ast,(0)}$ associated with \(\widetilde E^{(0)}=0\) may be discontinuous at \(T\), but it is used only as an auxiliary input to \(\mathcal B\). Since $\tilde E^{(0)}=0$ refers only to the three unknown nonlinear
	components, we have
	\begin{equation*}
		E^{\ast,(0)}(t)=\cos(t)E_{lin}^{\cos}(t)+\sin(t)E_{lin}^{\sin}(t),
		\qquad T<t\le T+\delta.
	\end{equation*}
	
	We state first the main a priori estimate for the iteration map $\mathcal{B}$.
	
	\begin{lemma}\label{keyLe} Assume $\delta\leq \delta(T)$ is sufficiently small. Then, if $\tilde E^1, \tilde{E}^2\in G^{2,\delta}$ satisfy
		\begin{equation}\label{a27}
			\|\tilde E^1\|_{G^{2,\delta}}+\|\tilde E^2\|_{G^{2,\delta}}\le \epsilon\delta^{-1/1000},
		\end{equation}
		then
		\begin{align}\label{a26}
		&\|\mathcal{B}(\tilde E^1)\|_{G^{2,\delta}}+\|\mathcal{B}(\tilde E^2)\|_{G^{2,\delta}}\lesssim_T\epsilon,\\&	
		\label{a24'}[\mathcal{B}(\tilde E^1)]_{2,\delta;T}+[\mathcal{B}(\tilde E^2)]_{2,\delta;T}\lesssim  \epsilon\delta^{\kappa_0},\\&
			\label{a24}\|\mathcal{B}(\tilde E^1)-	\mathcal{B}(\tilde E^2)\|_{G^{1,\delta}}\lesssim \delta^{\frac{1}{20}}\|\tilde E^1-\tilde E^2\|_{G^{1,\delta}}.
		\end{align}
	\end{lemma}
	
	We will prove this lemma in subsection \ref{subsection7.1} below. Assuming this lemma, we construct a unique solution to \eqref{a14} by Picard iteration. With $\widetilde{E}^{(0)}=0$ as before, we define the sequence $\{\tilde E^{(n)}\}_{n\ge0}$ by
	\begin{equation}\label{sec7eq10}
		\tilde E^{(n+1)}= \mathcal{B}(\tilde E^{(n)}), \qquad n\ge 0.
	\end{equation}
	
	By \eqref{a26} the vector-fields $\tilde{E}^{(n)}$ are well defined in the ball of radius $\approx_T\epsilon$ in $G^{2,\delta}$, for all $n\geq 0$. Using the bounds \eqref{a26} and \eqref{a24}, 
	\begin{equation}\label{picard-uniform-two}
	\sup_{n\geq0}\|\tilde E^{(n)}\|_{G^{2,\delta}}\lesssim_T \epsilon,\qquad \|\tilde E^{(n+1)}-\tilde E^{(n)}\|_{G^{1,\delta}}\lesssim_T \epsilon \delta^{\frac{n}{40}}\text{ for all }n\geq 0.
	\end{equation}
	\begingroup
	Therefore \(\widetilde E^{(n)}\) converges in \(G^{1,\delta}\) to a limit \(\widetilde E\).
	Passing to the limit in the characteristic formulas and weak-* in the highest derivatives gives
	\endgroup
	\begin{equation}\label{limit-order-two}
	 \mathcal{B}(\tilde{E})=\tilde{E},\qquad \|\tilde E\|_{G^{2,\delta}}\lesssim_T\epsilon,\qquad [\tilde E]_{2,\delta;T}
	\lesssim\epsilon\delta^{\kappa_0},\qquad\|\tilde E-\mathcal{B}(\tilde E^{(0)})\|_{G^{1,\delta}}\lesssim \epsilon\delta^{\frac1{50}}.
	\end{equation}
	\begingroup
	Here the highest derivatives are initially weak derivatives, with essential suprema;
	their continuity is established after \eqref{sec7eq13} below.
	The formula \eqref{sec7eq8} and kernel bounds \eqref{s7c12}, with the
	time-splitting argument of Step~5 in the proof of Lemma~\ref{keyLe}
	applied at any terminal time, give
	$\widetilde E^{reg}\in C([T,T+\delta];C_x^1)$.
	\endgroup

	Let $\widetilde E=(\widetilde E^{reg},\widetilde E^{\cos}_{non},\widetilde E^{\sin}_{non})$
	be the fixed point constructed above, define the field $E^\ast$ on $[0,T+\delta]$ as in
	\eqref{sec7eq3}, let $(X^\ast_{s,t},V^\ast_{s,t})$ denote its characteristic flow, and set
	\begin{equation}\label{sec7eq11}
		f(t,x,v):=
		f_0\bigl(X^\ast_{0,t}(x,v),V^\ast_{0,t}(x,v)\bigr)
		-\int_0^t E^\ast\bigl(s,X^\ast_{s,t}(x,v)\bigr)\cdot
		\nabla_v\mu\bigl(V^\ast_{s,t}(x,v)\bigr)\,ds .
	\end{equation}
	Differentiating along the flow shows that $f$ is the unique classical solution of the
	linear transport problem
	\begin{equation}\label{sec7eq12}
		\partial_tf+v\cdot\nabla_xf+E^\ast\cdot\nabla_vf=-E^\ast\cdot\nabla_v\mu,
		\qquad f(0)=f_0 ,
	\end{equation}
	and, by \eqref{sec7eq5} and the uniqueness for \eqref{sec7eq12}, $f$ coincides on $[0,T]$
	with the given solution. 
	\medskip
	
	{\bf{Step 3.}} It remains to prove that $E^\ast$ is the field generated by $f$,
	i.e.
	\begin{equation}\label{sec7eq13}
		E^\ast(t)=\nabla\Delta^{-1}\rho(t),\qquad
		\rho(t,x):=\int_{\R^2}f(t,x,v)\,dv,
		\qquad t\in[0,T+\delta].
	\end{equation}
	This is the only point at which the specific algebraic form of the map $\mathcal B$ is used,
	and it requires an argument because the moment identity \eqref{z5'} that enters the construction of
	$\mathcal B$ is available only for genuine solutions of \eqref{eq2}, while at this stage
	$E^\ast$ is merely a fixed point.

	Notice that every term in \eqref{de1.1} and \eqref{sec7eq8}--\eqref{sec7eq9} carries a factor
	$\nabla\Delta^{-1}$, and $E$ is a gradient field on $[0,T]$, so $E^\ast$ is a gradient
	field on $[0,T+\delta]$ and we may write
	\begin{equation}\label{sec7eq14}
		E^\ast(t)=\nabla\Delta^{-1}\gamma (t),
		\qquad
		\gamma:=\operatorname{div}E^\ast .
	\end{equation}
	
	The identity \eqref{sec7eq13} is thus equivalent to $\rho=\gamma$. As in \eqref{z1.5}--\eqref{z1.6}, we define
		\begin{equation}\label{sec7eq14.5}
		\begin{split}
		&\mathbf R_k(t,x):=\mathbf R_k(E^\ast)(t,x)=\int_{\R^2}v^{\otimes k}F(t,x,v)\,dv,\\
		&F(t,x,v):=f(t,x,v)+\int_0^tE^\ast(s,x-(t-s)v)\cdot\nabla_v\mu(v)\,ds,
		\end{split}
		\end{equation}
		for $k\in\{0,1,2\}$. Applying $\partial_t+v\cdot\nabla_x$ to the definition of $F$ and using \eqref{sec7eq12},
		\[
			\partial_tF+v\cdot\nabla_xF=-E^\ast(t,x)\cdot\nabla_vf(t,x,v).
		\]
\begingroup
The field bounds for $E^\ast$ allow the ray-integration argument
of Section~2, giving $|F|\lesssim_T\epsilon\langle v\rangle^{-5}$;
also, \eqref{sec7eq11} gives
$|f|+|\nabla_v f|\lesssim_T\epsilon\langle v\rangle^{-4}$.
Thus the second velocity moment of $F$ is finite, $\rho$ is bounded,
and the velocity integrations by parts below are justified.
\endgroup
		Integrating in $v$ and using
		$\int_{\R^2}\nabla_vf\,dv=0$ and $\int_{\R^2}v^j\partial_{v_i}f\,dv=-\delta_{ij}\rho$,
		we obtain
		\begin{equation}\label{sec7eq15}
			\partial_t\mathbf R_0=-\operatorname{div}\mathbf R_1,
			\qquad
			\partial_t\mathbf R_1=-\operatorname{div}\mathbf R_2+\rho\,E^\ast .
		\end{equation}
		Both identities hold for an arbitrary field $E^\ast$. This should be compared with
		\eqref{z5}--\eqref{z5'}, where the quadratic expression
		$\operatorname{div}(E\otimes E)-\frac12\nabla|E|^2$ appears in place of $\rho E^\ast$, and
		the two agree only when $E$ is the field generated by $\rho$.

     		We integrate now \eqref{sec7eq14.5} in $v$ to see that
		\[
			\rho(t,x)=\mathbf R_0(t,x)-\int_0^t\!\!\int_{\R^2}
			E^\ast(s,x-(t-s)v)\cdot\nabla_v\mu(v)\,dv\,ds .
		\]
		We take the Fourier transform in $x$ and use
		$\int_{\R^2}e^{-ir v\cdot\xi}\nabla_v\mu(v)\,dv=ir\xi\,\widehat\mu(r\xi)$ together with
		$\xi\cdot\widehat{E^\ast}(s,\xi)=-i\,\widehat{\gamma}(s,\xi)$, which follows from
		\eqref{sec7eq14}. Therefore
		\begin{equation}\label{sec7eq16}
			\widehat\rho(t,\xi)
			+\int_0^t(t-s)\,\widehat\mu\bigl((t-s)\xi\bigr)\,
			\widehat{\gamma}(s,\xi)\,ds
			=\widehat{\mathbf R_0}(t,\xi).
		\end{equation}
		This is the Volterra equation \eqref{Volterra}, except that the memory term carries
		$\gamma$ rather than $\rho$.

		Let $\kappa$ denote the solution of the Volterra equation with the same source, i.e.
		\begin{equation}\label{sec7eq17}
		\begin{split}
			&\widehat{\kappa}(t,\xi)+\int_0^t(t-s)\widehat\mu\bigl((t-s)\xi\bigr)\widehat{\kappa}(s,\xi)\,ds
			=\widehat{\mathbf R_0}(t,\xi),\\
			&\text{or equivalently}\qquad
			\kappa(t)=\mathbf R_0(t)-\int_0^t\sin(s)e^{-s|\nabla|}\mathbf R_0(t-s)\,ds,
		\end{split}
		\end{equation}
		the equivalence being the resolvent computation \eqref{z40}. The passage from \eqref{z40}
		to \eqref{z40.2} consists of two integrations by parts which use only the two identities
		$\partial_s\mathbf R_0=-\operatorname{div}\mathbf R_1$ and
		$\partial_s\mathbf R_1=-\operatorname{div}\mathbf R_2+\mathcal Q$, and which place
		$-\int_0^t\sin(t-s)\chi(|\nabla|)e^{-(t-s)|\nabla|}\operatorname{div}\mathcal Q(s)\,ds$ in the last line
		of \eqref{z40.2}. Denoting by $\mathcal L(t)$ the sum of the first four terms of the
		right-hand side of \eqref{z40.2}, which involve only $\mathbf R_0,\mathbf R_1,\mathbf R_2$
		and the data, we conclude from \eqref{sec7eq15} that
		\begin{equation}\label{sec7eq18}
			\kappa(t)=\mathcal L(t)
			-\int_0^t\sin(t-s)\chi(|\nabla|)e^{-(t-s)|\nabla|}\operatorname{div}\bigl(\rho E^\ast\bigr)(s)\,ds .
		\end{equation}
		
		On the other hand, applying $\operatorname{div}$ to \eqref{sec7eq3} and using the fixed-point
		relation $\mathcal B(\widetilde E)=\widetilde E$ together with
		\eqref{sec7eq8}--\eqref{sec7eq9}, we obtain the same expression with
		$\operatorname{div}(E^\ast\otimes E^\ast)-\frac12\nabla|E^\ast|^2$ in place of
		$\rho E^\ast$. Since $E^\ast$ is a gradient field,
		\[
			\operatorname{div}(E^\ast\otimes E^\ast)-\tfrac12\nabla|E^\ast|^2
			=(\operatorname{div}E^\ast)E^\ast=\gamma\,E^\ast ,
		\]
		and therefore
		\begin{equation}\label{sec7eq19}
			\gamma(t)=\mathcal L(t)
			-\int_0^t\sin(t-s)\chi(|\nabla|)e^{-(t-s)|\nabla|}\operatorname{div}\bigl(\gamma E^\ast\bigr)(s)\,ds .
		\end{equation}

		Finally we set $g:=\rho-\gamma$ and
		\begin{equation}\label{sec7eq20}
			\mathcal G(t):=\int_0^t\sin(t-s)\chi(|\nabla|)e^{-(t-s)|\nabla|}
			\operatorname{div}\bigl(gE^\ast\bigr)(s)\,ds .
		\end{equation}
		Subtracting \eqref{sec7eq18} from \eqref{sec7eq19} gives
		$\gamma-\kappa=\mathcal G$. Subtracting \eqref{sec7eq17} from \eqref{sec7eq16} gives
		\[
			\widehat\rho-\widehat{\kappa}
			=-\int_0^t(t-s)\widehat\mu\bigl((t-s)\xi\bigr)
			\bigl(\widehat{\gamma}-\widehat{\kappa}\bigr)(s,\xi)\,ds
			=-\int_0^t(t-s)\widehat\mu\bigl((t-s)\xi\bigr)\widehat{\mathcal G}(s,\xi)\,ds .
		\]
		Since $g=(\rho-\kappa)-(\gamma-\kappa)$, we conclude that
		\begin{equation}\label{sec7eq21}
			g(t)=-\mathcal G(t)-\int_0^t(t-s)\,\mathcal M_{t-s}\ast\mathcal G(s)\,ds,
			\qquad
			\mathcal M_r(z):=r^{-2}\mu(z/r),
		\end{equation}
		where we used that $\widehat\mu(r\xi)$ is the Fourier multiplier of convolution with
		$\mathcal M_r$. The symbol $i\xi\chi(|\xi|)$ is smooth and compactly supported, so $\chi(|\nabla|)\nabla$ has an
		integrable convolution kernel. Therefore, by \eqref{sec7eq20},
		\[
			\|\mathcal G(t)\|_{L^\infty_x}
			\lesssim\int_0^t\|g(s)\|_{L^\infty_x}\|E^\ast(s)\|_{L^\infty_x}\,ds .
		\]
		Using now \eqref{sec7eq21} and $\|\mathcal M_r\|_{L^1}=1$,
		\[
			\|g(t)\|_{L^\infty_x}
			\lesssim_T \sup_{s\in[0,t]}\|\mathcal G(s)\|_{L^\infty_x}.
		\]
		We combine the last two inequalities and use
		$\sup_{[0,T+\delta]}\|E^\ast(s)\|_{L^\infty_x}\lesssim_T\epsilon$. Since $g(t)=0$ for $t\in [0,T]$, it follows using Gronwall's inequality that $g$ vanishes on the entire interval $[0,T+\delta]$. Thus $\rho(t)=\gamma(t)$ for all $t\in[0,T+\delta]$, and the desired identity \eqref{sec7eq13} follows.
	\begingroup
	The regularity argument at the start of Section~\ref{section4} now gives
	$E^\ast\in C_tC^2_{\mathrm{loc}}$ and a $C^2$ phase flow.
	The continuous moment fields and the low-frequency formulas \eqref{sec7eq9}
	also give continuity of the required oscillatory derivatives;
	thus $\widetilde E\in G^{2,\delta}$ with classical derivatives.
	\endgroup
	\medskip

	\textbf{Step 4.} We prove now that our solution is unique in the set $\mathfrak S_{T,\delta}$ of classical solutions $(f,E)$ of \eqref{eq2} on
	$[0,T+\delta]$, with $f\in C([0,T+\delta];C^1_{x,v})$, which agree on $[0,T]$ with the given
	solution and for which the triple $\widetilde E:=\bigl(E^{reg},E^{\cos}_{non},E^{\sin}_{non}\bigr)\big|_{[T,T+\delta]}$
	defined by \eqref{de1.1} and \eqref{z40.4}, belongs to $G^{2,\delta}$ and satisfies $\|\widetilde E\|_{G^{2,\delta}}\leq C_T\epsilon$, where $C_T$ is the implicit constant in \eqref{a26}.

	If $(f,E)\in\mathfrak S_{T,\delta}$ then, by the derivation of
	\eqref{z40.2}--\eqref{z40.4} in Section~2, which is valid for any solution of \eqref{eq2},
	the field $\widetilde{E}$ satisfies $\mathcal B(\widetilde E)=\widetilde E$. By \eqref{a24}
	the map $\mathcal B$ has at most one fixed point in the ball of radius $C_T\epsilon$ of
	$G^{2,\delta}$. Hence $\widetilde E$,
	and therefore $E$ by \eqref{sec7eq3}, is uniquely determined; the characteristic flow of $E$
	is then uniquely determined, and so is $f$ by \eqref{t1}.
	The uniqueness asserted in Proposition~\ref{Localwell} among classical solutions
	with finite $A_{T+\delta}$ follows from the weighted energy argument in
	Step~1 of the proof of Theorem~\ref{mainthm}, since these bounds imply
	the weighted field condition used there.
	
	Finally, we make the constant in \eqref{bo2} explicit.
	For $T>0$, set $\beta:=23/6+\kappa_0$. The time exponents in
	\eqref{cbc2} have absolute value at most $11/6$, and
	$\langle t,x\rangle^{2+\kappa_0}
	\leq\langle t\rangle^{2+\kappa_0}
	\langle x\rangle^{2+2\kappa_0}.$
	For a universally small upper bound on $\delta$, each bootstrap
	weight at $t\in[T,T+\delta]$ is at most twice its value at $T$.
	Decomposing each field derivative into its value at $T$ and its
	increment, and using \eqref{limit-order-two}, therefore gives
	\[
	A_{T+\delta}\leq2A_T+
	C_{\mathrm{loc}}\langle T\rangle^\beta
	\epsilon\delta^{\kappa_0},
	\]
	with $C_{\mathrm{loc}}$ independent of $T$, $\epsilon$, and the data.
	Choose $\delta$ so that
	$C_{\mathrm{loc}}\langle T\rangle^\beta\delta^{\kappa_0}\leq1$.
	Since $A_T\leq\epsilon$, this yields $A_{T+\delta}\leq3\epsilon$.
	For $T=0$, the mapping bound in \eqref{limit-order-two} and the
	definition \eqref{cbc2}, with $\delta\leq1$, give
	$A_\delta\leq C_{\mathrm{init}}\epsilon$ for a fixed constant
	$C_{\mathrm{init}}$. Thus \eqref{bo2} holds with
	$C_1:=\max\{3,C_{\mathrm{init}}\}$.

	The other smallness conditions on $\delta$ in the local construction
	involve positive powers of $\delta$ and constants bounded by powers
	of $\langle T\rangle$. Hence, taking fixed $N>0$ sufficiently large
	and $c>0$ sufficiently small, with $N\kappa_0\geq\beta$, we may choose $
	\delta(T)=c\langle T\rangle^{-N}.$
	In particular, $\inf_{0\leq T\leq R}\delta(T)
	\geq c\langle R\rangle^{-N}>0$ for every finite $R$.
	This completes the proof of the proposition.

	\subsection{Proof of Lemma \ref{keyLe}}\label{subsection7.1} Recall our standing assumption that $T$ is fixed and $\delta$ is sufficiently small relative to $T$. Using the definitions, to prove \eqref{a26} we must show that
	\begin{equation}\label{s7B4}
		\sum_{n\leq2}|\nabla_x^n\mathcal B^{reg}(\tilde E^j)|
		+\sum_{n\leq3}|\nabla_x^n\mathcal B^{\Box}(\tilde E^j)|
		+\sum_{n\leq2}|\nabla_x^n\partial_t\mathcal B^{\Box}(\tilde E^j)|
		\lesssim_T\epsilon\langle x\rangle^{-2-2\kappa_0},
	\end{equation}
	for $j\in\{1,2\}$ and $\Box\in \{\cos,\sin\}$, while for \eqref{a24} we must show that
	\begin{equation}\label{s7B5}
		\sum_{n\leq1}|\nabla_x^n\mathrm d\mathcal B^{reg}|
		+\sum_{n\leq2}\sum_{\Box}|\nabla_x^n\mathrm d\mathcal B^{\Box}|
		+\sum_{n\leq1}\sum_{\Box}|\nabla_x^n\partial_t\mathrm d\mathcal B^{\Box}|
		\lesssim\delta^{\frac1{20}}\|\tilde{E}^1-\tilde{E}^2\|_{G^{1,\delta}}\,\langle x\rangle^{-2-2\kappa_0},
	\end{equation}
	where $\mathrm d\mathcal B^{reg}:=\mathcal B^{reg}(\tilde E^1)-\mathcal B^{reg}(\tilde E^2)$, $\mathrm d\mathcal B^{\Box}:=\mathcal B^{\Box}(\tilde E^1)-\mathcal B^\Box(\tilde E^2)$. We will prove these bounds in Steps 3 and 4 below, while the remaining bounds \eqref{a24'} are proved in Step 5 below.
	
Throughout the proof of the lemma we let $E^j$ denote the field
	\eqref{sec7eq3} associated with $\tilde E^j$, write $(X^j_{s,t},V^j_{s,t})$ for its
	characteristic flow, and let $\mathcal I_k(E^j)$, $\mathcal R_k(E^j)$, $\mathbf R_k(E^j)$ denote the
	moment quantities \eqref{sec7eq7}. As in \eqref{newDefY}, set
	\begin{equation}\label{s7c2}
	\begin{split}
		&\widetilde Y^j_{s,t}(x,v):=Y^j_{s,t}(x-tv,v)=X^j_{s,t}(x,v)-x+(t-s)v,\\
		&\widetilde W^j_{s,t}(x,v):=W^j_{s,t}(x-tv,v)=V^j_{s,t}(x,v)-v.
	\end{split}
	\end{equation}
	By \eqref{Yinteq},
	\begin{equation}\label{s7c3}
		\widetilde Y^j_{s,t}(x,v)=\int_s^t(r-s)E^j\bigl(r,X^j_{r,t}(x,v)\bigr)\,dr,
		\quad
		\widetilde W^j_{s,t}(x,v)=-\int_s^tE^j\bigl(r,X^j_{r,t}(x,v)\bigr)\,dr .
	\end{equation}
	Finally, for $t\in[0,T+\delta]$ and $v\neq 0$ we let
	\begin{equation}\label{s7c4}
		\mathfrak I_t(x,v):=\int_0^t\langle x-\sigma v\rangle^{-2}\,d\sigma
		\lesssim\min\Bigl\{t,\ |v|^{-1}\langle D_v(x)\rangle^{-\frac12}\Bigr\},
		\qquad
		D_v(x):=|x|^2-\frac{(x\cdot v)^2}{|v|^2},
	\end{equation}
	see \eqref{a40.1}. We prove first two elementary lemmas. 
	
	\begin{lemma}\label{convLe}
		Let $b\in(2,3]$ and $A\geq0$.

		(i) If $K\in L^1(\R^2)$ and $|g(y)|\leq A\langle y\rangle^{-b}$ for all $y$, then
		\begin{equation*}
			|(K\ast g)(x)|\lesssim A\,\Lambda_b(K)\,\langle x\rangle^{-b},
			\qquad
			\Lambda_b(K):=\|K\|_{L^1(\R^2)}+\sup_{R\geq1}R^{b}\sup_{|z|\geq R}|K(z)| .
		\end{equation*}

		(ii) If $K:[0,t]\times\R^2\to\C$ and $|g(s,y)|\leq A\langle y\rangle^{-b}$ for all
		$(s,y)\in[0,t]\times\R^2$, then
		\begin{equation*}
			\Bigl|\int_0^t \int_{\R^2} K(t-s,x-y) g(s,y)\,dyds\Bigr|
			\lesssim A\,\Lambda_b(K)\,\langle x\rangle^{-b},
		\end{equation*}
		where now
		\[
			\Lambda_b(K):=\int_0^t\|K(s,\cdot)\|_{L^1(\R^2)}\,ds
			+\sup_{R\geq1}R^{b}\int_0^t\sup_{|z|\geq R}|K(s,z)|\,ds.
		\]
	\end{lemma}

	\begin{proof} We prove (ii); the proof of (i) is identical. For $|x|\leq2$ the bound follows from
	the first term of $\Lambda_b(K)$, so we may assume $|x|\geq2$. On the set $\{|y|\geq|x|/2\}$ we
	have $\langle y\rangle^{-b}\lesssim\langle x\rangle^{-b}$, and the corresponding contribution is
	$\lesssim A\langle x\rangle^{-b}\int_0^t\|K(s,\cdot)\|_{L^1}ds$. On the set $\{|y|<|x|/2\}$ we have
	$|x-y|\geq|x|/2\geq1$, hence $|K(t-s,x-y)|\leq\sup_{|z|\geq|x|/2}|K(t-s,z)|$. Since
	$\int_{\R^2}\langle y\rangle^{-b}\,dy\lesssim1$, the corresponding contribution is
	$\lesssim A|x|^{-b}\sup_{R\geq1}R^b\int_0^t\sup_{|z|\geq R}|K(s,z)|\,ds$.
	\end{proof}

	The second lemma is a refinement of \eqref{a40} which produces a power of $\eta$ when one of the two
	time integrations is restricted to a short interval.

	\begin{lemma}\label{intLe} For $0<\eta\leq t$, $\Lambda\geq1$ and $x\in\R^2$ we have
		\begin{align}
			&\int_{\R^2}\mathfrak I_t(x,v)\,\mathfrak I_{\eta}(x,v)\,\frac{dv}{\langle v\rangle^2}
			\lesssim\eta^{\frac18}\langle t\rangle^{3}\langle x\rangle^{-\frac{11}{4}},
			\label{s7c7}\\
			&\int_{\{|v|\geq\Lambda\}}\mathfrak I_t(x,v)^2\,\frac{dv}{\langle v\rangle^2}
			\lesssim\Lambda^{-\frac18}\langle t\rangle^{4}\langle x\rangle^{-\frac{21}{8}} .
			\label{s7c8}
		\end{align}
	\end{lemma}

	\begin{proof} We argue as in the proof of \eqref{a40}. If $|x|\geq2\langle t\rangle|v|$ then
	$|x-\sigma v|\gtrsim|x|$ for $\sigma\in[0,t]$, so $\mathfrak I_t\lesssim t\langle x\rangle^{-2}$ and
	$\mathfrak I_{\eta}\lesssim\eta\langle x\rangle^{-2}$; since
	$\int_{\{|v|\leq|x|/(2\langle t\rangle)\}}\langle v\rangle^{-2}dv\lesssim\log(2+|x|)$, the
	contribution of this region to \eqref{s7c7} is $\lesssim t\eta\langle x\rangle^{-4}\log(2+|x|)$,
	which is bounded by the right-hand side of \eqref{s7c7}. If $|x|\leq2\langle t\rangle|v|$, we
	interpolate the two bounds in \eqref{s7c4} to get
	\[
		\mathfrak I_t(x,v)\mathfrak I_{\eta}(x,v)
		\lesssim t^{\frac18}\eta^{\frac18}|v|^{-\frac74}\langle D_v(x)\rangle^{-\frac78}
		\lesssim t^{\frac18}\eta^{\frac18}|v|^{-\frac74}\langle D_v(x)\rangle^{-\frac34}.
	\]
	Then we conclude, using polar coordinates $v=r\theta$ and \eqref{a40.2}, that this
	contribution is
	\[
		\lesssim t^{\frac18}\eta^{\frac18}\langle x\rangle^{-1}
		\int_{r\gtrsim|x|/\langle t\rangle}r^{-\frac34}\langle r\rangle^{-2}\,dr
		\lesssim\eta^{\frac18}\langle t\rangle^{3}\langle x\rangle^{-\frac{11}{4}} .
	\]
	The bound \eqref{s7c8} follows in the same way, by replacing $\langle v\rangle^{-2}$ by
	$\Lambda^{-\frac{1}{8}}\langle v\rangle^{-\frac{15}{8}}$ on $\{|v|\geq\Lambda\}$.
	\end{proof}

	For integers $j_1,j_2\geq0$ and $\varsigma\in\{\chi, 1-\chi\}$ let
	$K^{\varsigma}_{j_1,j_2}(t,\cdot)$ denote the convolution kernels of the operators
	$\partial_t^{j_1}\nabla_x^{j_2}m_{-1}(\nabla)\varsigma(|\nabla|)e^{-t|\nabla|}$, where $m_{-1}$ is a homogeneous symbol of degree $-1$. Lemma \ref{keres} gives
	\begin{equation}\label{s7c10}
		\bigl|K^{\chi}_{j_1,j_2}(t,z)\bigr|\lesssim\langle t,z\rangle^{-1-j_1-j_2},
		\qquad
		\bigl|K^{1-\chi}_{j_1,j_2}(t,z)\bigr|\lesssim(t+|z|)^{-1-j_1-j_2}\langle t,z\rangle^{-7}.
	\end{equation}
	Therefore, in the notation of Lemma \ref{convLe} with $t\leq T+\delta$ and $b\in(2,3]$,
	\begin{equation}\label{s7c11}
		\Lambda_b\bigl(K^{\chi}_{j_1,j_2}\bigr)\lesssim\langle T\rangle\quad\text{if }j_1+j_2\geq2,
		\qquad
		\Lambda_b\bigl(K^{1-\chi}_{j_1,j_2}\bigr)\lesssim1\quad\text{if }j_1+j_2\leq1,
	\end{equation}
	and $\Lambda_b\bigl(K^{\chi}_{j_1,j_2}(s,\cdot)\bigr)\lesssim1$ uniformly in $s\geq0$ if
	$j_1+j_2\geq2$. We shall also use the bounds
	\begin{equation}\label{s7c12}
	\begin{split}
		&\Lambda_b\bigl(K^{\chi}_{j_1,j_2}(s)\mathbf 1_{s\leq\eta}\bigr)\lesssim\eta\quad\text{ if }j_1+j_2\geq 2,\\
		&\Lambda_b\bigl(K^{1-\chi}_{0,1}(s)\mathbf 1_{s\leq\eta}\bigr)\lesssim\eta\log(2/\eta),\qquad
		\Lambda_b\bigl(K^{1-\chi}_{1,1}(s)\mathbf 1_{s\geq\eta}\bigr)\lesssim\log(2/\eta),
		\end{split}
	\end{equation}
	valid for $\eta\in(0,1]$. These bounds follow from \eqref{s7c10} by direct computation.
	\medskip
	
	\textbf{Step 1.} Since $E^j=E$ on $[0,T]$, the bounds \eqref{a35} give
	\begin{equation}\label{s7c13}
		\sum_{n\leq2}|\nabla_x^nE^j(t,x)|\lesssim_T\epsilon\,\langle x\rangle^{-2},
		\qquad t\in[0,T],\,j\in\{1,2\}.
	\end{equation}
	For $t\in(T,T+\delta]$ the definition \eqref{sec7eq3}, the hypothesis \eqref{a27}, and Lemma
	\ref{eslinear} give
	\begin{equation}\label{s7c14}
	\begin{split}
		&\sum_{n\leq2}|\nabla_x^nE^j(t,x)|\lesssim \epsilon\delta^{-1/1000}\langle x\rangle^{-2},\\
		&\sum_{n\leq1}\bigl|\nabla_x^n(E^1-E^2)(t,x)\bigr|\lesssim \|\tilde{E}^1-\tilde{E}^2\|_{G^{1,\delta}}\,\langle x\rangle^{-2-2\kappa_0},
		\end{split}
	\end{equation}
	while $E^1-E^2$ vanishes identically on $[0,T]$. 
	
	We prove first bounds on the characteristics. Precisely, we claim that if $0\leq s\leq t\leq T+\delta$, $j\in\{1,2\}$, and $n\leq2$, then
	\begin{align}
		&\bigl|\nabla_x^n\widetilde Y^j_{s,t}(x,v)\bigr|+\bigl|\nabla_x^n\widetilde W^j_{s,t}(x,v)\bigr|
		\lesssim_T\epsilon\,\mathfrak I_t(x,v)+\epsilon\delta^{-1/1000}\mathfrak I_\delta(x,v),
		\label{s7c16}\\
		&\bigl|\widetilde Y^j_{s,t}(x,v)\bigr|+\bigl|\widetilde W^j_{s,t}(x,v)\bigr|\lesssim\epsilon,
		\label{s7c17}
	\end{align}
	the implicit constant in \eqref{s7c17} being absolute. Moreover, if $T\leq s\leq t\leq T+\delta$ then
	\begin{equation}\label{s7c18}
		\bigl|\nabla_x^n\widetilde Y^j_{s,t}(x,v)\bigr|+\bigl|\nabla_x^n\widetilde W^j_{s,t}(x,v)\bigr|
		\lesssim_T\epsilon\delta^{-1/1000}\mathfrak I_\delta(x,v),\qquad n\leq2 .
	\end{equation}
	In particular
	\begin{equation*}
		\bigl\langle X^j_{s,t}(x,v)\bigr\rangle\approx\langle x-(t-s)v\rangle,
		\qquad
		\bigl\langle V^j_{s,t}(x,v)\bigr\rangle\approx\langle v\rangle,
		\qquad 0\leq s\leq t\leq T+\delta .
	\end{equation*}
	Finally, we have the following bounds for the differences: for $0\leq s\leq t\leq T+\delta$ and $n\leq1$,
	\begin{equation}\label{s7c21}
		\bigl|\nabla_x^n\bigl(\widetilde Y^1_{s,t}-\widetilde Y^2_{s,t}\bigr)\bigr|
		+\bigl|\nabla_x^n\bigl(\widetilde W^1_{s,t}-\widetilde W^2_{s,t}\bigr)\bigr|
		\lesssim_T\|\tilde{E}^1-\tilde{E}^2\|_{G^{1,\delta}}\mathfrak I_\delta(x,v) .
	\end{equation}

	To prove these bounds assume first $s\leq t\in [T, T+\delta]$. Notice that if 
	$|\widetilde Y^j_{r,t}(x,v)|\leq 1$ for $r\in[s,t]$ we have
	$\langle X^j_{r,t}\rangle\approx\langle x-(t-r)v\rangle$, so that \eqref{s7c3} and \eqref{s7c14} give
	\[
		\bigl|\widetilde Y^j_{s,t}(x,v)\bigr|+\bigl|\widetilde W^j_{s,t}(x,v)\bigr|
		\lesssim\epsilon\delta^{\frac{-1}{1000}}\int_s^t\langle x-(t-r)v\rangle^{-2}dr\lesssim\epsilon\delta^{\frac{-1}{1000}}
		\mathfrak I_\delta(x,v)\lesssim \epsilon\delta^{\frac12}.
	\]
	A continuity bootstrap argument in $s,t$ therefore gives \eqref{s7c18} for
	$n=0$. Differentiating \eqref{s7c3} in $x$ at most twice, the terms containing
	$\nabla_x^n\widetilde Y^j_{r,t}$ carry a factor of
	$\int_s^t|\nabla_xE^j(r,X^j_{r,t})|\,dr\lesssim\epsilon\delta^{1/2}$, so can be absorbed; this gives
	\eqref{s7c18} for $n\leq2$. 
	
	The bounds \eqref{s7c21} follow in the same way when $s\leq t\in [T, T+\delta]$. We start from
	\[
		\bigl|E^1(r,X^1_{r,t})-E^2(r,X^2_{r,t})\bigr|
		\lesssim \|\tilde{E}^1-\tilde{E}^2\|_{G^{1,\delta}}\langle x-(t-r)v\rangle^{-2}
		+\epsilon\delta^{\frac{-1}{1000}}\langle x-(t-r)v\rangle^{-2}\bigl|\widetilde Y^1_{r,t}-\widetilde Y^2_{r,t}\bigr|,
	\]
	which follows from \eqref{s7c14} and \eqref{s7c18}. Then we use the formulas \eqref{s7c3} and notice that we can absorb  the last term.

	Assume now $0\leq s\leq T\leq t\leq T+\delta$ and set
	$(y^j,w^j):=\bigl(X^j_{T,t}(x,v),V^j_{T,t}(x,v)\bigr)$. Since $X^j_{s,t}=X_{s,T}(y^j,w^j)$ and
	$V^j_{s,t}=V_{s,T}(y^j,w^j)$, the formulas \eqref{a31}--\eqref{a32} show that
	\begin{equation}\label{s7c22}
	\begin{split}
		&\widetilde Y^j_{s,t}(x,v)=\widetilde Y^j_{T,t}(x,v)-(T-s)\,\widetilde W^j_{T,t}(x,v)
		+\widetilde Y_{s,T}(y^j,w^j),\\
		&\widetilde W^j_{s,t}(x,v)=\widetilde W^j_{T,t}(x,v)+\widetilde W_{s,T}(y^j,w^j),
	\end{split}
	\end{equation}
	which is equivalent to \eqref{a31'}. The first two terms in each line are
	$\lesssim_T\epsilon\delta^{-1/1000}\mathfrak I_\delta(x,v)\lesssim\epsilon\delta^{1/2}$, due to \eqref{s7c18}. For the last terms, we first compare the rays. By \eqref{s7c18},
for $0\leq\sigma\leq T$,
\[
|y^j-\sigma w^j-(x-(t-T+\sigma)v)|\mathrel{\lesssim_T}\epsilon\delta^{1/2}.
\]
Consequently, $\mathfrak I_T(y^j,w^j)\lesssim\mathfrak I_t(x,v)$. 

For the solution on $[0,T]$, it follows from \eqref{sec5eq4} and
\eqref{n13}--\eqref{n17b} that for any $y,w\in\R^2$
\begin{align*}
\sum_{n\leq2}
|\nabla_y^n(\widetilde Y_{s,T},\widetilde W_{s,T})(y,w)|
&\lesssim_T\epsilon\,\mathfrak I_T(y,w),\\
\sum_{a+b\leq2}
|\nabla_y^a\nabla_w^b
(\widetilde Y_{s,T},\widetilde W_{s,T})(y,w)|
&\lesssim_T\epsilon.
\end{align*}

By \eqref{s7c18}, $|\nabla_xy^j|\lesssim1$ and 
$|\nabla_xw^j|+|\nabla_x^2y^j|+|\nabla_x^2w^j|\lesssim_T\epsilon\delta^{-1/1000}\mathfrak I_\delta(x,v)$.
When differentiating the old-flow compositions at most twice, the terms
containing only first derivatives of $y^j$ are controlled by the first
bound above. Every other term contains one of these small derivatives
and is controlled by the second bound. Thus
\[
\sum_{n\leq2}
\left|\nabla_x^n
\big[(\widetilde Y_{s,T},\widetilde W_{s,T})(y^j,w^j)\big]\right|
\lesssim_T
\epsilon\mathfrak I_t(x,v)
+\epsilon^2\delta^{-1/1000}\mathfrak I_\delta(x,v).
\]
Combining this with \eqref{s7c22} proves \eqref{s7c16}.
The absolute bound \eqref{s7c17} follows instead from \eqref{n13}
and the bounds on the explicit short-interval terms.

Finally, to prove the bounds \eqref{s7c21} when $0\leq s\leq T\leq t\leq T+\delta$, notice that the short-interval version of \eqref{s7c21} gives
\[
\sum_{n\leq1}
|\nabla_x^n(y^1-y^2,w^1-w^2)|
\lesssim_T
\|\tilde E^1-\tilde E^2\|_{G^{1,\delta}}\mathfrak I_\delta(x,v).
\]
The mean value theorem and the mixed-derivative bounds above therefore
control the difference of the old-flow compositions, together with
its first $x$-derivatives, by the same right-hand side multiplied by
$\epsilon$. Subtracting the two identities \eqref{s7c22} proves
\eqref{s7c21}. When $t\leq T$, the two flows coincide, and the remaining
bounds follow directly from \eqref{sec5eq4} and \eqref{n13}.
\medskip

\textbf{Step 2.} We prove that, for $k\in\{0,1,2\}$,
\begin{equation}\label{s7c23}
\sum_{n\leq2}|\nabla_x^n\mathbf R_k(E^j)(t,x)|
\lesssim_T\epsilon\langle x\rangle^{-2-2\kappa_0}.
\end{equation}
For $t\leq T$ this follows from \eqref{Z16-new} and \eqref{z34b}.
Assume therefore $t\in[T,T+\delta]$.

We estimate first the initial-data component. Set
\begin{equation}\label{s7c24}
g(y,w):=f_0(X_{0,T}(y,w),V_{0,T}(y,w)).
\end{equation}
The flow property gives
\[
\mathcal I_k(E^j)(t,x)=\int_{\R^2}v^{\otimes k}
g(X^j_{T,t}(x,v),V^j_{T,t}(x,v))\,dv.
\]
By \eqref{a16'} and \eqref{n13}--\eqref{n17b},
\begin{equation}\label{s7c25}
\sum_{n\leq2}|\nabla_{y,w}^ng(y,w)|
\lesssim_T[f_0]\langle y-Tw,w\rangle^{-3-20\kappa_0}
\langle w\rangle^{-3-\kappa_0}.
\end{equation}
Moreover, \eqref{s7c18} gives
\[
|X^j_{T,t}(x,v)-TV^j_{T,t}(x,v)-(x-tv)|
+|V^j_{T,t}(x,v)-v|\lesssim\epsilon,
\]
and $|\nabla_x^n(X^j_{T,t},V^j_{T,t})(x,v)|\lesssim1$
for $n=1,2$. Consequently,
\begin{equation}\label{s7c26}
\sum_{n\leq2}|\nabla_x^n\mathcal I_k(E^j)(t,x)|\lesssim_T[f_0]\int_{\R^2}
\langle v\rangle^{k-3-\kappa_0}
\langle x-tv,v\rangle^{-3-20\kappa_0}\,dv\lesssim_T\epsilon\langle x\rangle^{-2-2\kappa_0}.
\end{equation}
Indeed, since $k\leq2$, the integrand is bounded by
$C_T\langle x\rangle^{-2-2\kappa_0}
\langle v\rangle^{-2-19\kappa_0}$.

For the nonlinear component we follow the proof of
Lemma~\ref{Lem1}. We first justify the velocity change of variables.
By the hypotheses of Proposition~\ref{Localwell}, \eqref{a23}, and \eqref{s7c14},
\[
\int_0^{T+\delta}
\big(\|E^j(r)\|_{L^\infty}
+r\|\nabla_xE^j(r)\|_{L^\infty}\big)\,dr
\lesssim\epsilon.
\]
The integral equations \eqref{s7c3} and their first $v$-derivatives
therefore give
\[
\frac{|\widetilde Y^j_{s,t}(x,v)|
+|\nabla_v\widetilde Y^j_{s,t}(x,v)|}{t-s}
\lesssim\epsilon.
\]
Thus the construction in Lemma~\ref{lem:Psi} applies to $E^j$;
denote the resulting map by $\Psi^j_{s,t}$.

Using \eqref{s7c13}, \eqref{s7c14}, and \eqref{s7c17}
in the same integral equations gives the more precise bounds
\[
\frac{|\widetilde Y^j_{s,t}(x,v)|
+|\nabla_v\widetilde Y^j_{s,t}(x,v)|}{t-s}
+|\widetilde W^j_{s,t}(x,v)|
\lesssim_T\epsilon\big(
\mathbf1_{s<T}\mathfrak I_t(x,v)
+\delta^{-1/1000}\mathfrak I_\delta(x,v)\big).
\]
Since $|\Psi^j_{s,t}(x,v)-v|\lesssim\epsilon$, replacing $v$
by $\Psi^j_{s,t}(x,v)$ preserves these ray integrals up to
a constant depending polynomially on $\langle T\rangle$.
Let $\Sigma^j_{s,t}$ denote the multiplier \eqref{sec5eq7}
associated with $E^j$, $b=v^{\otimes k}$, and
$\omega=\nabla\mu$, componentwise.
The formula and \eqref{sigma-basic} imply
\begin{equation}\label{sec7-Sigma}
|\Sigma^j_{s,t}(x,v)|
\lesssim_T\epsilon\langle v\rangle^{-2}
\big(\mathbf1_{s<T}\mathfrak I_t(x,v)
+\delta^{-1/1000}\mathfrak I_\delta(x,v)\big).
\end{equation}

We use superscripts $<$ and $>$ to indicate restrictions
of the time integration to $[0,T]$ and $[T,t]$, respectively.
After at most two $x$-derivatives, the principal terms in
$\mathcal R_k(E^j)$ are
$\mathcal R_k(\nabla_x^nE^j,E^j)$, estimated using
\eqref{q1} and \eqref{sec7-Sigma}.
Every remaining term contains a factor
$\nabla_x^a\widetilde Y^j_{s,t}$ or
$\nabla_x^a\widetilde W^j_{s,t}$, $1\leq a\leq2$,
and involves at most two derivatives of $E^j$.
Thus 
\begin{equation}\label{sec7-R-pieces}
\begin{split}
\sum_{n\leq2}|\nabla_x^n\mathcal R_k^<(E^j)(t,x)|
&\lesssim_T\epsilon^2\int_{\R^2}\mathfrak I_t(x,v)
\big(\mathfrak I_t(x,v)+\delta^{-1/1000}\mathfrak I_\delta(x,v)\big)
\frac{dv}{\langle v\rangle^2}\lesssim_T\epsilon^2\langle x\rangle^{-11/4},\\
\sum_{n\leq2}|\nabla_x^n\mathcal R_k^>(E^j)(t,x)|
&\lesssim_T\epsilon^2\delta^{-2/1000}
\int_{\R^2}\mathfrak I_\delta(x,v)^2
\frac{dv}{\langle v\rangle^2}\lesssim\epsilon^2\delta^{1/5}\langle x\rangle^{-11/4},
\end{split}
\end{equation}
using \eqref{s7c13}--\eqref{s7c18}, \eqref{a40}, and \eqref{s7c7}. Together with \eqref{s7c26}, these bounds prove \eqref{s7c23}.

For later use, write
\[
\mathcal C^j_{k,s,t}(x,v):=v^{\otimes k}
\big[E^j(s,x-(t-s)v)\cdot\nabla\mu(v)
-E^j(s,X^j_{s,t}(x,v))\cdot\nabla\mu(V^j_{s,t}(x,v))\big].
\]
The mean value theorem and \eqref{s7c17} give
\begin{equation}\label{s7c27}
\begin{split}
|\mathcal C^j_{k,s,t}(x,v)|
&\lesssim\langle v\rangle^{k-4}
\langle x-(t-s)v\rangle^{-2}
\|\langle\cdot\rangle^2(E^j,\nabla_xE^j)(s)\|_{L^\infty}\\
&\qquad\times
\big(|\widetilde Y^j_{s,t}(x,v)|
+|\widetilde W^j_{s,t}(x,v)|\big).
\end{split}
\end{equation}
\medskip

\textbf{Step 3.} We prove \eqref{a26}.
For $\nabla_x^n\mathcal B_1^{reg}$, $n\leq2$, put the
$n$ spatial derivatives on the moment fields, leaving kernels
of the form $K^{1-\chi}_{0,1}$.
Lemma~\ref{convLe}, \eqref{s7c11}, and \eqref{s7c23} give
\begin{equation}\label{s7c29}
\sum_{n\leq2}|\nabla_x^n\mathcal B_1^{reg}(\tilde E^j)(t,x)|
\lesssim_T\epsilon\langle x\rangle^{-2-2\kappa_0}.
\end{equation}

For $\mathcal B_2^{reg}$, put all spatial derivatives on
the low-frequency kernels, which have at least two derivatives.
The quadratic source is bounded by
$C_T\epsilon^2\langle x\rangle^{-4}$ on $[0,T]$ and by
$C\epsilon^2\delta^{-2/1000}\langle x\rangle^{-4}$ on $[T,t]$.
Applying \eqref{s7c11} and \eqref{s7c12} with $b=3$ yields
\begin{equation}\label{s7c30}
\sum_{n\leq2}|\nabla_x^n\mathcal B_2^{reg}(\tilde E^j)(t,x)|
\lesssim_T\epsilon^2
\big(1+\delta^{1-2/1000}\big)\langle x\rangle^{-3}
\lesssim_T\epsilon^2\langle x\rangle^{-3}.
\end{equation}

For the oscillatory components, define
\[
\mathcal S^j(t):=\chi(|\nabla|)\nabla\Delta^{-1}
\big(|\nabla|^2\mathbf R_0(E^j)
-2|\nabla|\operatorname{div}\mathbf R_1(E^j)
+\operatorname{div}^2\mathbf R_2(E^j)\big)(t).
\]
With $b_{\cos}(s)=-\sin s$ and $b_{\sin}(s)=\cos s$,
\eqref{sec7eq9} becomes
\begin{equation}\label{s7c31}
\begin{split}
\mathcal B^\Box(\tilde E^j)(t)
&=\int_0^t b_\Box(s)e^{-(t-s)|\nabla|}\mathcal S^j(s)\,ds,\\
\partial_t\mathcal B^\Box(\tilde E^j)(t)
&=b_\Box(t)\mathcal S^j(t)
-|\nabla|\mathcal B^\Box(\tilde E^j)(t).
\end{split}
\end{equation}
All these operators have low-frequency kernels with at least
two derivatives in the notation of \eqref{s7c10}.
Putting spatial derivatives on the kernels and using
\eqref{s7c11} and \eqref{s7c23}, we obtain
\begin{equation}\label{s7c32}
\begin{split}
\sum_{n\leq3}&\big(
|\nabla_x^n\mathcal S^j(t,x)|
+|\nabla_x^n\mathcal B^\Box(\tilde E^j)(t,x)|\big)\\
&+\sum_{n\leq3,\,a\in\{0,1\}}
|\nabla_x^n|\nabla|^a
\partial_t\mathcal B^\Box(\tilde E^j)(t,x)|
\lesssim_T\epsilon\langle x\rangle^{-2-2\kappa_0}.
\end{split}
\end{equation}
The additional bounds with $a=1$ will be used in Step 5.
Combining \eqref{s7c29}, \eqref{s7c30}, and \eqref{s7c32}
proves \eqref{s7B4}, hence \eqref{a26}.
\medskip

\textbf{Step 4.} We prove \eqref{a24}. Set
\[
\mathcal D:=\|\tilde E^1-\tilde E^2\|_{G^{1,\delta}},
\qquad
\mathrm d\mathbf R_k:=\mathbf R_k(E^1)-\mathbf R_k(E^2).
\]
The differences vanish for $t\leq T$, so assume
$t\in[T,T+\delta]$.
The mean value theorem, \eqref{s7c25}, \eqref{s7c21},
and $\mathfrak I_\delta\leq\delta$ give
\[
\sum_{n\leq1}
|\nabla_x^n(\mathcal I_k(E^1)-\mathcal I_k(E^2))(t,x)|
\lesssim_T\epsilon\mathcal D\delta
\langle x\rangle^{-2-2\kappa_0}.
\]

For the nonlinear component on $[0,T]$, the free-transport
terms cancel. Applying the mean value theorem to the
characteristic terms, and using \eqref{s7c13} and
\eqref{s7c21}, gives
\[
\begin{split}
\sum_{n\leq1}
|\nabla_x^n(\mathcal R_k^<(E^1)-\mathcal R_k^<(E^2))(t,x)|&\lesssim_T\epsilon\mathcal D
\int_{\R^2}\mathfrak I_t(x,v)\mathfrak I_\delta(x,v)
\frac{dv}{\langle v\rangle^2}\\
&\lesssim_T\epsilon\mathcal D\delta^{1/8}
\langle x\rangle^{-11/4}.
\end{split}
\]

On $[T,t]$, use the decomposition
\[
\mathcal R_k^>(E^1)-\mathcal R_k^>(E^2)
=\mathcal R_k^>(E^1-E^2,E^1)
+\mathcal R_k^>(E^2,E^1)-\mathcal R_k^>(E^2,E^2).
\]
For the first term, argue as in Step 2, using
\eqref{q1} and \eqref{sec7-Sigma} for the principal terms
after differentiation. This requires one derivative
of $E^1-E^2$. In the last two terms the free-transport contributions cancel;
the mean value theorem uses at most two derivatives of $E^2$
and one derivative of the differences of the characteristics.
Thus \eqref{s7c14}, \eqref{s7c18}--\eqref{s7c21} imply
\[
\begin{split}
\sum_{n\leq1}
|\nabla_x^n(\mathcal R_k^>(E^1)-\mathcal R_k^>(E^2))(t,x)|&\lesssim_T\epsilon\mathcal D\delta^{-1/1000}
\int_{\R^2}\mathfrak I_\delta(x,v)^2
\frac{dv}{\langle v\rangle^2}\\
&\lesssim_T\epsilon\mathcal D\delta^{1/4-1/1000}
\langle x\rangle^{-11/4}.
\end{split}
\]
Consequently,
\begin{equation}\label{s7c33}
\sum_{n\leq1}|\nabla_x^n\mathrm d\mathbf R_k(t,x)|
\lesssim_T\mathcal D\delta^{1/8}
\langle x\rangle^{-2-2\kappa_0}.
\end{equation}

We now repeat the kernel estimates of Step 3.
For $\mathcal B_1^{reg}$ put at most one derivative on
$\mathrm d\mathbf R_k$; for the oscillatory components
put all derivatives on the low-frequency kernels.
The difference of the quadratic sources vanishes on $[0,T]$
and is bounded on $[T,t]$ by
$C\epsilon\delta^{-1/1000}\mathcal D\langle x\rangle^{-4}$.
Its contribution is therefore bounded by
$C_T\epsilon\mathcal D\delta^{1-1/1000}
\langle x\rangle^{-3}$.
Together with \eqref{s7c31} and \eqref{s7c33}, this gives
\[
\|\mathcal B(\tilde E^1)-\mathcal B(\tilde E^2)\|_{G^{1,\delta}}
\lesssim_T\mathcal D\delta^{1/8}
\lesssim\mathcal D\delta^{1/20},
\]
which proves \eqref{a24}.
\medskip

\textbf{Step 5.} We prove \eqref{a24'}.
Fix $j\in\{1,2\}$ and $t\in[T,T+\delta]$.
By \eqref{s7c32},
\begin{equation}\label{s7c35}
\sum_{\substack{n\leq3\\a=0,1}}
\left|\nabla_x^n|\nabla|^a
\big[\mathcal B^\Box(\tilde E^j)(t)
-\mathcal B^\Box(\tilde E^j)(T)\big](x)\right|
\lesssim_T\epsilon\delta\langle x\rangle^{-2-2\kappa_0}.
\end{equation}

For $\mathcal B_1^{reg}$, split the time integrations at
$(T-\delta)_+$.
The contributions adjacent to $T$ are bounded by
$C_T\epsilon\delta\log(2/\delta)
\langle x\rangle^{-2-2\kappa_0}$, using \eqref{s7c12}.
On the remaining interval, differentiate the kernel and
the cosine factor with respect to the terminal time.
The bounds for $K^{1-\chi}_{1,1}$ in \eqref{s7c12}
give the same estimate.

For $\mathcal B_2^{reg}$, differentiating the low-frequency
kernel and the sine factor on $[0,T]$ gives a contribution
$C_T\epsilon^2\delta\langle x\rangle^{-3}$.
On $[T,t]$, use $|\sin(t-s)|\leq t-s$ and
\eqref{s7c14}; this contribution is bounded by
$C_T\epsilon^2\delta^{2-2/1000}\langle x\rangle^{-3}$.
It follows that
\begin{equation}\label{s7c36}
\begin{split}
&\sum_{n\leq2}
\left|\nabla_x^n
\big[\mathcal B^{reg}(\tilde E^j)(t)
-\mathcal B^{reg}(\tilde E^j)(T)\big](x)\right|\\
&\qquad\lesssim_T
\epsilon\delta\log(2/\delta)\langle x\rangle^{-2-2\kappa_0}
\lesssim_T\epsilon\delta^{1/2}
\langle x\rangle^{-2-2\kappa_0}.
\end{split}
\end{equation}

It remains to estimate the time derivatives of the oscillatory
components. In view of \eqref{s7c31} and \eqref{s7c35},
it suffices to prove
\begin{equation}\label{s7c37}
|\mathbf R_k(E^j)(t,x)-\mathbf R_k(E)(T,x)|
\lesssim_T\epsilon\delta^{19\kappa_0}
\langle x\rangle^{-2-2\kappa_0},
\qquad k=0,1,2.
\end{equation}

For the initial-data component, use \eqref{s7c24} to write
\[
\mathcal I_k(E^j)(t,x)-\mathcal I_k(E)(T,x)
=\int_{\R^2}v^{\otimes k}
\big[g(X^j_{T,t}(x,v),V^j_{T,t}(x,v))-g(x,v)\big]\,dv.
\]
If $|v|\leq\delta^{-1}$, the segment joining the two arguments
of $g$ preserves the weight in \eqref{s7c25}, up to a
constant depending on $T$. Its length is bounded by
$C_T(\delta\langle v\rangle+\epsilon\delta^{1/2})$,
by \eqref{s7c18}.
For $|v|>\delta^{-1}$, estimate the two terms separately.
The calculation in \eqref{s7c26} therefore gives
\[
\begin{split}
|\mathcal I_k(E^j)(t,x)-\mathcal I_k(E)(T,x)|
&\lesssim_T\epsilon\langle x\rangle^{-2-2\kappa_0}
\int_{\R^2}
\frac{\min\{\delta\langle v\rangle,1\}
+\epsilon\delta^{1/2}}
{\langle v\rangle^{2+19\kappa_0}}\,dv\\
&\lesssim_T\epsilon\delta^{19\kappa_0}
\langle x\rangle^{-2-2\kappa_0}.
\end{split}
\]

For the nonlinear component,
\[
\begin{split}
\mathcal R_k(E^j)(t,x)-\mathcal R_k(E)(T,x)
&=\mathcal R_k^>(E^j)(t,x)+\int_0^T\int_{\R^2}
\big[\mathcal C^j_{k,s,t}(x,v)
-\mathcal C^j_{k,s,T}(x,v)\big]\,dv\,ds.
\end{split}
\]
The first term is bounded by \eqref{sec7-R-pieces}.
Split the remaining velocity integral at $|v|=\delta^{-1/2}$.
On the large-velocity region, \eqref{s7c27},
\eqref{s7c16}, and \eqref{s7c8}, together with
Cauchy--Schwarz for the mixed ray integrals, give
\[
\begin{split}
&\left|\int_0^T\int_{|v|\geq\delta^{-1/2}}
\big[\mathcal C^j_{k,s,t}(x,v)
-\mathcal C^j_{k,s,T}(x,v)\big]\,dv\,ds\right|\\
&\qquad\lesssim_T
\epsilon^2\delta^{1/16-1/1000}\langle x\rangle^{-21/8}
\lesssim_T\epsilon^2\delta^{1/17}\langle x\rangle^{-21/8}.
\end{split}
\]

Assume now $|v|\leq\delta^{-1/2}$.
The composition formula \eqref{s7c22}, the short-interval
bounds \eqref{s7c18}, and the old-flow derivative bounds
used in Step 1 imply
\[
\begin{split}
&|\widetilde Y^j_{s,t}(x,v)-\widetilde Y_{s,T}(x,v)|
+|\widetilde W^j_{s,t}(x,v)-\widetilde W_{s,T}(x,v)|\lesssim_T\epsilon
\big(\delta^{-1/1000}\mathfrak I_\delta(x,v)
+\delta^{1/2}\mathfrak I_t(x,v)\big).
\end{split}
\]
Indeed, the change in the terminal position is bounded by
$\delta|v|+|\widetilde Y^j_{T,t}(x,v)|$, and the change in
the terminal velocity by $|\widetilde W^j_{T,t}(x,v)|$.
Apply the weighted spatial derivative bounds to the former
and the bounded velocity derivatives to the latter.

Since $|(t-T)v|\leq\delta^{1/2}$, the free rays at times
$t$ and $T$ have comparable weights. Applying the mean
value theorem twice to the definition of $\mathcal C^j$,
and using \eqref{s7c13}, \eqref{s7c16}, and the last display,
we obtain
\[
\begin{split}
|\mathcal C^j_{k,s,t}(x,v)-\mathcal C^j_{k,s,T}(x,v)|
&\lesssim_T\epsilon^2
\langle v\rangle^{-2}\langle x-(T-s)v\rangle^{-2}\\
&\quad\times
\big(\delta^{-1/1000}\mathfrak I_\delta(x,v)
+\delta^{1/2}\mathfrak I_t(x,v)\big).
\end{split}
\]
Integration and \eqref{a40} and \eqref{s7c7} yield
\[
\begin{split}
&\left|\int_0^T\int_{|v|\leq\delta^{-1/2}}
\big[\mathcal C^j_{k,s,t}(x,v)
-\mathcal C^j_{k,s,T}(x,v)\big]\,dv\,ds\right|\\
&\qquad\lesssim_T\epsilon^2
\big(\delta^{1/8-1/1000}+\delta^{1/2}\big)
\langle x\rangle^{-11/4}
\lesssim_T\epsilon^2\delta^{1/9}\langle x\rangle^{-11/4}.
\end{split}
\]
Together with the initial-data bounds and \eqref{sec7-R-pieces},
this proves \eqref{s7c37}.

Finally, all derivatives of the low-frequency operator
defining $\mathcal S^j$ have kernels satisfying
Lemma~\ref{convLe} with $b=2+2\kappa_0$.
Thus \eqref{s7c37} gives
\[
\sum_{n\leq2}
|\nabla_x^n[\mathcal S^j(t)-\mathcal S^j(T)](x)|
\lesssim_T\epsilon\delta^{19\kappa_0}
\langle x\rangle^{-2-2\kappa_0}.
\]
Using \eqref{s7c31}, \eqref{s7c32}, and \eqref{s7c35},
we conclude that
\[
\sum_{n\leq2}
\left|\nabla_x^n
\big[\partial_t\mathcal B^\Box(\tilde E^j)(t)
-\partial_t\mathcal B^\Box(\tilde E^j)(T)\big](x)\right|
\lesssim_T\epsilon\delta^{19\kappa_0}
\langle x\rangle^{-2-2\kappa_0}.
\]
Combining this with \eqref{s7c35} and \eqref{s7c36} proves
\[
[\mathcal B(\tilde E^1)]_{2,\delta;T}
+[\mathcal B(\tilde E^2)]_{2,\delta;T}
\lesssim\epsilon\delta^{\kappa_0},
\]
which is \eqref{a24'} and completes the proof of
Lemma~\ref{keyLe}.

\section{Proof of Theorem \ref{mainthm}}

Let $C\geq1$ be the constant in Proposition~\ref{propoglobal3d}
and let $C_1\geq1$ be the constant in \eqref{bo2}.
Fix the bootstrap threshold $\epsilon_0$ sufficiently small for
Propositions~\ref{propoglobal3d} and \ref{Localwell}, with
$(1+2C)\epsilon_0\leq1$. Choose the data threshold in
Theorem~\ref{mainthm} so that
$0<\bar\epsilon\leq
\frac{\epsilon_0}{2(C_1+1)(2C+1)}.$
\medskip

\textbf{Step 1: uniqueness.}
Fix $R>0$ and let $(f^j,E^j)$, $j=1,2$, be two classical solutions
on $[0,R]$ with the same initial data and satisfying the weighted
field condition in the theorem. Then
$E^j\in L^\infty([0,R];L_x^2\cap W_x^{1,\infty})$.
Their characteristic flows have bounded velocity shifts and first
derivatives on $[0,R]$. Using \eqref{a16'}, \eqref{t1}, volume
preservation, and
$|\nabla_v\mu|+|\nabla_v^2\mu|\lesssim\langle v\rangle^{-4}$, we obtain
\[
\sup_{0\leq t\leq R}
\|\langle v\rangle^{5/2}f^j(t)\|_{L^2_{x,v}}<\infty,
\qquad
|\nabla_v(\mu+f^j)(t,x,v)|\leq C_R\langle v\rangle^{-4}.
\]
For the first bound, use $E^j\in L^\infty_tL_x^2$ and
$\langle v\rangle^{5/2}\nabla_v\mu\in L_v^2$ in \eqref{t1};
the second follows by differentiating that formula. Set $h=f^1-f^2$, $e=E^1-E^2$, and
$Q(t):=\|\langle v\rangle^{5/2}h(t)\|_{L^2_{x,v}}^2$.
Subtracting the equations gives
\[
\partial_t h+v\cdot\nabla_xh+E^1\cdot\nabla_vh
=-e\cdot\nabla_v(\mu+f^2).
\]
The difference of the continuity equations and the Poisson relation give
\[
\partial_te=-\nabla_x\Delta_x^{-1}\operatorname{div}_xJ,
\qquad
J(t,x):=\int_{\mathbb R^2}v\,h(t,x,v)\,dv.
\]
By Cauchy--Schwarz,
\[
\|J(t)\|_{L_x^2}^2
\leq\left(\int_{\mathbb R^2}|v|^2\langle v\rangle^{-5}\,dv\right)Q(t)
\lesssim Q(t).
\]
Multiplying the transport equation by $\langle v\rangle^5h$ and
integrating, justified by spatial and velocity cutoffs, gives
$Q'(t)\leq C_R(Q(t)+\|e(t)\|_{L_x^2}^2)$.
Here we use $|\nabla_v\langle v\rangle^5|\lesssim\langle v\rangle^5$
and $\int_{\mathbb R^2}\langle v\rangle^{-3}\,dv<\infty$.
Since $\nabla_x\Delta_x^{-1}\operatorname{div}_x$ is bounded on $L_x^2$,
the field identity yields
\[
\frac{d}{dt}\bigl(Q(t)+\|e(t)\|_{L_x^2}^2\bigr)
\leq C_R\bigl(Q(t)+\|e(t)\|_{L_x^2}^2\bigr).
\]
The initial difference is zero, so Gronwall's inequality gives
$h=e=0$ on $[0,R]$. Since $R$ is arbitrary, uniqueness follows.
\medskip

\textbf{Step 2: global existence.}
Set $\varepsilon_{\mathrm{in}}:=[f_0]\leq\bar\epsilon$. If $\varepsilon_{\mathrm{in}}=0$, the zero solution has all
the asserted properties and is unique by Step 1; hence assume $\varepsilon_{\mathrm{in}}>0$.
With $M_T$ defined in \eqref{bo1}, let
\[
T^*:=\sup\big\{T>0:\text{\eqref{eq2} admits a unique classical
solution on $[0,T]$ with }M_T\leq\epsilon_0\big\}.
\]
Proposition~\ref{Localwell}, applied with $T=0$ and $\epsilon=\varepsilon_{\mathrm{in}}$,
gives an initial solution with
$M_{\delta(0)}\leq(C_1+1)\varepsilon_{\mathrm{in}}\leq\epsilon_0/2$; thus $T^*>0$.
For every $T<T^*$, Proposition~\ref{propoglobal3d} gives
\[
A_T\leq2C\varepsilon_{\mathrm{in}},
\qquad
M_T=A_T+\varepsilon_{\mathrm{in}}\leq(2C+1)\varepsilon_{\mathrm{in}}\leq\epsilon_0/2,
\]
where $A_T$ is defined in \eqref{local-AT}.

Suppose $T^*<\infty$. The proof of Proposition~\ref{Localwell}
gives a common continuation time $\delta_*>0$ for $T\in[0,T^*]$.
For $T<T^*$ sufficiently close to $T^*$, apply that proposition
with $\epsilon=(2C+1)\varepsilon_{\mathrm{in}}\leq\epsilon_0$. The extension satisfies
\[
\begin{aligned}
M_{T+\delta_*}
=A_{T+\delta_*}+\varepsilon_{\mathrm{in}}\leq C_1(2C+1)\varepsilon_{\mathrm{in}}+\varepsilon_{\mathrm{in}}\leq(C_1+1)(2C+1)\bar\epsilon
\leq\epsilon_0/2.
\end{aligned}
\]
Taking $T+\delta_*>T^*$ contradicts the definition of $T^*$.
Thus $T^*=\infty$, and the preceding energy
argument gives uniqueness in the stated class.
The bounds $A_T\leq2C[f_0]$, valid for every $T$, together
with Lemma~\ref{eslinear}, give all the electric-field
estimates in the theorem. In particular, the constructed solution satisfies
the weighted field condition in its statement.

On each finite interval $[0,R]$, the local construction gives $E$ and
$\nabla_xE$ continuous in $(t,x)$. The characteristic flow and its first
$(x,v)$-derivatives therefore depend continuously on time, uniformly on
compact phase-space sets. Formula~\eqref{t1} and its first $(x,v)$-derivatives
then give $f\in C([0,R];C^1_{x,v})$. Moreover, \eqref{eq2} yields
\[
\partial_tf=-v\cdot\nabla_xf-E\cdot\nabla_vf-E\cdot\nabla_v\mu
\in C([0,R];C^0_{x,v}).
\]
Thus $f\in C^1([0,R];C^0_{x,v})$ in the topology specified before the
statement of the theorem. Since $R$ is arbitrary, this proves the asserted
time regularity.
\medskip

\textbf{Step 3: convergence of the characteristics.}
By \eqref{n14} and \eqref{xf},
\[
\sup_{t\geq s}
\big(\operatorname{Lip}_{x,v}Y_{s,t}
+\langle s\rangle\operatorname{Lip}_{x,v}W_{s,t}\big)
\lesssim[f_0].
\]
Apply the mean value theorem to the composition formulas
\eqref{a31}--\eqref{a32}. Using also \eqref{n13}, we obtain,
for $0\leq s\leq t'\leq t$,
\[
\begin{split}
&|Y_{s,t}(x,v)-Y_{s,t'}(x,v)|
+\langle s\rangle|W_{s,t}(x,v)-W_{s,t'}(x,v)|\\
&\qquad\lesssim
|Y_{t',t}(x,v)|+\langle t'\rangle|W_{t',t}(x,v)|
\lesssim[f_0]\mathfrak a_0(t',v).
\end{split}
\]
Since $\mathfrak a_0(t',v)\lesssim
\langle t'\rangle^{-1/2-\kappa_0}$, the limits
\[
(Y_{s,\infty},W_{s,\infty})(x,v)
:=\lim_{t\to\infty}(Y_{s,t},W_{s,t})(x,v)
\]
exist uniformly in $(x,v)$, and
\begin{equation}\label{a51}
\begin{split}
&|Y_{s,t}(x,v)-Y_{s,\infty}(x,v)|
+\langle s\rangle|W_{s,t}(x,v)-W_{s,\infty}(x,v)|\lesssim[f_0]\mathfrak a_0(t,v)
\lesssim[f_0]\langle t\rangle^{-1/2-\kappa_0}.
\end{split}
\end{equation}
The limits inherit the Lipschitz bounds above and the
pointwise bounds \eqref{n13}.

The convergence is uniform when $s$ ranges over a bounded
interval. Passing to the limit in the characteristic
equations in integral form gives
\[
\begin{split}
\partial_sY_{s,\infty}(x,v)&=W_{s,\infty}(x,v),\\
\partial_sW_{s,\infty}(x,v)
&=E(s,x+sv+Y_{s,\infty}(x,v)).
\end{split}
\]
Thus these functions are continuously differentiable
in $s$, and $W_{s,\infty}(x,v)\to0$ as $s\to\infty$.
\medskip

\textbf{Step 4: scattering.}
Conservation of $f+\mu$ along characteristics gives
\[
\begin{split}
f(t,z+tv,v)
={}&f_0(z+Y_{0,t}(z,v),v+W_{0,t}(z,v))+\mu(v+W_{0,t}(z,v))-\mu(v).
\end{split}
\]
We therefore define
\[
\begin{split}
f_\infty(z,v)
:={}&f_0(z+Y_{0,\infty}(z,v),v+W_{0,\infty}(z,v))+\mu(v+W_{0,\infty}(z,v))-\mu(v).
\end{split}
\]
By Step 3 and \eqref{a16'}, $f_\infty\in C^{0,1}_{z,v}$.

Taking $s=0$ in \eqref{a51} and applying the mean value theorem
directly to the formulas for $f(t,z+tv,v)$ and $f_\infty(z,v)$ gives
\[
\begin{split}
\langle v\rangle^4|f(t,z+tv,v)-f_\infty(z,v)|
&\lesssim\bigl([f_0]^2\langle v\rangle^{-2-21\kappa_0}
+[f_0]\bigr)\langle t\rangle^{-\frac12-\kappa_0}\lesssim[f_0]\langle t\rangle^{-\frac12-\kappa_0},
\end{split}
\]
where we used $[f_0]\leq1$. This proves the simpler uniform
scattering estimate stated after \eqref{po}.

Fix $(t,x,v)$ and set $z:=x-tv$.
The corrections at $s=0$ are $O([f_0])$, so their shifts
preserve the weights in \eqref{a16'}.
Using the mean value theorem, \eqref{a51}, and
$|\nabla\mu(v)|\lesssim\langle v\rangle^{-4}$, we obtain
\begin{equation}\label{sec8-profile-comparison}
\begin{split}
\langle v\rangle^3|f(t,x,v)-f_\infty(z,v)|
&\lesssim[f_0]^2\langle z,v\rangle^{-3-20\kappa_0}
\langle t\rangle^{-1/2-\kappa_0}\\
&+\langle v\rangle^{-1}
|W_{0,t}(z,v)-W_{0,\infty}(z,v)|.
\end{split}
\end{equation}

Define $Z_s(z,v):=z+sv+Y_{s,\infty}(z,v)$.  The equations for $W_{s,t}$ and $W_{s,\infty}$ imply
\[
\begin{split}
W_{0,t}(z,v)-W_{0,\infty}(z,v)
={}&-W_{t,\infty}(z,v)+\int_0^t\big[
E(s,Z_s(z,v))
-E(s,z+sv+Y_{s,t}(z,v))\big]\,ds.
\end{split}
\]
The two trajectories differ by
$O([f_0]\langle t\rangle^{-1/2-\kappa_0})$,
by \eqref{a51}, and their spatial weights are comparable.
The field estimates therefore give
\[
\begin{split}
|E(s,Z_s(z,v))
-E(s,z+sv+Y_{s,t}(z,v))|\lesssim[f_0]^2
\langle t\rangle^{-1/2-\kappa_0}
\langle s\rangle^{-2/3+\kappa_0}
\langle s,z+sv\rangle^{-2}.
\end{split}
\]
Consequently,
\begin{equation}\label{sec8-velocity-comparison}
\begin{split}
\langle v\rangle^{-1}
|W_{0,t}(z,v)-W_{0,\infty}(z,v)|&\lesssim
\langle v\rangle^{-1}|W_{t,\infty}(z,v)|\\
&+[f_0]^2\langle t\rangle^{-1/2-\kappa_0}
\int_0^t\langle s\rangle^{-2/3+\kappa_0}
\langle s,z+sv\rangle^{-2}\,ds.
\end{split}
\end{equation}

It remains to estimate $W_{t,\infty}$.
Step 3 and the field bounds give
\[
W_{t,\infty}(z,v)
=-\int_t^\infty E(s,Z_s(z,v))\,ds,
\qquad
\partial_sZ_s(z,v)=v+W_{s,\infty}(z,v).
\]
In particular,
$|\partial_sZ_s(z,v)|\lesssim\langle v\rangle$ and
$\langle s,Z_s(z,v)\rangle\approx\langle s,z+sv\rangle$.
Decompose $E$ as in \eqref{de1} and integrate by parts
in the factors $\cos s$ and $\sin s$. This gives
\[
\begin{split}
|W_{t,\infty}(z,v)|
&\lesssim
\sum_{\ast\in\{\cos,\sin\}}
|E^\ast(t,Z_t(z,v))|
+\int_t^\infty|E^{reg}(s,Z_s(z,v))|\,ds\\
&\quad+
\sum_{\ast\in\{\cos,\sin\}}\int_t^\infty
\big(
|(\partial_sE^\ast)(s,Z_s(z,v))|
+\langle v\rangle|\nabla_xE^\ast(s,Z_s(z,v))|
\big)\,ds.
\end{split}
\]
The boundary terms at infinity vanish by \eqref{Ebounds1}.
Using \eqref{cbc2}, \eqref{Ebounds1}, and
$Z_t(z,v)=x+Y_{t,\infty}(z,v)$, we conclude that
\begin{equation}\label{sec8-velocity-tail}
\begin{split}
\langle v\rangle^{-1}|W_{t,\infty}(z,v)|
&\lesssim[f_0]\langle t\rangle^{50\kappa_0}
\langle t,x\rangle^{-2}+[f_0]\int_t^\infty
\langle s\rangle^{-1/2}
\langle s,z+sv\rangle^{-2}\,ds.
\end{split}
\end{equation}

Combining \eqref{sec8-profile-comparison},
\eqref{sec8-velocity-comparison}, and
\eqref{sec8-velocity-tail}, and using
$[f_0]\leq1$ and
$\langle z,v\rangle^{-3-20\kappa_0}
\leq\langle z,v\rangle^{-2-19\kappa_0}$,
proves \eqref{po}. This completes the proof of
Theorem~\ref{mainthm}.

\end{document}